\documentclass[11pt]{amsart}
\usepackage{amssymb, microtype, tikz}
\usepackage[colorlinks, linkcolor = blue, citecolor = blue]{hyperref}
\allowdisplaybreaks

\let\SS\relax

\newtheorem{thm}{Theorem}[section]
\newtheorem{lem}[thm]{Lemma}
\newtheorem{prp}[thm]{Proposition}
\newtheorem{cor}[thm]{Corollary}
\theoremstyle{definition}
\newtheorem{rmk}[thm]{Remark}

\newtheorem{cnj}[thm]{Conjecture}

\newcommand{\CC}{\mathbb{C}}
\newcommand{\DD}{\mathbb{D}}

\newcommand{\PP}{\mathbb{P}}
\newcommand{\QQ}{\mathbb{Q}}
\newcommand{\RR}{\mathbb{R}}
\newcommand{\SS}{\mathbb{S}}
\newcommand{\TT}{\mathbb{T}}

\newcommand{\ZZ}{\mathbb{Z}}

\newcommand{\CP}{\mathbb{CP}} 

\newcommand{\es}{\varnothing} 
\newcommand{\sd}{\smallsetminus} 
\newcommand{\sub}{\subset} 
\newcommand{\dr}{\mathrm{dR}} 
\newcommand{\FS}{\mspace{1mu}\mathrm{FS}} 
\newcommand{\PD}{\mathrm{PD}} 

\DeclareMathOperator{\SP}{Sp} 

\DeclareMathOperator{\ind}{ind} 
\DeclareMathOperator{\Int}{Int} 

\theoremstyle{plain}
\newtheorem*{t0}{Type 0}
\newtheorem*{t1}{Type 1}
\newtheorem*{t2a}{Type 2a}
\newtheorem*{t2b}{Type 2b}
\newtheorem*{t3}{Type 3}

\title{Packing Integral Tori in Del Pezzo Surfaces}
\author{Karim Boustany}
\thanks{The work presented here was partially supported by NSF grant DMS-1547292.}
\address{Karim Boustany, Department of Mathematics, University of Notre Dame}
\email{kboustan@nd.edu}

\begin{document}

\begin{abstract}
We extend a packing result of R. Hind and E. Kerman \cite{HK} for integral Lagrangian tori in $\SS^{2} \times \SS^{2}$ to the Del Pezzo surfaces $(\DD_{n}, \omega_{\mspace{1mu}\DD_{n}})$ for $n = 1, \dots, 5$. An \emph{integral} torus is one whose relative area homomorphism is integer-valued, and we seek a \emph{maximal integral packing}. By definition, this is a disjoint collection $\{L_{i}\}$ of integral Lagrangian tori with the following property: any other integral Lagrangian torus not in this collection must intersect at least one of the $L_{i}$. We show that one can always find such a packing consisting of only the Clifford torus. The structure of our argument closely follows the one in \cite{HK}.
\end{abstract}

\maketitle

\section{Introduction}

The study of Lagrangian submanifolds and their properties is ubiquitous in symplectic topology. This is partly due to the fact that many questions in the field can be formulated as questions involving Lagrangian submanifolds. For instance, if $(M, \omega)$ is a symplectic manifold and $\varphi \colon (M, \omega) \to (M, \omega)$ is a symplectomorphism, then its graph $\Gamma_{\varphi}$ is a Lagrangian submanifold of $(M \times M, (-\omega) \oplus \omega)$. Consequently, studying the fixed points of $\varphi$ is equivalent to studying the set $\Gamma_{\varphi} \cap \Delta$ of intersection points in $M \times M$.

Turning therefore to the question of Lagrangian intersections, we must mention the progress made by Lagrangian Floer theory, first developed by Floer under some relative asphericity conditions and then generalized by Oh to the case of monotone Lagrangians in closed symplectic manifolds. This latter work was done in a series of papers starting with \cite{Oh}. The theory has been extended considerably since then, but we will not delve into the history of these developments. We do mention however that one of the main motivations for the development of Lagrangian Floer theory was the \emph{Arnold-Givental conjecture}, a version of which can be stated as follows. If $(M, \omega)$ is a closed symplectic manifold and $L \sub M$ is the fixed point set of an anti-symplectic involution on $M$ (so in particular, a Lagrangian submanifold), then
\[
\#(L \cap \varphi(L)) \geq \dim{H_{\ast}(L)},
\]
where $\varphi \colon M \to M$ is any Hamiltonian diffeomorphism and $L$ and $\varphi(L)$ intersect transversally. By the work of Oh mentioned above, this conjecture is true when $L$ is a monotone Lagrangian torus. (The definition of monotonicity will be recalled in Section 1.) Observe that this result gives no information about intersections of arbitrary Lagrangians with $L$, but only Lagrangians that are images of $L$ under Hamiltonian diffeomorphisms. It is therefore a statement about the \emph{nondisplaceability of $L$}.

One can ask more general questions about Lagrangian intersections. For instance, let $(M, \omega)$ be any symplectic manifold and $\mathcal{L}$ be a collection of Lagrangians in $M$ satisfying some set of geometric or topological restrictions. Then the following two questions naturally present themselves:
\begin{enumerate}
\item Do any two Lagrangians in $\mathcal{L}$ intersect?
\item Is there some fixed Lagrangian in $\mathcal{L}$ that intersects every other Lagrangian in $\mathcal{L}$?
\end{enumerate}

Evidently, question (2) is more accessible than question (1). The following result of Gromov, see $2.3.B_{4}^{\prime\prime}$ in \cite{Gro}, is along these lines and will have important ramifications in the present paper.

\begin{thm}[Gromov]\label{t:gromov}
Let $\Lambda$ be a closed manifold and $L$ be a closed and exact Lagrangian submanifold of $(T^{\ast}\Lambda, d\lambda)$, where $\lambda$ is the canonical Liouville $1$-form on $T^{\ast}\Lambda$. Then $L$ must intersect the zero section $\Lambda$.
\end{thm}

For $\Lambda = \TT^{2}$, this theorem, when combined with the work of Dimitroglou Rizell, Goodman, and Ivrii on the Nearby Lagrangian Conjecture for $T^{\ast}\TT^{2}$, see Theorem B of \cite{DRGI}, implies the following result in the spirit of question (1).

\begin{prp}\label{p:cotangentexact}
Any two exact Lagrangian tori in $T^{\ast}\TT^{2}$ must intersect.
\end{prp}

\begin{proof}
Let $L$ be an exact Lagrangian torus in $T^{\ast}\TT^{2}$. We use Theorem B of \cite{DRGI} to obtain a Hamiltonian diffeomorphism $\varphi \colon T^{\ast}\TT^{2} \to T^{\ast}\TT^{2}$ taking $L$ to the zero section. If $L^{\prime}$ is another Lagrangian torus in $T^{\ast}\TT^{2}$ that is exact and disjoint from $L$, then $\smash{\varphi(L^{\prime})}$ would be an exact Lagrangian torus which is disjoint from the zero section, contradicting Theorem \ref{t:gromov} above.
\end{proof}

\begin{rmk}
The previous proposition is true whenever $\Lambda$ is a smooth manifold for which the Nearby Lagrangian Conjecture is verified, see \cite{DRGI}.
\end{rmk}

The objective of this paper is to study question (2) for Lagrangian tori in a special class of symplectic $4$-manifolds, the so-called \emph{symplectic Del Pezzo surfaces}. These are the monotone symplectic manifold $(\SS^{2} \times \SS^{2}, \sigma \oplus \sigma)$, where $\sigma(\SS^{2}) = 2$, the monotone symplectic manifold $(\CP^{2}, \omega_{\FS})$, where $\omega_{\FS}$ is the Fubini-Study form with $\omega_{\FS}(\CP^{1}) = 3$, and eight other monotone symplectic manifolds obtained by blowing up along Darboux balls of capacity 1 in $(\CP^{2}, \omega_{\FS})$. We now describe recent progress on answering question (2) for this class of manifolds, together with a description of our main results.

\subsection{Main Results}

Let $(M, \omega)$ be an integral symplectic $4$-manifold. (A symplectic manifold is \emph{integral} if its symplectic form takes on integer values when integrated over any closed $2$-dimensional submanifold.) A Lagrangian torus $L \sub M$ is said to be \emph{integral} if the relative area homomorphism $\omega \colon \pi_{2}(M, L) \to \RR$ is integer-valued. A disjoint collection $\{L_{i}\}$ of such tori is said to be a \emph{maximal integral packing of $(M, \omega)$} if any integral torus in $M$ not belonging to this collection intersects at least one member of the collection. The integrality condition is required only so that the above intersection condition is not trivially false, since one can always embed a Lagrangian torus in any sufficiently small Darboux chart in $(M, \omega)$ missing the collection $\{L_{i}\}$. The following result was recently obtained by Richard Hind and Ely Kerman in \cite{HK}.

\begin{thm}[Hind, Kerman]
The Clifford torus is a maximal integral packing of $(\SS^{2} \times \SS^{2}, \sigma \oplus \sigma)$.
\end{thm}

This answers question (2) affirmatively for integral Lagrangian tori in $(\SS^{2} \times \SS^{2}, \sigma \oplus \sigma)$. The present paper takes inspiration from the above result and extends it to the symplectic Del Pezzo surfaces $(\DD_{n}, \omega_{\mspace{1mu}\DD_{n}})$ for $n = 1, \dots, 5$. These are obtained by symplectically blowing up along $n$ pairwise disjoint Darboux balls of capacity $1$ in $(\CP^{2}, \omega_{\FS})$. The resulting symplectic manifolds are independent of the symplectic isotopy classes of the above Darboux balls. In particular, if $(\DD_{n}, \omega_{\mspace{1mu}\DD_{n}})$ is obtained by symplectically blowing up along $n$ Darboux balls in $\CP^{2}$ that lie in the complement of the Clifford torus $L_{1, 1}$, then this lifts to a Lagrangian torus in $\DD_{n}$ which we also call the Clifford torus and denote by $L_{1, 1}$. This is always possible for $n \leq 3$, and more recently was shown to be true for $n = 4$ and $n = 5$, see Theorem 6.1 of \cite{LOV}. Our main result is then the following.

\begin{thm}\label{t:main}
The Clifford torus is a maximal integral packing of $(\DD_{n}, \omega_{\mspace{1mu}\DD_{n}})$ for $n = 1, \dots, 5$.
\end{thm}

In other words, every integral Lagrangian torus in $(\DD_{n}, \omega_{\mspace{1mu}\DD_{n}})$ for $n = 1, \dots, 5$ must intersect $L_{1, 1}$. In view of the description of blowing up and down given in Section 2 below, the case $n = 1$ turns out to imply the following slightly weaker result about maximal integral packings in certain open submanifolds of $\CP^{2}$.

\begin{cor}\label{c:main}
If $B$ is any closed Darboux ball of capacity $1$ in $\CP^{2} \sd L_{1, 1}$, then the Clifford torus is a maximal integral packing of $(\CP^{2} \sd B, \omega_{\FS})$.
\end{cor}

\subsection{Commentary}

The conclusion of Theorem \ref{t:main} is not known to hold for $n > 5$ since, to our knowledge, there is no known embedding of $n > 5$ disjoint closed Darboux balls of capacity $1$ in $\CP^{2} \sd L_{1, 1}$, although we emphasize that this is the \emph{only} obstruction to the above theorem holding in those cases, since our methods reduce all considerations to the case $n =1$.

Actually, the results in this paper imply a more general statement than that of Theorem \ref{t:main}. Namely, our work actually shows that the Clifford torus is a maximal packing of \emph{any} symplectic manifold obtained from $\CP^{2}$ by symplectically blowing up along finitely many disjoint Darboux balls of integral capacity in the complement of $L_{1, 1}$, so long as \emph{at least one of them} is of capacity $1$. This a priori sounds like a much larger class of manifolds then the Del Pezzo surfaces, but relies on the possibility of symplectically embedding closed balls of integral capacity greater than $1$ in $\CP^{2} \sd L_{1, 1}$. In fact, in joint work in progress with Richard Hind, it is shown that such a ball must have capacity strictly less than $2$, meaning our current result really does exhaust all possibilities.

There are two limitations to the techniques used to prove Theorem \ref{t:main}:
\begin{enumerate}
\item They require the ambient symplectic manifold to be a ruled surface over $\SS^{2}$, since they make heavy use of the existence of foliations by closed pseudoholomorphic spheres, as well as some of the special features of pseudoholomorphic curves in dimension $4$.
\item They require the tori in a maximal integral packing to be \emph{monotone}, since this restricts the kinds of degenerations of pseudoholomorphic curves that can occur when stretching the neck, and allows one to compute the indices of pseudoholomorphic planes given information about the relative area classes of their compactifications.
\end{enumerate}

There are two interesting situations where one of the two requirements above are not met and for which the integral packing problem appears to be accessible. The first is $(\CP^{2}, \omega_{\FS})$, where the first requirement fails and for which the following conjecture seems reasonable.

\begin{cnj}\label{cn:packing1}
$L_{1, 1}$ is a maximal integral packing of $(\CP^{2}, \omega_{\FS})$.
\end{cnj}

In the cases where this problem has been solved thus far and where the symplectic manifold admits a moment image, namely $(\SS^{2} \times \SS^{2}, \sigma \oplus \sigma)$ and $(\DD_{n}, \omega_{\mspace{1mu}\DD_{n}})$ for $n = 1, 2, 3$, there is a maximal packing consisting of precisely the integer points in the moment image, which in all cases is the Clifford torus projecting to the point $(1, 1)$. It would be appealing to think that this holds for all (precompact) toric symplectic manifolds, but curiously a second result from \cite{HK}, see Theorem 1.2 in that reference, shows it to be false for the polydisk $P(a, b)$ when $\min(a, b) > 2$.

Another situation in which the packing problem appears accessible thr\-ough the methods of this paper are the toric manifolds $(Y_{m}, \omega_{m})$ for $m \geq 2$ obtained from $\CP^{2}$ with lines of area $m + 2$ by symplectically blowing up along the standard Darboux ball of capacity $m$ centered at the point $[1 : 0 : 0]$. In this case, the monotonicity of the ambient symplectic manifold is lost but the moment image contains $m$ integer points on the line $x + y = m + 1$, whose preimages are the product tori $L_{k, l}$ with $k + l = m + 1$. We can then make the following conjecture.

\begin{cnj}\label{cn:packing2}
The collection $\{L_{k, l} : k + l = m + 1\}$ is a maximal integral packing of $(Y_{m}, \omega_{m})$ for each $m \geq 2$.
\end{cnj}

If true, this would produce the first examples of maximal integral packings consisting of more than one torus, and would also produce examples of maximal integral packings consisting of $m$ tori for each $m \geq 2$. We hope to address the packing problem in this and in the aforementioned case of $(\CP^{2}, \omega_{\FS})$ in future work.

\subsection{Organization of the Paper}

We begin in Section 2 by reviewing some standard background material from symplectic topology. We briefly discuss symplectic area and monotonicity, recall some basic facts from the theory of closed pseudoholomorphic curves in dimension $4$, describe the symplectic blowup operation in the form in which it will be used in this paper, and conclude by describing a standard inflation procedure that can be performed on pseudoholomorphic spheres inside symplectic $4$-manifolds. Although much of the material is standard, this section also serves the purpose of fixing our notations and conventions for the rest of the paper. The reader is therefore encouraged to at least skim through it.

In Section 3, we give a rather brief exposition of the theory of punctured pseudoholomorphic spheres, closely following the one given in \cite{DRGI}. This begins with a discussion of symplectic manifolds with cylindrical ends, and proceeds in increasing generality to culminate in the definition of a pseudoholomorphic building tailored to our specific context. This context is the one that arises when producing buildings as limits of closed spheres by applying the splitting procedure and compactness results from symplectic field theory, see \cite{BEH} and \cite{CM1}. We conclude by recording some index formulae to be used in later sections.

In Section 4, we specialize to the situation of interest to us, and define the symplectic Del Pezzo surfaces. In particular, we give a detailed description of a specific toric model of the Del Pezzo surface $(\DD_{1}, \omega_{\mspace{1mu}\DD_{1}})$, which we denote by $(X, \omega_{X})$. It will be our main object of interest because, as we show in Proposition \ref{p:simpli}, the proof of Theorem \ref{t:main} reduces to studying Lagrangian tori in $(X, \omega_{X})$. This is followed by an analysis of closed pseudoholomorphic spheres in $(X, \omega_{X})$. These spheres will be used to construct sequences whose limits after the splitting procedure will be pseudoholomorphic buildings. Next, we state and prove a displacement result for Lagrangian tori in $(X, \omega_{X})$, which will greatly simplify subsequent considerations. This being done, we introduce one of the central techniques used in this paper, that of relative pseudoholomorphic foliations. These will be foliations of the complement of a collection of monotone Lagrangian tori in $(X, \omega_{X})$ by pseudoholomorphic leaves obtained as limits of closed spheres via the splitting procedure. We carefully describe the properties of such foliations. We then discuss punctured pseudoholomorphic curves that are transverse to such foliations, and which we call \emph{essential}. Finally, we study homologicaly essential tori in certain subsets of $(X, \omega_{X})$ and show how the above relative foliations can shed light on the packing problem considered there. 

Section 5 can be thought of as the actual starting point for the proof of Theorem \ref{t:main}. We assume it to be false and proceed by contradiction. The first step is Proposition \ref{p:simpli}, which says that under such an assumption, we can find disjoint monotone tori $L_{\mathfrak{r}}$ and $L_{\mathfrak{s}}$ in $(X, \omega_{X})$ satisfying some special properties. Once this is done, we construct a relative foliation of $X \sd (L_{\mathfrak{r}} \cup L_{\mathfrak{s}})$ and begin the analysis of punctured spheres in the complement. Using these considerations, we are able to obtain a full classification of building types arising from the splitting procedure. We then show that under certain conditions, there exists buildings of one of the above types with some quantifiable properties. The section concludes with the statement of Proposition \ref{p:black}, a large black box result that encapsulates all of the facts needed in order to complete the proof of Theorem \ref{t:main}.

Section 6 concerns itself with the actual proof of Theorem \ref{t:main}. Starting with the facts established at the end of the previous section, we perform a series of blowups, inflations, and blowdowns, keeping track of the intersection numbers and symplectic areas of various objects along each step. Some concluding arguments then yield the desired contradiction. This is the only place in the paper where the technique of inflation is used.

Section 7 describes some deformation procedures that can be performed on the top level curves of certain pseudoholomorphic buildings, and then analyzes various intersection patterns that arise from these deformations. The entirety of this section is devoted to proving Proposition \ref{p:black}. Finally, we construct in Appendix A a $3$-punctured symplectic sphere in $(T^{\ast}\TT^{2}, d\lambda)$, which is needed in the proof of Proposition \ref{p:simpli}.

\subsection{Acknowledgments}

I would like to thank Richard Hind for his many invaluable suggestions and patient guidance. I would also like to thank the participants in an AMS special session on quantitative symplectic topology, which took place in April of 2023 at the University of Cincinnati. Their interest in this project led to many helpful comments and suggestions for further research.

\section{Background from Symplectic Topology}

We assume known the basic concepts of symplectic topology and closed pseudoholomorphic curves, as found for example in \cite{MS1} and \cite{MS2}. However, in order to make the presentation more readable, we will recall some elementary concepts and fix some notation.

\subsection{Notation and Conventions}

Let $M$ be a smooth $4$-manifold without boundary and $L \sub M$ be an embedded submanifold of codimension $2$. We will always consider singular homology and cohomology with integer coefficients, and so do not specify coefficients in our notation. We also identify the de Rham and real singular cohomologies of $M$. If $L$ is compact and endowed with a fixed Riemannian metric, its cotangent bundle $T^{\ast}L$ also inherits a Riemannian metric, and for $\epsilon > 0$ we define
\begin{gather*}
D_{\epsilon}^{\ast}L = \{(q, p) \in T^{\ast}L : \Vert p \Vert \leq \epsilon\}, \qquad B_{\epsilon}^{\ast}L = \{(q, p) \in T^{\ast}L : \Vert p \Vert < \epsilon\} \\
S_{\epsilon}^{\ast}L = \{(q, p) \in T^{\ast}L : \Vert p \Vert = \epsilon\}.
\end{gather*}
In particular, we can consider the \emph{unit cotangent bundle} $S^{\ast}L = S^{\ast}_{1}L$.

When $M$ is compact, it is well-known that any class in $H_{2}(M)$ is realizable by a compact oriented submanifold. In view of this, we use the same symbol to denote a submanifold and the homology class it represents. If $M$ is also oriented, then the \emph{intersection number} of homology classes $A, B \in H_{2}(M)$ is defined to be
\[
A \cdot B = \langle \PD(A) \smile \PD(B), [M] \rangle,
\]
where $\PD(A) \in H^{2}(M)$ denotes the Poincar\'{e} dual of $A$. The number $A \cdot B$ is equal to the signed count of intersection points between any two closed, oriented, embedded submanifolds of $M$ that represent the classes $A$ and $B$ and intersect transversally. This last condition can always be achieved generically using standard arguments from differential topology. If $L$ is also compact and oriented, we can sometimes consider a \emph{relative intersection number} $A \cdot B$, where $A, B \in H_{2}(M, L)$. The convention we adopt is to always interpret this as a signed count of intersection points for transversally intersecting smooth representatives, and only when these intersection points occur away from the boundaries (if any) of both representatives. In this paper, all relative intersection numbers will be well defined.

When $M$ is symplectic and $L$ is a Lagrangian torus, the following local picture will be useful. We first fix a diffeomorphism $\psi \colon \TT^{2} \to L$, where we view $\TT^{2}$ as the quotient $\RR^{2}/\ZZ^{2}$ with coordinates $(q_{1}, q_{2})$ for $0 \leq q_{1}, q_{2} \leq 1$. We call $\psi$ a \emph{parameterization of $L$}, and use it to transport the flat metric to $L$. The Weinstein neighborhood theorem, see \cite{MS1}, allows us to extend this parametrization to a symplectic embedding, also denoted by $\psi$, of an open neighborhood of the zero section in $T^{\ast}\TT^{2}$ into $M$. Its image is an open neighborhood of $L$ denoted by $\mathcal{U}(L)$ and called a \emph{Weinstein neighborhood of $L$}. It has coordinates induced from the above symplectomorphism, which we denote by $(q_{1}, q_{2}, p_{1}, p_{2})$. Note that $\psi$ induces an identification of $H_{1}(L)$ with $\ZZ^{2}$, and we denote this copy of $\ZZ^{2}$ by $H_{1}(L; \psi)$. 

\subsection{Area and Monotonicity}

If $(M, \omega)$ is a symplectic $4$-manifold, there is an obvious \emph{area homomorphism} $\omega \colon H_{2}(M) \to \RR$, defined by
\[
\omega(A) = \int_{S}\omega,
\]
where $S \sub M$ is any smooth representative of $A$. We can restrict the domain of $\omega$ to the \emph{spherical part} of $H_{2}(M)$, which is the image of $\pi_{2}(M)$ in $H_{2}(M)$ under the Hurewicz homomorphism. We thus obtain a map
\[
\omega \colon \pi_{2}(M) \to \RR, \qquad [f] \mapsto \omega([f]) = \int_{\SS^{2}}f^{\ast}\omega.
\]
Since $TM$ has a well defined first Chern class $c_{1}(M) \in H^{2}(M)$, we also have a \emph{Chern homomorphism} $c_{1} \colon H_{2}(M) \to \ZZ$ given by $c_{1}(A) = \langle c_{1}(M), A \rangle$, which can again be restricted to a homomorphism on $\pi_{2}(M)$. We then call $(M, \omega)$ \emph{monotone} if there is a real constant $\kappa > 0$, called the \emph{monotonicity constant}, such that the Chern and area homomorphisms restricted to the spherical part of $H_{2}(M)$ satisfy $c_{1} = \kappa\mspace{1mu}\omega$. This is sometimes referred to as \emph{spherical monotonicity} in the literature, but for the manifolds we consider in this paper the distinction will be irrelevant.

If $L \sub M$ is a Lagrangian torus, then we also have an obvious \emph{relative area homomorphism} $\omega \colon H_{2}(M, L) \to \RR$ and a restricted homomorphism on $\pi_{2}(M, L)$. We also have a \emph{Maslov homomorphism} $\mu \colon \pi_{2}(M, L) \to \ZZ$, whose precise definition can be found for instance in \cite{Oh}. We then say that $L$ is \emph{monotone} if there is a real constant $\lambda > 0$, also called the \emph{monotonicity constant}, such that $\mu = \lambda\mspace{1mu}\omega$. In Remark 2.3 of \cite{Oh} it is shown that the presence of such a torus forces $(M, \omega)$ to also be monotone as a symplectic manifold, with monotonicity constant $\kappa = \lambda/2$. In particular, all monotone tori in $M$ have the same monotonicity constant $\lambda = 2\kappa$. We record the following easy lemma for use in some later arguments.

\begin{lem}\label{l:monotone}
Let $(M, \omega)$ be monotone with monotonicity constant $\lambda/2$ and suppose that there are smooth disks $v_{1}, v_{2} \colon (\DD^{2}, \SS^{1}) \to (M, L)$ that satisfy $\mu([v_{i}]) = \lambda\mspace{1mu}\omega([v_{i}])$ for $i = 1, 2$ and whose boundaries form an integral basis for $H_{1}(L)$. Then $L$ is monotone with monotonicity constant $\lambda$.
\end{lem}

\subsection{Pseudoholomorphic Spheres}

The basic reference for this material is \cite{MS2}. Let $(M, \omega)$ be a symplectic $4$-manifold endowed with a tame almost complex structure $J$ and let $u \colon \SS^{2} \to M$ be a closed $J$-holomorphic sphere. Then the class $[u] = u_{\ast}[\Sigma] \in H_{2}(M)$ is called the \emph{class of $u$} and satisfies $\omega([u]) \geq 0$ with equality if and only if $u$ is constant. Recall also that the \emph{index of $u$} is given by
\[
\ind(u) = -2 + 2c_{1}([u]).
\]
If $(M, \omega)$ is monotone, then $c_{1}([u]) > 0$ whenever $u$ is nonconstant, and consequently $\ind(u) \geq 0$. Frequently, we will want to choose $J$ to make a certain symplectically embedded sphere $S \sub M$ the image of a $J$-holomorphic sphere. In this case, by a slight abuse of language, we will say that $J$ is chosen to be \emph{integrable near $S$}.

We will make use of the usual machinery of pseudoholomorphic curves in dimension $4$, such as automatic regularity, the adjunction inequality, and positivity of intersections: any two closed, connected $J$-holomorphic curves with nonidentical images have finitely many intersection points, and these always contribute positively to the intersection number. In particular, we have $[u] \cdot [v] = 0$ if and only $u$ and $v$ have disjoint images, and $[u] \cdot [v] = 1$ if and only if $u$ and $v$ have a single transverse intersection. This is a local phenomenon, and we will be able to apply it equally well to intersections between punctured pseudoholomorphic curves in the following section.
 
\subsection{The Blowup Operation}

Our primary sources for this material are \cite{MS1} and \cite{We2}. The complex blowup at a point is a classical operation in complex and algebraic geometry, which can be generalized to the almost complex category as follows. Let $(M, J)$ be an almost complex $4$-manifold, fix a point $p \in M$, and assume that $J$ is integrable in an open neighborhood of $p$. Then the \emph{almost complex blowup of $M$ at $p$}, denoted by $M^{\prime}$, is the disjoint union of $M \sd \{p\}$ and the projectivization of $T_{p}M$, which is the set of all complex $1$-dimensional subspaces in $T_{p}M$. It can be given the structure of a smooth manifold as well as an almost complex structure $J^{\prime}$. Reversing this operation yields a \emph{blowdown map} $\beta \colon M^{\prime} \to M$, which is a pseudoholomorphic diffeomorphism from $M^{\prime} \sd \{\beta^{-1}(p)\}$ to $M \sd \{p\}$. Blowing up increases the rank of the second homology group by $1$. Furthermore, the preimage $E = \beta^{-1}(p)$, which is called the \emph{exceptional divisor}, is an embedded $J^{\prime}$-holomorphic sphere in $M^{\prime}$ with $E \cdot E = -1$.

If $u^{\prime}$ is any $J^{\prime}$-holomorphic sphere in $M^{\prime}$, then obviously $u = \beta \circ u^{\prime}$ is a $J$-holomorphic sphere in $M$. Conversely, a $J$-holomorphic sphere in $M$ will always lift to a $J^{\prime}$-holomorphic sphere in $M^{\prime}$ by removal of singularities. It is not too hard to see that if $u^{\prime}$ is immersed and transverse to $E$, then $u$ is also immersed. We then have
\[
c_{1}([u]) = c_{1}([u^{\prime}\mspace{1mu}]) + [u^{\prime}\mspace{1mu}] \cdot E.
\]
Additionally, the sphere $u$ is embedded and passes through $p$ if and only if $u^{\prime}$ is embedded with $[u^{\prime}\mspace{1mu}] \cdot E = 1$, in which case we have
\[
[u] \cdot [u] = [u^{\prime}\mspace{1mu}] \cdot [u^{\prime}\mspace{1mu}] + 1.
\]
These formulas also hold if $u$ and $u^{\prime}$ are only smooth maps, a fact which we shall occasionally use. Note however that a smooth map in $M$ \emph{need not} always lift to a smooth map in $M^{\prime}$.

There is a symplectic analogue of the almost complex blowup, where the role of a point is played by a symplectically embedded ball. The following proposition, which is essentially Proposition 9.3.3 from \cite{MS2}, describes it in terms of almost complex blowups. See also Theorem 3.13 in \cite{We2}. Recall that the \emph{open ball of capacity $\delta$} is
\[
B(\delta) = \{(z_{1}, z_{2}) \in \CC^{2} : \pi\vert z_{1} \vert^{2} + \pi\vert z_{2} \vert^{2} < \delta\}.
\]
Here is the result.

\begin{prp}\label{p:sympblowup}
Let $(M, \omega)$ be a symplectic $4$-manifold and $B(\delta + \epsilon) \sub M$ be a Darboux ball centered at $p \in M$ for $\delta, \epsilon > 0$. Let $J$ be a tame almost complex structure on $M$, which is integrable on $B(\delta + \epsilon)$ and $(M^{\prime}, J^{\prime})$ be the almost complex blowup of $M$ at $p$ with blowdown map $\beta \colon M^{\prime} \to M$. Then there exists a symplectic form $\omega^{\prime}$ on $M^{\prime}$ that tames $J^{\prime}$, agrees with $\beta^{\ast}\omega$ on $M^{\prime} \sd \beta^{-1}(B(\delta + \epsilon))$, and satisfies
\[
\omega^{\prime}(A) = \omega(\beta_{\ast}A) - \delta(A \cdot E),
\]
where $A \in H_{2}(M^{\prime})$. If $J$ is compatible with $\omega$, then $J^{\prime}$ is compatible with $\smash{\omega^{\prime}}$.
\end{prp}

The symplectic manifold $(M^{\prime}, \omega^{\prime})$ in the previous proposition will be called the \emph{symplectic blowup of $M$ along a Darboux ball of capacity $\delta$}. Up to symplectomorphism, it turns out to be independent of all choices except for the parameter $\delta$ and the symplectic isotopy class of the embedding of $B(\delta)$. However, it follows from the results of \cite{Mc1} that for the manifolds we consider in this paper, there is no dependency on the symplectic isotopy class.

If $u$ is any $J$-holomorphic curve in $M$ that lifts to a $J^{\prime}$-holomorphic curve $u^{\prime}$ in $M^{\prime}$, the formula in Proposition \ref{p:sympblowup} reads
\[
\omega^{\prime}([u^{\prime}\mspace{1mu}]) = \omega([u]) - \delta([u^{\prime}\mspace{1mu}] \cdot E).
\]
We note that the formula in Proposition \ref{p:sympblowup} will also hold if $L \sub M$ is a Lagrangian torus and $A \in H_{2}(M^{\prime}, L)$, so long as $E$ is disjoint from $L$. We will use this fact numerous times in later sections.

We now give a necessary and sufficient condition for the symplectic blowup of a monotone symplectic manifold to be monotone.

\begin{prp}\label{p:monoblowup}
Let $(M, \omega)$ be a monotone symplectic $4$-manifold with monotonicity constant $\kappa > 0$ and $(M^{\prime}, \omega^{\prime})$ be the symplectic blowup of $M$ along a Darboux ball of capacity $\delta > 0$. Then $(M^{\prime}, \omega^{\prime})$ is monotone if and only if $\delta = 1/\kappa$, in which case the monotonicity constant of $(M^{\prime}, \omega^{\prime})$ is $\kappa$.
\end{prp}

\begin{proof}
Let $S^{\prime}$ be an embedded representative of an element of $\pi_{2}(M^{\prime})$, and $S$ be its image under the blowdown map $\beta$. Then $S$ represents an element of $\pi_{2}(M)$, and so $c_{1}(S) = \kappa\mspace{1mu}\omega(S)$ by monotonicity. Then
\begin{align*}
c_{1}(S^{\prime}) &= c_{1}(S) - S^{\prime} \cdot E \\
&= \kappa\mspace{1mu}\omega(S) - S^{\prime} \cdot E \\
&= \kappa(\omega^{\prime}(S^{\prime}) + \delta(S^{\prime} \cdot E)) - S^{\prime} \cdot E \\
&= \kappa\mspace{1mu}\omega^{\prime}(S^{\prime}) + (\kappa \delta - 1)(S^{\prime} \cdot E).
\end{align*}
At this point, it is clear that $\delta = 1/\kappa$ is a sufficient condition for $(M^{\prime}, \omega^{\prime})$ to be monotone. To see that this condition is also necessary, set $S^{\prime} = E$.
\end{proof}

We will also need the \emph{symplectic blowdown} operation. It is defined in a completely analogous fashion to the symplectic blowup, by characterizing it in terms of a certain \emph{almost complex blowdown} operation. All we will need to know for our purposes is that this can be performed on any symplectic $4$-manifold $(M^{\prime}, \omega^{\prime})$ in the presence of a symplectically embedded sphere $E$ with $E \cdot E = -1$. We can arrange for any fixed collection of symplectic submanifolds that are disjoint from $E$ to remain symplectic in the resulting blowdown $(M, \omega)$. There is also a blowdown map $\beta \colon M^{\prime} \to M$, which is a diffeomorphism of $M^{\prime} \sd E$ onto $M \sd \{\beta(E)\}$, as well as an obvious analogue of the formula in Proposition \ref{p:sympblowup}, which we shall use frequently. For more details on these matters, we refer the reader to Section 3.2 and in particular Theorem 3.14 of \cite{We2}.

\subsection{Inflation}

We will need to make use of the \emph{inflation} procedure in some later arguments. The following proposition, which is Theorem 4.9 in \cite{ALL}, contains all the facts that will be relevant to us.

\begin{prp}\label{p:inflate}
Let $(M, \omega)$ be a symplectic $4$-manifold together with a compatible almost complex structure $J$, and let $S \sub M$ be an embedded $J$-holomorphic sphere with $S \cdot S \geq -1$. Then there is a family $\{\omega_{t}\}_{t \geq 0}$ of symplectic forms on $M$, with $t < \omega(S)$ when $S \cdot S = -1$, which all tame $J$ and satisfy $[\omega_{t}] = [\omega] +t\mathrm{PD}(S)$ in $H^{2}_{\dr}(M)$.
\end{prp}

We call an \emph{inflation of capacity $t$ along $S$} the process of replacing $(M, \omega)$ with $(M, \omega_{t})$ in the previous proposition. For any class $A \in H_{2}(M)$, we then have
\[
\omega_{t}(A) = \omega(A) + t(S \cdot A).
\]
In other words, inflation modifies the symplectic area of objects according to their intersections with $A$. Furthermore, if $L \sub M$ is a Lagrangian torus that is disjoint from $S$, then the local nature of the inflation procedure guarantees that the above formula continues to hold when $A \in H_{2}(M, L)$.

\begin{rmk}\label{r:inflation}
The statement we have given in Proposition \ref{p:inflate} is an amalgam of \emph{positive inflation}, as found in \cite{Mc2}, and \emph{negative inflation}, as found in \cite{Bu}. However, our statement assumes a compatible almost complex structure $J$, and produces a family of symplectic forms that only \emph{tame} $J$. This is weaker than the results found in the above two references, but as pointed out in \cite{ALL}, there are certain issues with the inflation procedure when $J$ is only assumed to be tame that are as of yet not fully resolved. The version we are using avoids these issues and is sufficient for our purposes.
\end{rmk}

\section{Punctured Pseudoholomorphic Spheres}

In this section, we recall some basic definitions and facts concerning punctured pseudoholomorphic spheres and buildings. We then discuss the so-called \emph{splitting procedure}, used to produce buildings as limits of closed pseudoholomorphic curves in an appropriate geometric setting. We also record some standard and useful index formulas. The primary references for much of this material are \cite{BEH} and \cite{CM1}. Some additional references that are closer to the spirit of this paper are \cite{CM2}, \cite{DRGI}, \cite{HL}, and \cite{HO}, the last three of which we draw on heavily.

\subsection{Cylindrical Ends}

Let $(V, \alpha)$ be a contact manifold with Reeb vector field $R$, and consider the symplectic manifold
\[
(\RR \times V, d(e^{t}\alpha)).
\]
This is called the \emph{symplectization of $(V, \alpha)$}. A compatible almost complex structure $J$ is then called \emph{cylindrical} if it is preserved by the flow of $\partial_{t}$, if $J(\partial_{t}) = R$, and if $J(\ker\alpha) = \ker\alpha$. All of these constructions carry over to when $\RR$ is replaced by an arbitrary interval.

Now consider a \emph{noncompact} symplectic manifold $(M, \omega)$ containing a compact regular domain $W$ with contact-type boundary given by a disjoint union
\[
(V_{+}, \alpha_{+}) \sqcup (V_{-}, \alpha_{-}).
\]
We require that at least one of $V_{+}$ or $V_{-}$ be nonempty, and allow one or both of them  to be disconnected. We then say that $(M, \omega)$ has \emph{cylindrical ends} if the complement $(M \sd \Int(W), \omega)$ is symplectomorphic to the disjoint union of the symplectizations
\[
E_{+} = ([0, +\infty) \times V_{+}, d(e^{t}\alpha_{+})), \qquad E_{-} = ((-\infty, 0] \times V_{-}, d(e^{t}\alpha_{-})).
\]
The connected components of $E_{+}$ are called the \emph{positive ends of $(M, \omega)$ over $(V_{+}, \alpha_{+})$}, and those of $E_{-}$ are called the \emph{negative ends of $(M, \omega)$ over $(V_{-}, \alpha_{-})$}. (In the literature, these are more commonly called \emph{concave} and \emph{convex} ends, respectively.) There is a \emph{collapse map} $M \to W$ defined by simply collapsing the cylindrical ends of $M$ to the boundary of $W$. The symplectization $(\RR \times V, d(e^{t}\alpha))$ of any contact manifold $(V, \alpha)$ itself has exactly one positive and one negative cylindrical end, both over $(V, \alpha)$. Here one can take $W = [-\epsilon, \epsilon] \times V$ for any $\epsilon > 0$.

Call an almost complex structure $J$ \emph{adapted to $(M, \omega)$} if its restriction to $W$ is tame and its restriction to each component of $E_{+} \cup E_{-}$ is cylindrical. It can be shown using standard arguments that these form a nonempty and contractible set. We can require $J$ to match any given tame almost complex structure on a proper compact subset of $W$, or some given cylindrical almost complex structure far enough along any component of $E_{+} \cup E_{-}$.

\subsection{Punctured Pseudoholomorphic Spheres}

Throughout, let $(M, \omega)$ be a symplectic manifold with cylindrical ends over $(V_{\pm}, \alpha_{\pm})$ and $J$ be an almost complex structure adapted to $(M, \omega)$. An \emph{$n$-punctured $J$-holomorphic sphere in $(M, \omega)$} consists of the following data:
\begin{enumerate}
\item A proper $J$-holomorphic map $u \colon \SS^{2} \sd Z \to M$, where $\SS^{2}$ has a given complex structure and $Z = \{z_{1}, \dots, z_{n}\}$ is a finite set of \emph{punctures}. It is partitioned into $Z_{+} \sqcup Z_{-}$, where the punctures in $Z_{+}$ are called \emph{positive} and those in $Z_{-}$ are called \emph{negative}.
\item A (possibly multiply covered) periodic Reeb orbit $\gamma_{i}$ in $(V_{\pm}, \alpha_{\pm})$ assigned to each puncture $z_{i} \in Z_{\pm}$, parameterized by the circle $\SS^{1} = \RR/\ZZ$ as $t \mapsto \gamma_{i}(P_{i}t)$, where $P_{i}$ is the period of $\gamma_{i}$ obtained by integrating $\alpha$ over it.
\item The map $u$ is required to be \emph{asymptotic to $\gamma_{i}$ at $z_{i}$}, meaning that the restriction of $u$ to a punctured disk centered at $z_{i} \in Z_{\pm}$ has a \emph{cylindrical representation} $(a_{i}, \tilde{u}_{i}) \colon \DD^{2} \sd \{0\} \to \RR \times V_{\pm}$ satisfying
\[
\lim_{r \to 0}a_{i}(re^{2\pi i t}) = \pm \infty, \qquad \lim_{r \to 0}\tilde{u}_{i}(re^{2\pi i t}) = \gamma_{i}(\pm P_{i}t).
\]
These limits are required to be uniform with respect to $t$.
\end{enumerate}
More generally, a $J$-holomorphic map $u$ satisfying (1) and (2) also satisfies (3) if and only if it has \emph{finite energy}, see \cite{BEH}. This will be the case for all punctured spheres considered in this paper. A $1$-punctured $J$-holomorphic sphere is called a \emph{plane}, and a $2$-punctured $J$-holomorphic sphere is called a \emph{cylinder}. Furthermore, the above definition pertains specifically to a \emph{parameterized sphere}. Under a suitable definition of equivalence, we can (and will) think of the image of $u$ above as an \emph{unparameterized sphere}. See \cite{BEH} for a precise definition of this equivalence.

The domain $\SS^{2} \sd Z$ of $u$ can be compactified to a topological surface $\Sigma$ with nonempty boundary by simply replacing each puncture in $Z$ with a circle. The asymptotic properties of $u$ imply that it extends continuously over $\partial\Sigma$ by mapping a boundary circle to the corresponding asymptotic Reeb orbit. Recalling the regular domain $W \sub M$, we can compose the extension of $u$ with the collapse map $M \to W$, obtaining a continuous map $\overline{u} \colon (\Sigma, \partial\Sigma) \to (M, V_{+} \cup V_{-})$, which is called the \emph{compactification of $u$}. Its image defines a relative class $[\overline{u}] \in H_{2}(M, V_{+} \cup V_{-})$, called the \emph{class of $u$}.

Assume now that all periodic Reeb orbits of $\alpha_{\pm}$ come in \emph{Morse-Bott} families, see \cite{Bo}, and denote by $\smash{\Gamma_{i}^{\pm}}$ the connected Morse-Bott manifold containing $\smash{\gamma_{i}^{\pm}}$, where the sign indicates whether the orbit corresponds to a positive or negative puncture. Any symplectic trivialization of the contact structures on $\smash{(V_{\pm}, \alpha_{\pm})}$ over these orbits extends to a symplectic trivialization of $TM\vert_{\smash{\gamma_{i}^{\pm}}} \cong T(\RR \times V_{\pm})\vert_{\smash{\gamma_{i}^{\pm}}}$, since the remaining directions are spanned by the Reeb and global coordinate vector fields. The bundle $u^{\ast}TM$ will then inherit a symplectic trivialization over $\partial\Sigma$, denoted by $\tau$, and we can define the following two quantities:
\begin{enumerate}
\item[$\bullet$] Each orbit $\smash{\gamma_{i}^{\pm} \in  \Gamma_{i}^{\pm}}$ gives rise to an \emph{asymptotic operator} defined on an appropriate space of sections, see \cite{We1}, and relative to $\tau$, it corresponds to a path in $\SP(2)$. This in turn has a \emph{Robbin-Salamon index} $\smash{\mu_{\mathrm{RS}}(\gamma_{i}^{\pm}; \tau)}$ in the sense of \cite{RS}. It is a half-integer that turns out to depend only on $\tau$ and $\smash{\Gamma_{i}^{\pm}}$.
\item[$\bullet$] The bundle $u^{\ast}TM$ has a \emph{relative first Chern class} $c_{1}(u; \tau)$. It can be defined as an algebraic count of the zeros of a generic section $\sigma$ of the line bundle $\Lambda^{2}_{\CC}(u^{\ast}TM)$, where we require $\sigma$ to be nonvanishing and constant with respect to $\tau$ near all cylindrical ends of $(M, \omega)$.
\end{enumerate}
Setting $s_{\pm} = \vert Z_{\pm} \vert$ and following \cite{Bo}, the \emph{index of $u$} (as an \emph{unparameterized} curve with \emph{unconstrained} ends) is then given by
\begin{align*}
\ind(u) = -2 + s_{+} + s_{-} + c_{1}(u; \tau) &+ \sum_{i=1}^{s_{+}}\Big(\mu_{\mathrm{RS}}(\gamma_{i}^{+}; \tau) + \frac{1}{2}\dim\Gamma_{i}^{+}\Big) \\
&- \sum_{i=1}^{s_{-}}\Big(\mu_{\mathrm{RS}}(\gamma_{i}^{-}; \tau) - \frac{1}{2}\dim\Gamma_{i}^{-}\Big).
\end{align*}
It follows immediately from this formula that the total contribution of the last three terms must be independent of $\tau$.

We need the following generalization of punctured spheres. A \emph{punctured nodal $J$-holomorphic sphere in $(M, \omega)$} consists of the following data:
\begin{enumerate}
\item An unordered collection $\mathbf{u} = \{u_{1}, \dots, u_{k}\}$ of punctured $J$-holomor\-phic spheres in $(M, \omega)$. We write $S_{i} = \SS^{2} \sd Z_{i}$ for the punctured domain of the map $u_{i}$.
\item A finite subset $\Delta \sub S_{1} \cup \cdots \cup S_{k}$ with a fixed-point free involution. In particular, the set $\Delta$ consists of an even number of points, called \emph{nodes}.
\item If $p \in S_{i}$ and $q \in S_{j}$ are nodes that are matched under the above involution, then $u_{i}(p) = u_{j}(q)$. It is also required that the surface obtained by taking a connected sum of the $S_{i}$ along all matching nodes be homeomorphic to a punctured sphere.
\end{enumerate}
When $\Delta = \es$, then $\mathbf{u}$ is a punctured $J$-holomorphic sphere as before. In this case, we will say $\mathbf{u}$ is \emph{smooth} to emphasize that there are no nodes. As for smooth punctured spheres, we identify $\mathbf{u}$ with the union of the images of the $u_{i}$, thinking of $\mathbf{u}$ as an \emph{unparameterized nodal sphere}. This can be made precise using a suitable notion of equivalence, see \cite{BEH}. We can also form the compactification of each of the $u_{i}$ as before and take their union, which we also call the \emph{compactification of $\mathbf{u}$}. It gives rise to a homology class $A \in H_{2}(M, V_{+} \cup V_{-})$, which is the sum of the classes of all the $u_{i}$. Finally, we define the \emph{index of $\mathbf{u}$} to be
\[
\ind(\mathbf{u}) = \sum_{i=1}^{k}\ind(u_{i}) + 2(k-1).
\]
This is compatible with the definition of the index of a nodal sphere in the closed case, as is easily checked.

\subsection{The Cotangent Bundle}

The global coordinates $(q_{1}, q_{2})$ on $\smash{\TT^{2}}$ induce coordinates $(q_{1}, q_{2}, p_{1}, p_{2})$ on the cotangent bundle $\smash{T^{\ast}\TT^{2} = \TT^{2} \times \RR^{2}}$. This has a canonical symplectic structure induced by the $2$-form $d\lambda$, where
\[
\lambda = p_{1}d{q_{1}} + p_{2}d{q_{2}}
\]
is the \emph{Liouville $1$-form}. Under the natural identification of $S^{\ast}\TT^{2}$ with the $3$-torus $\TT^{3} = \RR^{3}/\ZZ^{3}$, the restriction $\alpha = \lambda\vert_{S^{\ast}\TT^{2}}$ then takes the form
\[
\alpha = \cos(2\pi\theta)d{q_{1}} + \sin(2\pi\theta)d{q_{2}}.
\]
This is a contact form on $S^{\ast}\TT^{2}$, which becomes a contact-type hypersurface in $(T^{\ast}\TT^{2}, d\lambda)$. Its Reeb vector field is given by
\[
R = \cos(2\pi\theta)\partial_{q_{1}} + \sin(2\pi\theta)\partial_{q_{2}}.
\]
Under the identification of $T^{\ast}\TT^{2}$ with $T\TT^{2}$ induced by the flat metric, the flow of $R$ corresponds to the geodesic flow, and the projection $\smash{S^{\ast}\TT^{2} \to \TT^{2}}$ then induces a correspondence between periodic Reeb orbits of fixed multiplicity on $S^{\ast}\TT^{2}$ and closed oriented geodesics on $\TT^{2}$, which in turn correspond to $\theta \in \SS^{1}$ such that $\tan(2\pi\theta) \in \QQ \cup \{\infty\}$. Thus to each nonzero class $\eta \in H_{1}(\TT^{2})$, there corresponds a $1$-dimensional family of periodic Reeb orbits of fixed multiplicity, with families of simple Reeb orbits corresponding to primitive homology classes. We denote by $\Gamma_{\eta}$ the manifold of Reeb orbits corresponding to the class $\eta$. This is a Morse-Bott manifold in the sense of \cite{Bo}.

Consider the symplectomorphism $(\RR \times S^{\ast}\TT^{2}, d(e^{t}\alpha)) \to (T^{\ast}\TT^{2} \sd \TT^{2}, d\lambda)$ given by
\[
(t, q_{1}, q_{2}, p_{1}, p_{2}) \mapsto (q_{1}, q_{2}, e^{t}p_{1}, e^{t}p_{2}).
\]
Restricting it to $[0, +\infty) \times S^{\ast}\TT^{2}$ shows that $(T^{\ast}\TT^{2}, d\lambda)$ has a single positive cylindrical end over $S^{\ast}\TT^{2}$, and restricting it to $(-\infty, 0] \times S^{\ast}\TT^{2}$ shows that $(D_{1}^{\ast}\TT^{2} \sd \TT^{2}, d\lambda)$ has a single negative end over $S^{\ast}\TT^{2}$. We then define an almost complex structure $J_{\mathrm{std}}$ on $T^{\ast}\TT^{2}$ by
\[
J_{\mathrm{std}}(\partial_{q_{i}}) = -\rho(\Vert (p_{1}, p_{2}) \Vert)\partial_{p_{i}}, \qquad i = 1, 2,
\]
where $\rho \colon [0, +\infty) \to (0, +\infty)$ is a smooth nondecreasing function that satisfies $\rho(t) = 1$ on $[0, 1]$ and $\rho(t) = t$ on $[2, +\infty)$. We also define an almost complex structure $J_{\mathrm{cyl}}$ on $\RR \times S^{\ast}\TT^{2}$ by
\[
J_{\mathrm{cyl}}(\partial_{q_{i}}) = -\Vert (p_{1}, p_{2}) \Vert\partial_{p_{i}}, \qquad i = 1, 2,
\]
using the identification coming from the above symplectomorphism. It is shown in Lemma 4.1 of \cite{DRGI} that $J_{\mathrm{cyl}}$ is a cylindrical almost complex structure on $(\RR \times S^{\ast}\TT^{2}, d(e^{t}\alpha))$ and that $J_{\mathrm{std}}$ is adapted to $(T^{\ast}\TT^{2}, d\lambda)$.

\subsection{Pseudoholomorphic Buildings}

We now narrow our focus to the setting in which buildings will arise in this paper. Let $(M, \omega)$ be a closed symplectic $4$-manifold and $L = L_{1} \cup \cdots \cup L_{k}$ be a union of pairwise disjoint Lagrangian tori in $M$. We will work inside the so-called \emph{split symplectic manifold}
\[
(M \sd L, \omega) \sqcup (T^{\ast}L, d\lambda),
\]
where $\lambda$ is the Liouville $1$-form on $T^{\ast}L$.

For $i = 1, \dots, k$, fix a parametrization $\psi_{i} \colon \TT^{2} \to L_{i}$. We can assume the induced Weinstein neighborhoods $\mathcal{U}(L_{i})$ are pairwise disjoint, and by rescaling the flat metric on $\TT^{2}$, we can also assume that $\mathcal{U}(L_{i})$ contains $\psi_{i}(D^{\ast}_{4}\TT^{2})$. We declare $\psi_{i}$ to be an isometry, so that $\psi_{i}(D^{\ast}_{4}\TT^{2}) = D^{\ast}_{4}L_{i}$. The entire discussion from the previous subsection can then be carried over to the current setting. In particular, since
\[
(D_{1}^{\ast}L_{i} \sd L_{i}, \omega) \cong (D_{1}^{\ast}\TT^{2} \sd \TT^{2}, d\lambda),
\]
we see that $(M \sd L, \omega)$ has a negative cylindrical end and that $(T^{\ast}L, d\lambda)$ has a positive cylindrical over $S^{\ast}L_{i}$ for each $i$. The almost complex structures $J_{\mathrm{std}}$ and $J_{\mathrm{cyl}}$ can also be transported to almost complex structures on $T^{\ast}L$ and $\RR \times S^{\ast}L$, respectively. We will then say a tame almost complex structure $J$ on $M \sd L$ is \emph{adapted to the parameterizations $\psi_{i}$} if, for each $i$, it takes the form
\[
J(\partial_{q_{i}}) = -\Vert (p_{1}, p_{2}) \Vert\partial_{p_{i}}
\]
inside $\mathcal{U}(L_{i}) \sd L_{i}$. Since this set contains $D_{1}^{\ast}L_{i} \sd L_{i}$ by assumption, it follows that $J$ is adapted to $(M \sd L, \omega)$.

We define a \emph{pseudoholomorphic building in $(M \sd L, \omega) \sqcup (T^{\ast}L, d\lambda)$} as a finite collection $\mathbf{F}$ of punctured nodal pseudoholomorphic spheres, called the \emph{components} of the building, which are divided among various \emph{levels} according to the following scheme:
\begin{enumerate}
\item[$\bullet$] \textsc{Top level}: A finite number of $J$-holomorphic components contained in $(M \sd L, \omega)$. There is always at least one such component. 
\item[$\bullet$] \textsc{Middle levels}: A finite number of $J_{\mathrm{cyl}}$-holomorphic components contained in consecutively ordered copies of $(\RR \times S^{\ast}L, d(e^{t}\alpha))$, of which there may be none. If there are however, then each must contain at least one component that is not a trivial cylinder over a periodic Reeb orbit.
\item[$\bullet$] \textsc{Bottom level}: A finite number of $J_{\mathrm{std}}$-holomorphic components contained in $(T^{\ast}L, d\lambda)$. There may be no such components.
\end{enumerate}
The smooth curves that constitute each component of $\mathbf{F}$ are collectively referred to as the \emph{curves of $\mathbf{F}$}. We require the following conditions:
\begin{enumerate}
\item If $u$ is a curve of $\mathbf{F}$ contained in the $i$th level and asymptotic to a parameterized Reeb orbit in $(S^{\ast}L, \alpha)$ at some positive puncture, there must exist a curve in the $(i+1)$th level asymptotic to that same orbit at a negative puncture. Likewise, if $u$ is asymptotic to a parameterized Reeb orbit in $(S^{\ast}L, \alpha)$ at some negative puncture, there must exist a curve in the $(i-1)$th level asymptotic to that same orbit at a positive puncture. All punctures must furthermore match bijectively in pairs.
\item Let $\Sigma$ be the (nodal) surface formed by taking the connected sum of all component domains along all pairs of asymptotically matching punctures and by identifying nodes under the involution in each domain. Then $\Sigma$ must be homeomorphic to a (nodal) sphere.
\end{enumerate}
We will say that $\mathbf{F}$ is a \emph{broken nodal sphere} if it consists of at least two nonempty levels, and an \emph{unbroken nodal sphere} otherwise. In either case, we replace the qualification \emph{nodal} by \emph{smooth} if each component of $\mathbf{F}$ is smooth (that is, if its domain has no nodes). In this case, the surface $\Sigma$ in (2) above will be homeomorphic to a genuine sphere. We also identify $\mathbf{F}$ with its image under a suitable notion of equivalence for buildings, see \cite{BEH}. In this way, we think of $\mathbf{F}$ as an \emph{unparameterized building}.

If $u$ is a top level curve of $\mathbf{F}$, then its compactification $\overline{u}$ is a map into $M$ with boundary on $S^{\ast}L$ corresponding to its asymptotic Reeb orbits. We extend it by cylinders to a map whose boundary components are the corresponding geodesics on $L$, obtaining a homology class $[\overline{u}] \in H_{2}(M, L)$, also called the \emph{class of $u$}. We then define the \emph{area of $u$} to be $\omega(\overline{u})$. (Observe that when $u$ has no punctures, this is just the usual symplectic area of a closed $J$-holomorphic sphere.) It can then be shown that $\omega(\overline{u}) \geq 0$, with equality if and only if $u$ is constant.

On the other hand, the compactifications of curves of $\mathbf{F}$ in the middle or bottom levels can be projected into $L$. We then obtain chains whose boundaries are geodesics corresponding to boundaries of compactified top level curves. Identifying these boundaries gives rise to a continuous map $\overline{\mathbf{F}} \colon \Sigma \to M$, called the \emph{compactification of $\mathbf{F}$}. Its image defines a class $A \in H_{2}(M)$, which we call the \emph{class of $\mathbf{F}$}. The \emph{area of $\mathbf{F}$} is then defined to be $\omega(A)$. It is clearly the sum of the areas of all top level curves, and we have $\omega(A) \geq 0$ with equality if and only if $\mathbf{F}$ is constant, see \cite{CM1}.

The following proposition records some additional facts about buildings.

\begin{prp}\label{p:buildingfacts}
Let $\mathbf{F}$ be a building in the class $A \in H_{2}(M)$.
\begin{enumerate}
\item[(a)] If $\mathbf{F}$ is broken, the it has at least two top level planes. 
\item[(b)] If there is no decomposition $A = A_{1} + A_{2}$ with $\omega(A_{1}), \omega(A_{2}) > 0$, then $\mathbf{F}$ is smooth.
\item[(c)] If the tori $L_{i}$ are all integral, then every top level curve of $\mathbf{F}$ has integral area.
\end{enumerate}
\end{prp}

\begin{proof}
To prove (a), first note that any nonconstant middle or bottom level curve of $\mathbf{F}$ has at least one puncture, since otherwise by exactness, it would have zero area. Furthermore, if it had exactly one puncture, then it would project to a disk in $S^{\ast}L$ bounding a Reeb orbit, which is impossible since all orbits are noncontractible for a flat torus. The curve must therefore have at least two punctures, and middle level curves must have at least one positive puncture by the maximum principle for subharmonic functions. From these facts it is easy to deduce that $\mathbf{F}$ has at least two planes in the top level.

The proof of (b) is clear if $\mathbf{F}$ is constant, so assume that $\mathbf{F}$ is nonconstant and nodal, and let $\Sigma$ be the domain of its compactification $\overline{\mathbf{F}}$. Then $\Sigma$ can be decomposed as $\Sigma_{1} \cup \Sigma_{2}$, where $\Sigma_{1} \cap \Sigma_{2}$ consists of a single point, such that the restriction of $\overline{\mathbf{F}}$ to each $\Sigma_{i}$ is the compactification of a nonconstant pseudoholomorphic building $\mathbf{F}_{i}$ whose homology class $A_{i} \in H_{2}(M)$ satisfies $\omega(A_{i}) > 0$ by (a). But this would lead to a decomposition $A = A_{1} + A_{2}$ with $\omega(A_{1}), \omega(A_{2}) > 0$, contradicting our assumption.

Finally, to prove (c), note that $(M, \omega)$ must then be integral. This implies that every spherical homology class in $H_{2}(M)$ has integral area, and the result for any top level curve without punctures follows. On the other hand, if $u$ was a top level curve with at least one puncture, then its compactification $\overline{u}$ represents a relative class in $H_{2}(M, L)$ and has boundary components on some of the tori $L_{1}, \dots, L_{k}$. We can then choose disks in $M$ that bound these components, and these have integral area by assumption. Gluing these disks to $\overline{u}$ yields a spherical class in $H_{2}(M)$, which also has integral area. By additivity, it then follows that $u$ must also have integral area.
\end{proof}

\subsection{The Splitting Procedure}

The primary method used for producing pseudoholomorphic buildings in this paper is the \emph{splitting procedure}, also known as \emph{stretching the neck}. This can be carried out in a fairly general setting, see \cite{BEH} and \cite{CM1}. We will give a description of this procedure that is tailored to our specific context, following \cite{DRGI}.

With the same setup and notation as the previous subsection, consider a family $\{J_{t}\}_{t \geq 0}$ of tame almost complex structures on $(M, \omega)$, which for each $i$ are defined according to the following prescription:
\begin{enumerate}
\item[$\bullet$] Outside all of the open sets $B_{4}^{\ast}L_{i}$, we set $J_{t} \equiv J$ for each $t \geq 0$.
\item[$\bullet$] Inside each of the open sets $B_{2}^{\ast}L_{i}$, we set $J_{t} \equiv J_{\mathrm{std}}$ for each $t \geq 0$.
\item[$\bullet$] Inside each $B_{4}^{\ast}L_{i} \sd B_{2}^{\ast}L_{i}$, we identify $J_{t}$ for each $t \geq 0$ with the restriction of $J_{\mathrm{cyl}}$ to the set $\smash{[\log{2}, \log{4} + t) \times S^{\ast}L_{i}}$ under the identifications
\begin{align*}
B_{4}^{\ast}L_{i} \sd B_{2}^{\ast}L_{i} &\approx [\log{2}, \log{4}) \times S^{\ast}L_{i} \\
&\approx [\log{2}, \log{4} + t) \times S^{\ast}L_{i}.
\end{align*}
Note that only the first of these is a symplectomorphism.
\end{enumerate}
We say the family $\{J_{t}\}_{t \geq 0}$ \emph{stretches to $J$}, and in this context we refer to the latter as a \emph{stretched} almost complex structure on $M \sd L$. We can then extract a sequence $(J_{t_{i}})_{i=1}^{\infty}$ such that $t_{i} \to +\infty$ as $i \to +\infty$. In this setting, we have the following \emph{compactness theorem}.

\begin{thm}\label{t:compactness}
Let $\smash{(u_{i})_{i=1}^{\infty}}$ be a sequence of closed $J_{t_{i}}$-holomorphic spheres in $(M, \omega)$ in a fixed class $A \in H_{2}(M)$. Then there exists a subsequence that converges to a pseudoholomorphic building in $(M \sd L, \omega) \sqcup (T^{\ast}L, d\lambda)$ whose homology class is also $A$.
\end{thm}

The building obtained in the conclusion of this theorem is called a \emph{limit building}. The exact definition of convergence is not relevant to us, since all we care about is the \emph{existence} of such buildings. See \cite{BEH} or \cite{CM1} for a more comprehensive discussion as well as proofs of the above theorem.

\begin{rmk}\label{r:splitting}
It is possible to generalize the discussion in this and the previous subsection by allowing $(M, \omega)$ to be noncompact with cylindrical ends. This requires a generalization of the notions of pseudoholomorphic building and convergence, but the only time we will need this more general theory is in Lemma \ref{l:simpli1} below. There we will be applying the splitting procedure given the contact-type hypersurfaces we have been considering so far as well as the boundary of a symplectically embedded ball. We will study the limit of a sequence of top level planes of minimal area, all belonging to the same relative homology class, and thus having uniformly bounded energy. This condition is sufficient to ensure compactness, and so no complicated limit configurations will arise. Because of this, we will not give an exposition of this more general setting, instead referring the reader to \cite{BEH}.
\end{rmk}

\subsection{Index Calculations}

With the setup and notation of the previous two subsections, let $\mathbf{u}_{1}, \dots, \mathbf{u}_{l}$ be the components of the building $\mathbf{F}$. If $m$ is the total number of Reeb orbits at which these components match asymptotically, then following \cite{HO}, we define the \emph{index of $\mathbf{F}$} to be
\[
\ind(\mathbf{F}) = \sum_{i=1}^{l}\ind(\mathbf{u}_{i}) - m.
\]
This is designed so that if $(u_{i})_{i=1}^{\infty}$ is a sequence of closed pseudoholomorphic spheres in $A \in H_{2}(M)$ converging to $\mathbf{F}$ as in Theorem \ref{t:compactness}, then
\[
\ind(\mathbf{F}) = \ind(u_{i}) = -2 + 2c_{1}(A).
\]
Computing the index of $\mathbf{F}$ therefore reduces to computing the index of its components, and thus the index of its curves in various levels. Here is one result in this direction.

\begin{prp}\label{p:indexmiddle}
Let $u$ be a punctured pseudoholomorphic sphere having $s_{+}$ positive punctures and $s_{-}$ negative punctures in $(\RR \times S^{\ast}L, d(e^{t}\alpha))$ or $(T^{\ast}L, d\lambda)$, where $s_{-} = 0$ in the latter case. Then
\[
\ind(u) = -2 + 2s_{+} + s_{-}.
\]
This index is nondecreasing under multiple covers, and is positive if $u$ is a middle or bottom level curve of a limit building.
\end{prp}

\begin{proof}
A proof of the formula can be found in \cite{DRGI} or \cite{HL}. The fact that it is nondecreasing under multiple covers follows from an argument using the Riemann-Hurwitz formula as in the proof of Lemma 3.2 of \cite{DRGI}. The last statement follows from the fact that, in this case, any such curve always has at least two punctures by (a) of Proposition \ref{p:buildingfacts}.
\end{proof}

The next proposition, which follows from the discussion in Section 3.1 of \cite{DRGI}, provides a convenient formula for the index of a top level curve using an extension of the Maslov index for surfaces with boundary on $L$. We refer to Appendix C.3 of \cite{MS2} for an exact definition.

\begin{prp}\label{p:indextop}
Let $u$ be a punctured $J$-holomorphic sphere in $M \sd L$ with $s_{-}$ negative punctures. Then
\[
\ind(u) = -2 + s_{-} + \mu(\overline{u}).
\]
This index is nondecreasing under multiple covers.
\end{prp}

\section{Initial Setup}

In this section, we begin by setting up our primary objects of study. Next, we discuss the elements of a \emph{relative foliation theory}, which will be crucial to our arguments in later sections. We then study some special properties of the Clifford torus in the context of such relative foliations.

\subsection{Del Pezzo Surfaces}

Consider the symplectic manifold $(\CP^{2}, \omega_{\FS})$, where the Fubini-Study form $\omega_{\FS}$ is scaled so that $\omega_{\FS}(\CP^{1}) = 3$. In this way $(\CP^{2}, \omega_{\FS})$ becomes a monotone symplectic manifold with monotonicity constant $1$. Recall that the standard torus $\TT^{2}$ is viewed as the quotient $\RR^{2}/\ZZ^{2}$ with global coordinates $(q_{1}, q_{2})$ for $0 \leq q_{i} \leq 1$. Following \cite{MS1}, the moment map $\mu \colon \CP^{2} \to \RR^{2}$ for the standard Hamiltonian toric action of $\TT^{2}$ on $(\CP^{2}, \omega_{\FS})$ is given in homogeneous coordinates by
\[
\mu([z_{0} : z_{1} : z_{2}]) = 3\left(\frac{\vert z_{1} \vert^{2}}{\Vert z \Vert^{2}}, \frac{\vert z_{2} \vert^{2}}{\Vert z \Vert^{2}}\right).
\]
The image of $\mu$ is the right triangle $\{(x, y) \in \RR^{2} : 0 \leq x + y \leq 3\}$, and we define the subsets
\[
S_{\infty} = \{x + y = 3\}, \qquad T_{0} = \{y = 0\}, \qquad T_{\infty} = \{x = 0\}
\]
and denote their preimages under $\mu$ by the same symbols. These are in fact symplectically embedded spheres in $\CP^{2}$, the first being the \emph{line at infinity}. The \emph{Clifford torus} is defined as
\[
L_{1, 1} = \mu^{-1}(1, 1) = \{[z_{0} : z_{1} : z_{2}] \in \CP^{2} : \vert z_{0} \vert = \vert z_{1} \vert = \vert z_{2} \vert\}.
\]
It is a monotone Lagrangian torus with monotonicity constant $2$. The following diagram illustrates the moment image with all of these objects:
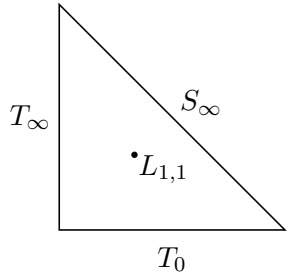
\[
\begin{tikzpicture}[scale=1]
\draw[semithick] (0, 0) -- (3, 0) -- (0, 3) -- cycle;
\node[yshift = -1em] at (1.5, 0) {$T_{0}$};
\node[xshift = -1em] at (0, 1.5) {$T_{\infty}$};
\node[xshift = 1em, yshift = 0.5em] at (1.5, 1.5) {$S_{\infty}$};
\filldraw[black] (1, 1) circle (1pt);
\node[xshift = 1em, yshift = -0.5em] at (1, 1) {$L_{1, 1}$};
\node at (0,3) {\phantom{$T_{0}$}};
\end{tikzpicture}
\]
By Proposition \ref{p:monoblowup}, the only way to obtain monotone symplectic manifolds from $(\CP^{2}, \omega_{\FS})$ by symplectically blowing up is to do so along Darboux balls of capacity $1$. It is a known fact that one can embed up to $8$ such disjoint closed balls in $\CP^{2}$ before encountering a volume obstruction. We then obtain for $n = 1, \dots, 8$ a monotone symplectic manifold $(\DD_{n}, \omega_{\mspace{1mu}\DD_{n}})$ with monotonicity constant $1$. These together with the symplectic manifolds $(\SS^{2} \times \SS^{2}, \sigma \oplus \sigma)$, where $\sigma(\SS^{2}) = 2$, and $(\CP^{2}, \omega_{\FS})$, are known as the \emph{symplectic Del Pezzo surfaces}. As we shall show in Proposition \ref{p:reduction}, once Theorem \ref{t:main} is proved for the case $n = 1$ then it will follow immediately for $n =2, \dots, 5$. Because of this, we will now examine a particularly explicit model of $(\DD_{1}, \omega_{\mspace{1mu}\DD_{1}})$ and record some of its features.

Consider the symplectic embedding $B(3) \hookrightarrow \CP^{2}$ centered at the point $[1 : 0 : 0]$, given by
\[
(z_{1}, z_{2}) \mapsto \left[\sqrt{(3/\pi) - \vert z_{1} \vert^{2} - \vert z_{2} \vert^{2}} : z_{1} : z_{2}\right].
\]
This restricts to a symplectic embedding of $B(1 + \epsilon)$, with $\epsilon > 0$ chosen small enough so that $B(1 + \epsilon)$ misses $L_{1, 1}$. We then apply Proposition \ref{p:sympblowup} to obtain the desired model of $(\DD_{1}, \omega_{\mspace{1mu}\DD_{1}})$. For reference, we will henceforth denote this model by $(X, \omega_{X})$. One of its salient features is an induced torus action with moment image obtained from that of $(\CP^{2}, \omega_{\FS})$ by excising the triangle $\{0 \leq x + y < 1\}$. The proper transforms of $S_{\infty}$, $T_{0}$, and $T_{\infty}$ in $X$ are still symplectically embedded spheres and are denoted by the same symbols. Additionally, the exceptional divisor in $(X, \omega_{X})$, denoted by $S_{0}$, corresponds to the edge $\{x + y = 1\}$ in the moment image. The following diagram sums up these observations:
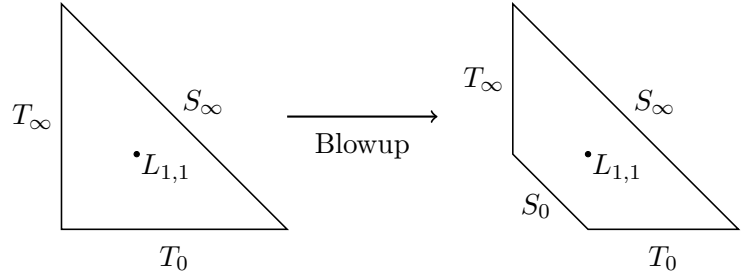
\[
\begin{tikzpicture}
\draw[semithick] (0, 0) -- (3, 0) -- (0, 3) -- cycle;
\node[xshift = 1em, yshift = 0.5em] at (1.5, 1.5) {$S_{\infty}$};
\node[yshift = -1em] at (1.5, 0) {$T_{0}$};
\node[xshift = -1em] at (0, 1.5) {$T_{\infty}$};
\filldraw[black] (1, 1) circle (1pt);
\draw[thick, ->] (3, 1.5) -- (5, 1.5);
\node[yshift = -1em] at (4, 1.5) {Blowup};
\draw[semithick] (7, 0) -- (9, 0) -- (6, 3) -- (6, 1) -- cycle;
\node[xshift = -0.5em, yshift = -0.5em] at (6.5, 0.5) {$S_{0}$};
\node[xshift = 1em, yshift = 0.5em] at (7.5, 1.5) {$S_{\infty}$};
\node[yshift = -1em] at (8, 0) {$T_{0}$};
\node[xshift = -1em] at (6, 2) {$T_{\infty}$};
\filldraw[black] (7, 1) circle (1pt);
\node[xshift = 1em, yshift = -0.5em] at (1, 1) {$L_{1, 1}$};
\node[xshift = 1em, yshift = -0.5em] at (7, 1) {$L_{1, 1}$};
\node at (0,3) {\phantom{$T_{0}$}};
\end{tikzpicture}
\]
The configuration of spheres $S_{0} \cup S_{\infty} \cup T_{0} \cup T_{\infty}$ will be a recurring object in this paper. We will refer to it as the \emph{axes of $(X, \omega_{X})$} and denote it by $Y$.

The second homology group $H_{2}(X)$ is free abelian of rank $2$ and contains three distinguished classes. There is the class of the spheres $T_{0}$ and $T_{\infty}$, as well as the classes $S_{0}$ and $S_{\infty}$. We will use $(S_{0}, T_{0})$ as an ordered basis for $H_{2}(X)$, identifying it with $\ZZ^{2}$. Thus $S_{\infty} = S_{0} + T_{0}$ is the class $(1, 1)$. The corresponding symplectic areas are
\[
\omega_{X}(S_{0}) = 1, \qquad \omega_{X}(T_{0}) = 2, \qquad \omega_{X}(S_{\infty}) = 3.
\]
By monotonicity, these are also the values of the corresponding Chern numbers. We also record some intersection numbers for future reference:
\begin{gather*}
S_{0} \cdot S_{0} = -1, \qquad T_{0} \cdot T_{0} = 0, \qquad S_{\infty} \cdot S_{\infty} = 1, \\
S_{0} \cdot S_{\infty} = 0, \qquad T_{0} \cdot S_{0} = 1, \qquad T_{0} \cdot S_{\infty} = 1.
\end{gather*}
In later considerations, we will work with tame almost complex structures making these spheres pseudoholomorphic, and positivity of intersections will then allow us to deduce many important conclusions.

\subsection{Existence of Pseudoholomorphic Spheres}

For many arguments in this paper, we will need to understand which (unparameterized) closed pseudoholomorphic spheres exist in $(X, \omega_{X})$. The following proposition contains all the facts we will need.

\begin{prp}\label{p:blowupcurves}
Let $J$ be a tame almost complex structure on $(X, \omega_{X})$.
\begin{enumerate}
\item[(a)] There is a unique embedded $J$-holomorphic sphere in the class $(1, 0)$.
\item[(b)] For any point $p \in X$, there is a unique embedded $J$-holomorphic sphere in the class $(0, 1)$ passing through $p$.
\item[(c)] For $d \geq 1$ and a codimension $2$ set of tuples $(p_{1}, \dots, p_{2d}) \in X^{d}$, there is a unique embedded $J$-holomorphic sphere in the class $(1, d)$ passing through the points $p_{1}, \dots, p_{2d}$.
\end{enumerate}
\end{prp}

\begin{proof}
Let $A$ denote any of the above homology classes. We first observe that because $A$ is primitive, every \emph{smooth} $J$-holomorphic sphere in the class $A$ must be simple and hence somewhere injective. It then follows from the adjunction inequality that any such sphere must be embedded. Furthermore, it follows from automatic regularity, with point constraints when necessary, that $J$ is regular for all spheres in the class $A$. The dimension of the corresponding constrained or unconstrained moduli space is always $0$. Furthermore, the self-intersection number $A \cdot A$ is strictly less than the number of point constraints (if any), implying that there is at most one sphere in the class $A$ satisfying these constraints.

Next, we prove existence of these curves in the class $A$ for the integrable complex structure $J_{0}$ on $X$. In this case the complex manifold $(X, J_{0})$ is known to be biholomorphic to the projectivization $\PP(\mathcal{O}(-1) \oplus \CC)$, where $\mathcal{O}(-1)$ is the line bundle of degree $-1$ over $\CP^{1}$. The unique $J_{0}$-holomorphic sphere in the class $(1, 0)$ is then precisely the zero section, and the sphere in the class $(0, 1)$ through a fixed point is precisely the compactified fiber of the above bundle passing through that point. The sphere in the class $(1, d)$ is a meromorphic section passing through $d - 1$ zeros on $S_{0}$ and $d$ poles on $S_{\infty}$, together with an additional normalizing point constraint chosen away from these two subsets. These points have to be chosen to lie on distinct fibers, meaning we have to avoid a codimension $2$ subset of $X^{2d}$.

Now, we show that spheres in the classes $(1, 0)$ and $(0, 1)$ exist for any tame $J$ by showing that no bubbling can occur, so that all relevant parametric moduli spaces, with path constraints in the case of (b), are compact. When $A = (1, 0)$, this follows since any decomposition $A = A_{1} + A_{2}$ with each $A_{i}$ representable by a $J$-holomorphic curve implies that $\omega_{X}(A_{i}) \geq 0$, and since $\omega_{X}(A) = 1$, this implies that one of the $A_{i}$ is the trivial class by integrality. When $A = (0, 1)$, a similar decomposition $A = A_{1} + A_{2}$ is trivial unless $\omega_{X}(A_{i}) = 1$ for $i = 1, 2$. Writing $A_{i} = (a_{i}, b_{i})$, this implies that $a_{i} + 2b_{i} = 1$, showing that this homology class is primitive. We also have $a_{1} + a_{2} = 0$ and $b_{1} + b_{2} = 1$. If $a_{1} = 0$ or $a_{2} = 0$, then the above decomposition would be trivial, so without loss of generality, we can assume that $a_{1} \geq 1$, and this forces $b_{1} \leq 0$. Then the adjunction inequality for spheres in the class $A_{2} = (-a_{1}, 1 - b_{1})$ eventually simplifies to
\[
0 \leq (a_{1} + 1)(2b_{1} - a_{1}),
\]
a contradiction since the right hand side is negative by assumption.

Finally, we show that no bubbling can occur when $A = (1, d)$ with $d \geq 1$, provided the point constraints are chosen away from a subset of codimension $2$ in $X^{2d}$. Using the fact that there exists $J$-holomorphic spheres in the classes $(1, 0)$ and $(0, 1)$ together with positivity of intersections, one sees that the only possible decomposition of the class $(1, d)$ into classes representable by $J$-holomorphic spheres is
\[
(1, d) = (1, d_{1}) + (0, d_{2}) + \cdots + (0, d_{k}),
\]
where $d_{1} + \cdots + d_{k} = d$ and $d_{i} \geq 0$ for each $i$. This then forces $k \leq d + 1$. Letting $S$ denote the unique $J$-holomorphic sphere in the class $(1, 0)$, we pick our point constraints so that at most $d - 1$ of them lie on $S$, and such that no two of them lie on the same $J$-holomorphic sphere in the class $(0, 1)$. We thus have to avoid a codimension $2$ subset of $X^{2d}$, and it is easy to check that any decomposition of the class $(1, d)$ as above would contradict this choice of constraints. Since we can also find disjoint paths of permissible point constraints, this shows that the corresponding parametric moduli space is always compact, and hence that the desired spheres exist for any $J$.
\end{proof}

\begin{rmk}\label{r:iso}
Implicit in the above proof is the fact that any two symplectic spheres in the class $A$ can be joined by a symplectic isotopy. In certain circumstances, we will also be able to find symplectic isotopies of certain \emph{nodal configurations} of symplectic spheres, such as the axes $Y = S_{0} \cup S_{\infty} \cup T_{0} \cup T_{\infty}$. By Corollary 3.7 in \cite{DRGI}, these symplectic isotopies can then be realized as Hamiltonian ones. By definition, we require that any two spheres in a nodal configuration intersect transversally, positively, and no more than once.
\end{rmk}

\subsection{A Displacement Result}

We now prove a displacement result, which, in addition to being needed in later sections, will be used to prove two important facts. The first of these will allow us to reduce the proof of Theorem \ref{t:main} to the case $n = 1$, and the second is a striking fact about integral tori in $(\DD_{n}, \omega_{\mspace{1mu}\DD_{n}})$. Here is the displacement result.

\begin{prp}\label{p:displaceclosed}
If $L$ is a Lagrangian torus in $(X, \omega_{X})$, then there is a Hamiltonian diffeomorphism $\varphi \colon X \to X$ such that $\varphi(L)$ is disjoint from $S_{0}$. If $L$ is also disjoint from $L_{1, 1}$, the corresponding Hamiltonian isotopy can be taken to be stationary on $L_{1, 1}$.
\end{prp}

\begin{proof}
We first find an unbroken smooth symplectic sphere in the class $(1, 0)$ that is disjoint from $L$. To this end, fix any parameterization $\psi \colon \TT^{2} \to L$, let $J$ be an almost complex structure on $X \sd L$ that is adapted to $\psi$ and regular for somewhere injective curves, and let $\{J_{t}\}_{t \geq 0}$ be a family of tame almost complex structures on $(X, \omega_{X})$ that stretches to $J$. We can extract a sequence $(J_{t_{i}})_{i=1}^{\infty}$ and consider the unique sequence of embedded $J_{t_{i}}$-holomorphic spheres in the class $(1, 0)$, which exist by (a) of Proposition \ref{p:blowupcurves}. Applying Theorem \ref{t:compactness}, we obtain a limit building $\mathbf{F}$ in the class $(1, 0)$, which is a smooth (and possibly broken) sphere by (b) of Proposition \ref{p:buildingfacts}. Letting $u_{1}, \dots, u_{k}$ be the top level curves of $\mathbf{F}$ and $v_{1}, \dots, v_{l}$ be all remaining curves, we have
\[
0 = -2 + 2c_{1}((1, 0)) = \ind(\mathbf{F}) = \sum_{i=1}^{k}\ind(u_{i}) + \sum_{i=1}^{l}\ind(v_{i}) - m,
\]
where $m$ is the total number of Reeb orbits at which the above curves match asymptotically. Now we know from Proposition \ref{p:indexmiddle} that
\[
\ind(v_{i}) = -2 + 2s_{+}^{i} + s_{-}^{i} = -\chi(S_{i}) + s_{+}^{i},
\]
where $S_{i}$ is the punctured domain of $v_{i}$. We also have $s_{+}^{1} + \cdots + s_{+}^{k} = m$, so
\[
0 = \sum_{i=1}^{k}\ind(u_{i}) + \sum_{i=1}^{l}\ind(v_{i}) - m = \sum_{i=1}^{k}\ind(u_{i}) - \sum_{i=1}^{l}\chi(S_{i}).
\]
Since each domain $S_{i}$ has at least two punctures, we know that $\chi(S_{i}) \leq 0$ for each $i$. But now observe that $\ind(u_{i}) \geq 0$ for each $i$, which follows immediately from the fact that $J$ is regular when $u_{i}$ is somewhere injective, and then by Proposition \ref{p:indextop} if $u_{i}$ is multiply covered. The above equality then forces $\ind(u_{i}) = 0$ for each $i$. But by (a) of Proposition \ref{p:buildingfacts}, at least one of the $u_{i}$ must be a plane, and then $\ind(u_{i}) = \mu(\overline{u}) - 1$ by Proposition \ref{p:indextop} again. Since $L$ is orientable, the Maslov index $\mu(\overline{u})$ is even, which forces $\ind(u_{i}) > 0$, a contradiction. It follows that $\mathbf{F}$ is an unbroken smooth sphere in the class $(1, 0)$ that is disjoint from $L$.

Now extend $J$ in $\mathcal{U}(L)$ away from $\mathbf{F}$ to a smooth tame almost complex structure defined on all of $X$, and consider a path of tame almost complex structures $\{J_{t}\}_{t \in [0, 1]}$ with $J_{0}$ integrable near $S_{0}$ and $J_{1} = J$. It follows from (a) of Proposition \ref{p:blowupcurves} and automatic regularity that the parametric moduli space for spheres in the class $(1, 0)$ is a compact, $1$-dimensional cobordism that gives rise to a symplectic isotopy taking $\mathbf{F}$ to $S_{0}$. This can then be upgraded to a Hamiltonian isotopy as mentioned in Remark \ref{r:iso}, and the corresponding time-one map is the desired Hamiltonian diffeomorphism $\varphi$. This proves the first statement in the theorem.

Finally, assume that $L$ is disjoint from $L_{1, 1}$ and fix a parametrization $\psi_{1, 1}$ of $L_{1, 1}$. (An explicit parameterization is given in Subsection 5.1 below). Let $J$ be tame, regular for somewhere injective curves, and adapted to both $\psi$ and $\psi_{1, 1}$. The above arguments then go through with no modification to show that there is a unique unbroken smooth $J$-holomorphic sphere $\mathbf{F}$ in the class $(1, 0)$ that is disjoint from both $L$ and $L_{1, 1}$. Furthermore, there is such a sphere disjoint from $L_{1, 1}$ for \emph{all} tame almost complex structures adapted to $\psi_{1, 1}$, and not just regular ones. This is because $L_{1, 1}$ is integral, which by (c) of Proposition \ref{p:buildingfacts} implies that \emph{any} limit building in the class $(1, 0)$ is unbroken. As before, we take a path $\{J_{t}\}_{t \in [0, 1]}$ of tame almost complex structures on $X \sd L_{1, 1}$, with $J_{0}$ integrable near $S_{0}$ and $J_{1} = J$. The corresponding parametric moduli space is a compact, $1$-dimensional cobordism giving rise to a symplectic isotopy in $X \sd L_{1, 1}$ taking $\mathbf{F}$ to $S_{0}$. The corresponding Hamiltonian isotopy can then be extended to all of $X$ and taken to be stationary on $L_{1, 1}$, completing the proof of the second statement and consequently of the proposition.
\end{proof}

We now present two important applications of Proposition \ref{p:displaceclosed}. The first reduces the proof of Theorem \ref{t:main} to the case $n = 1$.

\begin{prp}\label{p:reduction}
If $L_{1, 1}$ is a maximal integral packing of $(\DD_{1}, \omega_{\mspace{1mu}\DD_{1}})$, then it is also a maximal integral packing of $(\DD_{n}, \omega_{\mspace{1mu}\DD_{n}})$ for $n = 2, \dots, 5$.
\end{prp}

\begin{proof}
After $n \leq 5$ blowups, there are $n$ pairwise disjoint exceptional divisors $E_{1}, \dots, E_{n}$ in $(\DD_{n}, \omega_{\mspace{1mu}\DD_{n}})$, each of area $1$ and with self-intersection number $-1$. These are symplectically embedded spheres, and consequently, one can find a tame almost complex structure integrable near all of them. The same arguments used in the proof of Proposition \ref{p:blowupcurves} then show that for any tame $J$, there is a unique $J$-holomorphic sphere in the class of each $E_{i}$. If $L \sub \DD_{n}$ were an integral torus that is disjoint from $L_{1, 1}$, a straightforward adaptation of the proof of Proposition \ref{p:displaceclosed} produces a Hamiltonian isotopy in $(\DD_{n}, \omega_{\mspace{1mu}\DD_{n}})$ that is stationary near $L_{1, 1}$ and displaces $L$ from the configuration $E_{1} \cup \cdots \cup E_{n}$. Blowing down each of these except $E_{1}$ then yields an integral torus in $(\DD_{1}, \omega_{\mspace{1mu}\DD_{1}})$ that is disjoint from $L_{1, 1}$, a contradiction since this latter torus is a maximal integral packing of $(\DD_{1}, \omega_{\mspace{1mu}\DD_{1}})$.
\end{proof}

The second application of Proposition \ref{p:displaceclosed} is the following nice result about integral Lagrangian tori in $(\DD_{n}, \omega_{\mspace{1mu}\DD_{n}})$, which aside from being interesting in its own right, will greatly simplify some later considerations. Also, since the result is no harder to prove for $(\CP^{2}, \omega_{\FS})$, we include that case as well.

\begin{thm}\label{t:monotone}
All integral Lagrangian tori in $(\CP^{2}, \omega_{\FS})$ and $(\DD_{n}, \omega_{\mspace{1mu}\DD_{n}})$ for $n = 1, \dots, 8$ are monotone.
\end{thm}

\begin{proof}
First off, the proof of Proposition \ref{p:reduction} shows that any integral torus in $(\DD_{n}, \omega_{\mspace{1mu}\DD_{n}})$ can be displaced from the configuration $E_{1} \cup \cdots \cup E_{n}$ by a Hamiltonian isotopy. Blowing down these exceptional divisors then yields an integral torus in $(\CP^{2}, \omega_{\FS})$, and it is therefore enough to prove the theorem in this case. Hence let $L \sub \CP^{2}$ be an integral torus, and use Theorem C of \cite{DRGI} to displace $L$ from $S_{\infty}$ by a Hamiltonian isotopy, thereby placing it inside the open ball $B(3)$. We can thus assume that $L$ is an integral torus in $(\RR^{4}, \omega_{0})$, and in this setting a homological argument allows us to recast the Maslov and area homomorphisms as maps from $H_{1}(L)$. By Lemma \ref{l:monotone}, it is then enough to find an integral basis $\{e_{1}, e_{2}\}$ of $H_{1}(L)$ such that $\mu(e_{i}) = 2\omega_{0}(e_{i})$ for $i = 1, 2$.

We now appeal to the results in Section 1.1 of \cite{CM2} to obtain a smooth map $v_{1} \colon (\DD^{2}, \SS^{1}) \to (\CP^{2}, L)$ such that $\omega_{\FS}([v_{1}]) = 1$ and $\mu([v_{1}]) = 2$. More precisely, the proofs of those results produce three maps, each satisfying these requirements, and which are pseudoholomorphic away from $L$ and obtained from the splitting procedure as a limit of spheres in the class $S_{\infty}$. By positivity of intersections, only one of these can then intersect $S_{\infty}$, and so we may assume without loss of generality that $v_{1}(\DD^{2}) \sub B(3)$. Letting $e_{1}$ be the class of  $v_{1}\vert_{\SS^{1}}$, we get $\mu(e_{1}) = 2$ and $\omega_{0}(e_{1}) = 1$. Now choose any integral class $e_{2}$ such that $(e_{1}, e_{2})$ is a basis for $H_{1}(L)$. Since $L$ is orientable, we know that $\mu(e_{2})$ is even, and by adding to it integer multiples of $e_{1}$ if necessary, we may assume that $\mu(e_{2}) = 2$. We will then be done if we can show that $\omega_{0}(e_{2}) = 1$. If this were not the case, then either $\omega_{0}(e_{2}) > 1$ or $\omega_{0}(e_{2}) < 1$, in which case we redefine $e_{2}$ to be the class $e_{2} + 2(e_{1} - e_{2})$. We are then left with an integral basis $(e_{1}, e_{2})$ of $H_{1}(L)$ such that
\[
\omega_{0}(e_{1}) = 1, \qquad \omega_{0}(e_{2}) \geq 2, \qquad \mu(e_{1}) = 2, \qquad \mu(e_{2}) = 2.
\]
This puts us in the scenario of Theorem 2 in \cite{HO}, which says that in the presence of such a basis, the inclusion $L \sub B(R)$ implies that $R > 3$, contradicting the present situation.
\end{proof}

\subsection{Relative Foliations}

Let $L = L_{1} \cup \cdots \cup L_{k}$ be a disjoint union of integral Lagrangian tori in $(X, \omega_{X})$, itself disjoint from $S_{0}$. By Proposition \ref{p:displaceclosed}, this is not a very restrictive situation. (The proof adapts in a straightforward way to deal with several tori.) We then fix a parametrization $\psi_{i} \colon \TT^{2} \to L_{i}$ for each $i$, and always assume that the induced Weinstein neighborhoods are pairwise disjoint and miss $S_{0}$. Finally, we choose an almost complex structure $J$ on $X \sd L$ that is adapted to all the $\psi_{i}$ and that is integrable near $S_{0}$. We will use this collection of data to construct a foliation of $X \sd L$ by $J$-holomorphic spheres and planes that is in some sense \emph{compatible with $L$}. Its existence and properties underlie all subsequent arguments in this paper. The construction is carefully discussed in \cite{DRGI} and \cite{HL}, and we use these as our primary references for this subsection. We will present all the machinery that we will need in the form of a single all-encompassing theorem.

\begin{thm}\label{t:foliation}
There is a foliation $\mathcal{F} = \mathcal{F}(L_{1}, \dots, L_{k}, \psi_{1}, \dots, \psi_{k}, J)$ of $X \sd L$ and a unique \textbf{foliation class} $\beta_{i} \in H_{1}(L_{i})$ for each $i$, which satisfy the following properties:
\begin{enumerate}
\item The foliation $\mathcal{F}$ is comprised of both \textbf{unbroken} leaves, which are $J$-holomorphic spheres in the class $(0, 1)$, as well as \textbf{broken} leaves, which are pairs of $J$-holomorphic planes asymptotic to Reeb orbits in the Morse-Bott families $\Gamma_{\beta_{i}}$ and $\Gamma_{-\beta_{i}}$ for some $i$.
\item The leaves of $\mathcal{F}$ are all simply covered and embedded, and the planes comprising a broken leaf are asymptotic to simply covered orbits. In particular, the foliation classes $\beta_{i}$ are all primitive.
\item Every leaf of $\mathcal{F}$ has a single transverse intersection with $S_{0}$, which for a broken leaf occurs with the plane asymptotic to an orbit in $\Gamma_{\beta_{i}}$.
\end{enumerate}
Moreover, the foliation classes $\beta_{i}$ are independent of the choice of $J$.
\end{thm}

\begin{proof}
Take a family $\{J_{t}\}_{t \geq 0}$ of tame almost complex structures on $(X, \omega_{X})$ that stretches to $J$, and consider buildings in the class $(0, 1)$ obtained as limits of $J_{t_{i}}$-holomorphic spheres in the class $(0, 1)$ passing through any fixed point in $X$, where $t_{i} \to +\infty$ as $i \to +\infty$. These sequences exist by (b) of Proposition \ref{p:blowupcurves}, and the corresponding limit building $\mathbf{F}$ is smooth by (b) of Proposition \ref{p:buildingfacts}, and has area $2$. It is therefore either an unbroken smooth sphere in $X \sd L$ in the class $(0, 1)$, or it is broken, in which case (a) of Proposition \ref{p:buildingfacts} implies that it must have at least two planes in its top level. But since the $L_{i}$ are all integral, these planes will have integral area by (c) of Proposition \ref{p:buildingfacts}, and thus $\mathbf{F}$ has exactly two planes in its top level, each of area $1$. There can be no additional top level curves, and curves in the middle and bottom levels are all cylinders. As explained in Section 5 of \cite{HL}, a diagonal subsequence argument together with positivity of intersections shows that the collection of top level curves of all such buildings give a partition of $X \sd L$. That it is actually a smooth foliation follows from automatic regularity for punctured curves, see \cite{We1}, together with an extended intersection number argument, as in the proof of Proposition 5.1 in \cite{HL}. To avoid a digression into Siefring intersection theory, we refer the reader to that source for the full argument.

Since $J$ is integrable near $S_{0}$, it follows from positivity of intersections that exactly one of the top level planes of a broken sphere as above intersects $S_{0}$. We denote this plane by $u$ and the remaining one by $v$. Suppose these are asymptotic to Reeb orbits on $S^{\ast}L_{i}$ for some fixed $i$. Then $u$ is asymptotic to an orbit contained in $\Gamma_{\beta_{i}}$ for some $\beta_{i} \in H_{1}(L_{i})$, and for topological reasons, the plane $v$ must be asymptotic to an orbit in $\Gamma_{-\beta_{i}}$. Furthermore, both $u$ and $v$ are simply covered since they have area $1$, and it follows that they must be embedded, being somewhere injective components of a limit of closed embedded spheres. To see that their asymptotic orbits are also simply covered, fix an unbroken sphere from our partition in the class $(0, 1)$ intersecting $S_{0}$ at a single point, and follow the proof of Proposition 4.12 in \cite{LU} to construct a symplectically embedded sphere $S$ in the class $(1, 1)$ that is disjoint from $S_{0}$. Observe that $S$ is also disjoint from $v$, and blowing down $S_{0}$ yields a plane $v$ in $(\CP^{2}, \mspace{2mu}\omega_{\FS})$. By Theorem 1.4 of \cite{Mc1}, there is a symplectomorphism of $(\CP^{2}, \mspace{2mu}\omega_{\FS})$ with itself taking $S$ to $S_{\infty}$. Keeping the same notation, the plane $v$ is now contained in $\CP^{2} \sd S_{\infty}$, and hence a subset of  $(\RR^{4}, \omega_{0})$. Viewing $L_{i}$ as a torus in $\RR^{4}$, it now follows that the compactification of $v$ represents a primitive homology class in $H_{2}(\RR^{4}, L_{i})$. But then it also follows that its boundary $-\beta_{i}$ represents a primitive class in $H_{1}(L_{i})$, since the connecting homomorphism $H_{2}(\RR^{4}, L_{i}) \to H_{1}(L_{i})$ is an isomorphism. Therefore $-\beta_{i}$ and $\beta_{i}$ are primitive classes in $H_{1}(L_{i})$, implying that the asymptotic orbits of $u$ and $v$ are simply covered.

Now consider another limit building and suppose its asymptotic orbits belonged to $\Gamma_{\gamma_{i}} \cup \Gamma_{-\gamma_{i}}$ for $\gamma_{i} \neq \beta_{i}$ in $H_{1}(L_{i})$. As shown in Lemma 5.8 of \cite{DRGI}, positivity of intersections implies that no component of this building can intersect a component of our earlier building in a discrete and nonempty set. But $\beta_{i}$ and $\gamma_{i}$ are both primitive classes, hence not collinear, and the description of cylinders in lower levels from Lemma 4.2 and Corollary 4.3 of \cite{DRGI} would then \emph{force} some of these curves to have a nonempty and discrete intersection, a contradiction. Hence $\gamma_{i} = \pm\beta_{i}$, and in fact $\gamma_{i} = \beta_{i}$, since otherwise, after possibly reparameterizing the plane asymptotic to an orbit in $\Gamma_{\gamma_{i}}$ and using an explicit cylinder in $T^{\ast}L_{i}$ as described in Lemma 4.2 of \cite{DRGI}, we can construct a sphere in some class $(a, b)$ of area $2$ and with two intersections with $S_{0}$. These conditions would force $a = -2/3$, which is impossible. This shows that the foliation classes $\beta_{i}$ are well defined.

Finally, we show that the foliation classes $\beta_{i}$ are independent of $J$. Choose a path $\{J_{t}\}_{t \in [0, 1]}$ of tame almost complex structures that are all adapted to the $\psi_{i}$ and integrable near $S_{0}$, with $J_{0} = J$. The planes corresponding to broken leaves of the foliation obtained from $J$ divide into moduli spaces $\mathcal{M}_{\beta_{i}}(J)$ and $\mathcal{M}_{-\beta_{i}}(J)$ of planes with asymptotics in $\Gamma_{\beta_{i}}$ and $\Gamma_{-\beta_{i}}$, respectively. The corresponding parametric moduli spaces with asymptotics in these same families are all smooth by automatic regularity, and compact since area $1$ planes cannot degenerate further by the integrality of the $L_{i}$. We therefore obtain smooth moduli spaces $\mathcal{M}_{\beta_{i}}(J_{1})$ and $\mathcal{M}_{-\beta_{i}}(J_{1})$ of embedded, simply covered $J_{1}$-holomorphic planes with simply covered asymptotic orbits, and the planes in $\mathcal{M}_{\beta_{i}}(J_{1})$ have a single transverse intersection with $S_{0}$. If we then construct a foliation of $X \sd L$ by $J_{1}$-holomorphic planes and spheres as in the first paragraph, positivity of intersections forces the resulting top level planes of broken spheres to coincide with the union of the two moduli spaces above. Hence the foliation $\mathcal{F}$ exists and has the desired properties for $J_{1}$, concluding the proof of the theorem.
\end{proof}

The foliation $\mathcal{F}$ constructed in the previous theorem induces a smooth projection map $\pi \colon X \to S_{0}$, which is defined as follows:
\begin{enumerate}
\item[$\bullet$] If $p \in X \sd L$, then $p$ lies on a unique (broken or unbroken) leaf of $\mathcal{F}$, and this leaf has a unique component that intersects $S_{0}$. Then $\pi(p)$ is defined to be this point of intersection.
\item[$\bullet$] If $p \in L_{i}$ for some $i$, then $p$ lies on a unique geodesic in the class $\beta_{i}$, since these foliate the torus. This geodesic corresponds to a unique simply covered orbit on $S^{\ast}L_{i}$, and the same argument used to prove Lemma 5.2 in \cite{HL} shows that there is a \emph{unique} broken leaf of $\mathcal{F}$ with a plane component asymptotic to the above orbit. Then $\pi(p)$ is defined to be the unique point of intersection of this leaf with $S_{0}$.
\end{enumerate}
As shown in Section 5 of \cite{DRGI}, the images $\pi(L_{i}) \sub S_{0}$ are pairwise disjoint embedded circles. Furthermore, it follows from Proposition 5.16 in \cite{DRGI} that the fiber $\pi^{-1}(p)$ for $p \in \pi(L_{i})$ is a smooth sphere obtained by gluing compactifications of planes in $\mathcal{M}_{\beta_{i}}(J)$ and $\mathcal{M}_{-\beta_{i}}(J)$ with matching orbits, after perturbing these planes in an arbitrarily small neighborhood of $L_{i}$. The intersection $L_{i} \cap \pi^{-1}(p)$ is the (unoriented) geodesic on $L_{i}$ covered by the above orbits.

\subsection{Essential Curves}

We carry over the setup and notation from the previous subsection. If $u$ is a punctured $J$-holomorphic sphere in $X \sd L$, it will be convenient to say $u$ has a \emph{puncture on $L_{i}$} if it has a puncture at which it is asymptotic to a Reeb orbit on $S^{\ast}L_{i}$. We will say such a puncture is of \emph{foliation type} if this orbit covers a closed geodesic on $L_{i}$ that is an integer multiple of the foliation class $\beta_{i}$. We then say the puncture is \emph{positive} or \emph{negative} according to whether this integer is positive or negative, respectively. (This should hopefully not cause confusion with the standard usage of these terms, since our definition applies only to curves in $X \sd L$.) The following lemma gives a geometric description of the behavior of punctures of foliation type. It follows directly from the standard exponential convergence theorem for punctured curves, which is Theorem 1.4 in \cite{HWZ}.

\begin{lem}\label{l:punctures}
Let $u$ be a punctured $J$-holomorphic sphere in $X \sd L$ and $\{C_{j}\}_{j=1}^{\infty}$ be a sequence of circles contained in a coordinate neighborhood of a puncture of $u$ on $L_{i}$. Assume each $C_{j}$ winds once around the puncture and converges to it in the Hausdorff topology.
\begin{enumerate}
\item[(a)] If the puncture is of foliation type, then the sets $\pi(u(C_{j}))$ converge to a single point on $\pi(L_{i})$. Additionally, each such set either coincides with the point, which implies that $u$ covers a plane in a broken leaf of $\mathcal{F}$, or it winds nontrivially around the point.
\item[(b)] If the puncture is not of foliation type, then the sets $\pi(u(C_{j}))$ converge to the entire circle $\pi(L_{i})$.
\end{enumerate}
\end{lem}

There is a special class of curves in $X \sd L$ for which the classification of punctures according to the above scheme is particularly transparent. A (closed or punctured) $J$-holomorphic sphere in $X \sd L$ is said to be \emph{essential with respect to $\mathcal{F}$}, or simply \emph{essential}, if its composition with the projection $\pi \colon X \sd L \to S_{0}$ is injective. These curves are in a sense well-detected by the foliation. They are partly distinguished by the following proposition.

\begin{prp}\label{p:essential}
Let $u$ be an essential $J$-holomorphic sphere in $X \sd L$. If $u$ has punctures on $L_{i}$, then they are either all of foliation type, in which case they are either all positive or all negative, or none of them are, in which case $u$ has a unique puncture on $L_{i}$.
\end{prp}

\begin{proof}
By Lemma \ref{l:punctures}, if $u$ simultaneously had punctures of and not of foliation type on $L_{i}$, then the projection of a small circle around a puncture of foliation type will intersect any circle sufficiently close to $\pi(L_{i})$, including the projection of a small circle around a puncture not of foliation type. This contradicts the essentiality of $u$ since $\pi \circ u$ is certainly not injective. On the other hand, if $u$ had two punctures not of foliation type on $L_{i}$, then the image of $\pi \circ u$ would have to intersect the regions in $S_{0}$ both inside and outside of $\pi(L_{i})$, since $u$ is essential. By connectedness, this would imply that some interior point of $u$ must intersect a fiber of $\pi$ over a point in $\pi(L_{i})$. Considering the image of a small circle around this point as in Lemma \ref{l:punctures}, the argument in the first part of this paragraph again yields a contradiction, and hence $u$ cannot have more than one puncture on $L_{i}$ that is not of foliation type. That punctures of foliation type are either all positive or all negative is proved in the same way as Lemma 6.2 in \cite{HL}, which we refer the reader to for additional details.
\end{proof}

The following scenario provides a rich source of essential curves. Consider a family $\{J_{t}\}_{t \geq 0}$ of tame almost complex structures on $(X, \omega_{X})$ that stretches to $J$, and extract a sequence $(J_{t_{i}})_{i=1}^{\infty}$ as usual. We then fix $d \geq 1$ and use (c) of Proposition \ref{p:blowupcurves} to obtain a sequence $(u_{i})_{i=1}^{\infty}$ of $J_{t_{i}}$-holomorphic spheres in the class $(1, d)$ passing through $2d$ generic point constraints. By Theorem \ref{t:compactness}, this produces a limit building $\mathbf{F}$ in the class $(1, d)$. We then have the following lemma, which among other things shows that $\mathbf{F}$ always has an essential top level curve.

\begin{lem}\label{l:buildingtop}
If $u$ is a top level curve of $\mathbf{F}$, then $u$ is either essential or else the image of $\pi \circ u$ is a single point. Moreover, essential top level curves of $\mathbf{F}$ have disjoint open images under $\pi$ whose union contains $S_{0} \sd \pi(L)$.
\end{lem}

\begin{proof}
If the image of $\pi \circ u$ is a point, then $u$ covers part of a broken leaf or an unbroken leaf of $\mathcal{F}$. On the other hand, if $\pi \circ u$ is nonconstant, then since $(1, d) \cdot (0, 1) = 1$, we see that $u$ intersects any unbroken leaf of $\mathcal{F}$ once or not at all. This implies that any double points of $\pi \circ u$ lie on $\pi(L)$. But by positivity of intersections, the nonconstant map $\pi \circ u$ is open and hence its double points form an open set, which is impossible in view of the above. It follows that $u$ must be essential, and all such curves have disjoint images under $\pi$ by applying this argument to the union of two curves. Intersection considerations also imply that all unbroken leaves of $\mathcal{F}$ must intersect at least one essential curve, implying the last statement.
\end{proof}

\begin{rmk}\label{r:generality}
At this point, we must mention that the considerations starting from Subsection 4.2 until now were stated in the context of the toric model $(X, \omega_{X})$ for concreteness, but they evidently apply \emph{mutatis mutandis} to any model of $(\DD_{1}, \omega_{\mspace{1mu}\DD_{1}})$ which contains $L_{1, 1}$, and we will now take this to be understood.
\end{rmk}

\subsection{Homologicaly Essential Tori}

This subsection will shed some light on why the considerations thus far should have any bearing on the packing problem we are considering. It begins with a lemma whose motivation might appear unclear but whose importance will become evident shortly. In line with Remark \ref{r:generality}, we state it in the context of an arbitrary model of $(\DD_{1}, \omega_{\mspace{1mu}\DD_{1}})$ containing $L_{1, 1}$. We let $\mathcal{S}_{0}$ denote the corresponding exceptional divisor, and consider a \emph{monotone} torus $L$ in $\DD_{1}$. Assume there is a symplectically embedded sphere $\mathcal{S}_{\infty} \sub \DD_{1}$ in the class of the line at infinity that is disjoint from $L$. (It turns out such a sphere can always be found, but we will not need this fact nor elaborate on it.)

Fix a parameterization $\psi \colon \TT^{2} \to L$, and choose a tame almost complex structure $J$ on $\DD_{1} \sd L$ that is adapted to $\psi$ and integrable near $\mathcal{S}_{0} \cup \mathcal{S}_{\infty}$. Using Theorem \ref{t:foliation}, we obtain a relative foliation $\mathcal{F} = \mathcal{F}(L, \psi, J)$ of $\DD_{1} \sd L$ by $J$-holomorphic spheres and planes. The broken leaves of $\mathcal{F}$ consist of pairs of embedded, simply covered, $J$-holomorphic planes, and there is a unique primitive foliation class $\beta \in H_{1}(L)$ such that the totality of these planes divide into moduli spaces $\mathcal{M}_{\beta}(J)$ and $\mathcal{M}_{-\beta}(J)$. The space $\mathcal{M}_{\beta}(J)$ consists of exactly those planes that have a single transverse intersection with $\mathcal{S}_{0}$. Here is the aforementioned lemma.

\begin{lem}\label{l:homtriv}
Assume that the planes in $\mathcal{M}_{-\beta}(J)$ are exactly those that intersect $\mathcal{S}_{\infty}$ in a single point, and let $\mathcal{T}_{0}$ and $\mathcal{T}_{\infty}$ be two unbroken leaves of $\mathcal{F}$ whose images under $\pi$ lie on opposite sides of $\pi(L)$. Then $L$ is homologicaly nontrivial in $\DD_{1} \sd (\mathcal{S}_{0} \cup \mathcal{S}_{\infty} \cup \mathcal{T}_{0} \cup \mathcal{T}_{\infty})$.
\end{lem}

\begin{proof}
Set $\mathcal{Y} = \mathcal{S}_{0} \cup \mathcal{S}_{\infty} \cup \mathcal{T}_{0} \cup \mathcal{T}_{\infty}$. We construct a relative homology class in $H_{2}(\DD_{1}, \mathcal{Y})$ whose intersection number with $L$ is nonzero, thus showing that $L$ is homologicaly nontrivial in $\DD_{1} \sd \mathcal{Y}$. To this end, choose an embedded path $\gamma \colon [0, 1] \to \mathcal{S}_{0}$ that starts at $\pi(\mathcal{T}_{0}) = \mathcal{T}_{0} \cap \mathcal{S}_{0}$ and that ends at $\pi(\mathcal{T}_{\infty}) = \mathcal{T}_{\infty} \cap \mathcal{S}_{0}$. Since $\pi(L)$ separates these points, we can assume that $\gamma$ intersects it transversally at some unique time $t_{0}$. For each $t \in [0, 1]$, we then choose an embedded path $\sigma_{t}$ in $\pi^{-1}(\gamma(t))$ that starts at $\mathcal{S}_{0}$ and that ends at $\mathcal{S}_{\infty}$. The union of these defines a homology class $\Sigma \in H_{2}(\DD_{1}, \mathcal{Y})$. Now by construction, the set $K = L \cap \pi^{-1}(\gamma(t_{0}))$ is a circle bounding (perturbed) compactified planes from $\mathcal{M}_{\beta}(J)$ and $\mathcal{M}_{-\beta}(J)$, which glue together to form a smooth sphere intersecting $\mathcal{S}_{0}$ and $\mathcal{S}_{\infty}$. Since $K$ separates these intersection points, it must intersect the path $\sigma_{t_{0}}$, and this is the only intersection between $L$ and $\Sigma$. Therefore $L$ must be homologicaly nontrivial in $\DD_{1} \sd \mathcal{Y}$.
\end{proof}

In particular, we can apply the above lemma in the toric model $(X, \omega_{X})$ when the set $\mathcal{Y}$ is the axes $Y = S_{0} \cup S_{\infty} \cup T_{0} \cup T_{\infty}$. The importance of producing homologicaly essential tori in $X \sd Y$ stems from the following crucial fact.

\begin{prp}\label{p:exact}
Any two monotone tori in $X \sd Y$ that are homologicaly nontrivial must intersect.
\end{prp}

\begin{proof}
Let $L \sub X \sd Y$ be a monotone torus that is homologicaly essential in $X \sd Y$. We can symplectically identify this latter set with a symplectic manifold of the form $(\TT^{2} \times U, d\lambda) \sub (T^{\ast}\TT^{2}, d\lambda)$, where $U \sub \RR^{2}$ is a convex subset containing the origin, and by Proposition \ref{p:cotangentexact}, it is enough to show that $L$ is exact when considered in $\TT^{2} \times U$. By (1) of Theorem B in \cite{DR}, there is a Hamiltonian isotopy in $(T^{\ast}\TT^{2}, d\lambda)$ from $L$ to the graph of a closed $1$-form on $\TT^{2}$, and, crucially, this isotopy can be taken to be supported in $\TT^{2} \times U$. We can therefore assume without loss of generality that $L$ is a section of $T^{\ast}\TT^{2} \to \TT^{2}$. Then taking any closed curve $\gamma \sub L$ together with a smooth bounding disk $D$ in $X$, we can construct a cylinder $C$ in $\TT^{2} \times U$ by projecting $\gamma$ onto the zero section and taking the trace of this projection. It is straightforward to see that $C$ has Maslov index $0$ in $T^{\ast}\TT^{2}$. By additivity of Maslov indices, we then see that $\mu(D \cup C) = \mu(D)$, and therefore the monotonicity of $L$ implies that $d\lambda(C) = 0$. Since the zero section is obviously exact in $T^{\ast}\TT^{2}$, it now follows from Stoke's theorem that $\smash{\int_{\gamma}\lambda = 0}$, showing that $L$ is exact in $\smash{T^{\ast}\TT^{2}}$, as desired.
\end{proof}

\section{Towards a Contradiction}

In this section, we begin working towards the proof of Theorem \ref{t:main} in earnest. The strategy is to assume it fails and derive a contradiction. Thus our standing assumption for the rest of the paper will be that there is an integral torus $L$ in $(\DD_{n}, \omega_{\mspace{1mu}\DD_{n}})$ for $n \leq 5$ that is disjoint from $L_{1, 1}$. By Proposition \ref{p:reduction}, we can take $n = 1$ and by Theorem \ref{t:monotone}, we know that $L$ is monotone. We will first show that this standing assumption leads to a slightly weaker statement about tori in the toric model $(X, \omega_{X})$ that is easier to contradict. We then use the relative foliations developed in the previous section together with an analysis of limit buildings to establish the existence of a distinguished building in $(X, \omega_{X})$ with some special properties. Finally, we state a black box result, which we will use to obtain the desired contradiction in Section 6. The proof of this black box result will in turn be taken up in Section 7.

\subsection{A Simplification}

This subsection is dedicated to proving the following proposition, in view of the standing assumption mentioned in the introduction to this section. Recall that $Y = S_{0} \cup S_{\infty} \cup T_{0} \cup T_{\infty}$ in $(X, \omega_{X})$.

\begin{prp}\label{p:simpli}
There are disjoint monotone Lagrangian tori $L_{\mathfrak{r}}$ and $L_{\mathfrak{s}}$ in $X \sd Y$ with the following properties:
\begin{enumerate}
\item[(a)] $L_{\mathfrak{r}}$ is homologicaly trivial in $X \sd Y$.
\item[(b)] $L_{\mathfrak{s}}$ is homologicaly nontrivial in $X \sd Y$.
\item[(c)] $T_{0}$ and $T_{\infty}$ are distinct classes in $H_{2}(X \sd L_{\mathfrak{r}})$ and in $H_{2}(X \sd L_{\mathfrak{s}})$.
\end{enumerate}
\end{prp}

There are several steps involved in proving the above proposition. The first is to displace $L$ from the exceptional divisor inside $\DD_{1} \sd L_{1, 1}$ by using Proposition \ref{p:displaceclosed} to displace it by a Hamiltonian isotopy that is stationary on $L_{1, 1}$. In view of the definition of symplectic blowup, the resulting situation is equivalent to the existence of a monotone Lagrangian torus $L$ in $(\CP^{2}, \omega_{\FS})$ that is disjoint from $L_{1, 1}$ together with a closed Darboux ball $B$ of capacity $1$ in $\CP^{2} \sd (L \cup L_{1, 1})$. To proceed further, we must now take a closer look at the Clifford torus in $\CP^{2}$.

Define a parameterization $\psi_{1, 1} \colon \TT^{2} \to L_{1, 1}$ by
\[
\psi_{1, 1}(q_{1}, q_{2}) = [1 : e^{-2\pi i q_{1}} : e^{-2\pi i q_{2}}].
\]
This extends to a symplectic embedding, denoted by the same symbol, of the open neighborhood
\[
\{(q_{1}, q_{2}, p_{1}, p_{2}) \in T^{\ast}\TT^{2} : \vert p_{1} \vert < 1/2, \vert p_{2} \vert < 1/2\},
\]
which is given explicitly by
\[
\psi_{1, 1}(q_{1}, q_{2}, p_{1}, p_{2}) = \Bigg[1 : \sqrt{\frac{p_{1} + 1}{1 - p_{1} - p_{2}}}e^{-2\pi i q_{1}} : \sqrt{\frac{p_{2} + 1}{1 - p_{1} - p_{2}}}e^{-2\pi i q_{2}}\Bigg].
\]
Its image is a Weinstein neighborhood $\mathcal{U}(L_{1, 1})$, which in practice will usually be made quite smaller. Observe that in these coordinates the moment map takes the form
\[
(q_{1}, q_{2}, p_{1}, p_{2}) \mapsto (p_{1} + 1, p_{2} + 1).
\]
Consider the smooth embedded disk in $\CP^{2}$ with boundary on $L_{1, 1}$ given by $z \mapsto [z : 1 : 1]$. This has Maslov index $2$ and a unique transverse and positive intersection with $S_{\infty}$. A calculation using polar coordinates on $\DD^{2} \sd \{0\}$ and coordinates $\smash{(q_{1}, q_{2}, p_{1}, p_{2})}$ on $\mathcal{U}(L_{1, 1}) \sd L_{1, 1}$ shows that $J_{\mathrm{cyl}}$ preserves the span of its differential. We may then choose a tame almost complex structure $J_{1, 1}$ on $\CP^{2} \sd L_{1, 1}$ that is adapted to $\psi_{1, 1}$, integrable near $S_{\infty}$, and makes the interior of the above disk, which we denote by $u$, into a $J_{1, 1}$-holomorphic plane with a puncture on $L_{1, 1}$. It is asymptotic at its unique puncture to the Reeb orbit $\smash{(\theta, \theta, 1/\sqrt{2}, 1/\sqrt{2})} \in S^{\ast}\TT^{2}$.

The same considerations apply to the smooth embedded disks given by
\[
z \mapsto \left[1 : \frac{1}{\sqrt{2 - \vert z \vert^{2}}} : \frac{z}{\sqrt{2 - \vert z \vert^{2}}}\right], \qquad 
z \mapsto \left[1 : \frac{z}{\sqrt{2 - \vert z \vert^{2}}} : \frac{1}{\sqrt{2 - \vert z \vert^{2}}}\right].
\]
The interiors of these disks are then $J_{1, 1}$-holomorphic planes with punctures on $L_{1, 1}$, which we denote by $v$ and $w$, respectively. The plane $v$ has a unique transverse and positive intersection with $T_{0}$, and is asymptotic at its unique puncture to the Reeb orbit $(-\theta, 0, -1, 0)$. Similarly, the plane $w$ has a unique transverse and positive intersection with $T_{\infty}$, and is asymptotic at its unique puncture to the Reeb orbit $(0, -\theta, 0, -1)$. By monotonicity, all three of the above planes have index and area $1$. These facts together with automatic regularity imply that the corresponding moduli spaces are all compact and smooth $1$-manifolds. This plays a key role in the following lemma.

\begin{lem}\label{l:simpli1}
There is a Hamiltonian diffeomorphism $\varphi \colon \CP^{2} \to \CP^{2}$ that fixes $L_{1, 1}$ and such that $\varphi(L)$ and $\varphi(B)$ are disjoint from the planes $u$, $v$, and $w$.
\end{lem}

\begin{proof}
The idea of the proof is similar to that of Proposition \ref{p:displaceclosed} in light of Remark \ref{r:splitting}. Fix a parameterization $\psi$ of $L$ in addition to $\psi_{1, 1}$ above. Then let $J$ be a tame almost complex structure on $\CP^{2} \sd (L \cup L_{1, 1} \cup B)$ adapted to $\psi$ and $\psi_{1, 1}$ and of appropriate form for applying the splitting procedure along $\partial{B}$. Let $\{J_{t}\}_{t \geq 0}$ be a family of tame almost complex structures on $\CP^{2} \sd L_{1, 1}$ that stretches to $J$ such that $J_{t}$ is adapted to $\psi_{1, 1}$ for each $t$. Now deform $J_{1, 1}$ to $J_{t}$ for each $t$ through almost complex structures that are adapted to $\psi_{1, 1}$. By compactness, the moduli spaces of $u$, $v$, and $w$ will persist for each $t$. Applying the more general splitting procedure for symplectic manifolds with cylindrical ends mentioned in Remark \ref{r:splitting}, we obtain moduli spaces of $J$-holomorphic limit planes in $\CP^{2} \sd (L \cup L_{1, 1} \cup B)$. Then extending $J$ smoothly past $L$ and $B$ and away from the above moduli spaces, we can deform it back to $J_{1, 1}$ through almost complex structures that are all adapted to $\psi_{1, 1}$. Finally, we can find a Hamiltonian isotopy in $\CP^{2}$ that fixes $L_{1, 1}$ and displaces $L$ and $B$ from the planes $u$, $v$, and $w$. Its time-one map is the desired Hamiltonian diffeomorphism.
\end{proof}

In view of the previous lemma, there is no loss of generality in assuming that we are in the situation depicted schematically in the following figure:
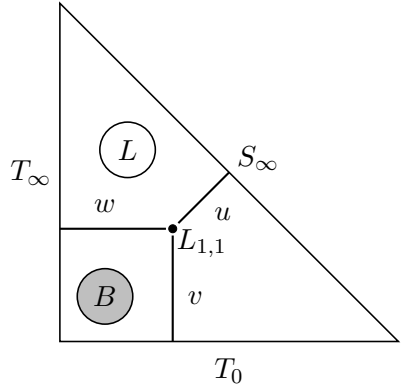
\[
\begin{tikzpicture}[scale = 1.5]
\draw[semithick] (0, 0) -- (3, 0) -- (0, 3) -- cycle;
\draw[thick] (1.05, 1.05) -- (1.5, 1.5);
\draw[thick] (0, 1) -- (0.94, 1);
\draw[thick] (1, 0) -- (1, 0.94);
\node[xshift = 1em, yshift = 0.5em] at (1.5, 1.5) {$S_{\infty}$};
\node[yshift = -1em] at (1.5, 0) {$T_{0}$};
\node[xshift = -1em] at (0, 1.5) {$T_{\infty}$};
\node[xshift = 0.75em, yshift = -0.45em] at (1.25, 1.25) {$u$};
\node[yshift = -0.45em] at (1.2, 0.5) {$v$};
\node at (0.4, 1.2) {$w$};
\filldraw[black] (1, 1) circle (1pt);
\filldraw[color = black, fill = gray!50, semithick] (0.4, 0.4) circle (7pt);
\node at (0.4, 0.4) {$B$};
\filldraw[color = black, fill = white, semithick] (0.6, 1.7) circle (7pt);
\node at (0.6, 1.7) {$L$};
\node[xshift = 1em, yshift = -0.5em] at (1, 1) {$L_{1, 1}$};
\end{tikzpicture}
\]
This is evidently meant to be suggestive, as $B$ and $L$, while being disjoint from $u$, $v$, and $w$, could still intersect any fiber lying above one of the three lines in the above moment image.

\begin{lem}\label{l:simpli2}
There exists a symplectically embedded sphere in the class of $S_{\infty}$ that is disjoint from $L$, $L_{1, 1}$ and $B$. Furthermore, this sphere can be arranged to have a unique transverse and positive intersection with exactly one of the planes $u$, $v$, or $w$.
\end{lem}

\begin{proof}
We will construct the desired sphere by examining the situation more closely in $\mathcal{U}(L_{1, 1})$. For concreteness, our proof will show how to construct a sphere having a single intersection with $u$, but the arguments adapt in the obvious way to the other two cases. Now by definition of $\psi_{1, 1}$, the projection of $\mathcal{U}(L_{1, 1})$ onto the $(p_{1}, p_{2})$-plane is a small square in the moment image of $(\CP^{2}, \omega_{\FS})$, where the origin corresponds to $L_{1, 1}$. We can then consider the product torus $L_{a, b}$ for $(a, b)$ close enough to $(1, 1)$, and translate the parametrization $\psi_{1, 1}$ in the obvious way to obtain a parameterization $\psi_{a, b}$ of $L_{a, b}$. We assume the corresponding Weinstein neighborhood $\mathcal{U}(L_{a, b})$ is made small enough to still be contained in $\mathcal{U}(L_{1, 1})$.

We then choose an almost complex structure $J_{a, b}$ on $\CP^{2} \sd L_{a, b}$ that coincides with $J$ outside of $\mathcal{U}(L_{1, 1})$ and that is adapted to $\psi_{a, b}$, and translate the coordinate representations of the planes $v$ and $w$ to obtain symplectic cylinders that are $J_{a, b}$-holomorphic in $\mathcal{U}(L_{a, b})$. We also translate the entire plane $u$ by a toric translation using the vector $(a, b)$, obtaining a symplectic plane that is $J_{a, b}$-holomorphic in $\mathcal{U}(L_{a, b})$. This process is illustrated in the following figure for a specific translation:
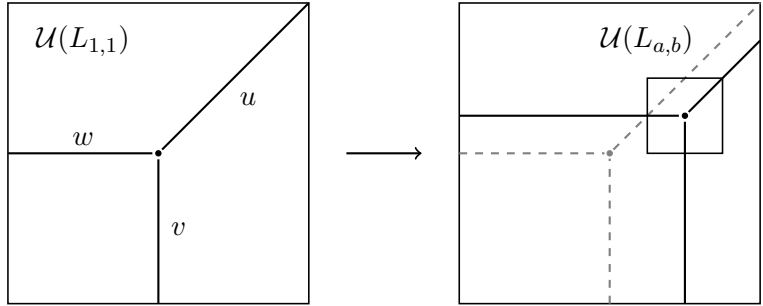
\[
\begin{tikzpicture}
\draw[semithick] (0, 0) -- (0, 4) -- (4, 4) -- (4, 0) -- cycle;
\draw[thick] (0, 2) -- (1.92, 2);
\draw[thick] (2, 0) -- (2, 1.92);
\draw[thick] (2.05, 2.05) -- (4, 4);
\node[xshift = 0.5em, yshift = -0.7em] at (3, 3) {$u$};
\node[xshift = 0.7em] at (2, 1) {$v$};
\node[yshift = 0.5em] at (1, 2) {$w$};
\node at (1, 3.5) {$\mathcal{U}(L_{1, 1})$};
\filldraw[black] (2, 2) circle (1pt);
\draw[thick, ->] (4.5, 2) -- (5.5, 2);
\draw[semithick] (6, 0) -- (6, 4) -- (10, 4) -- (10, 0) -- cycle;
\draw[thick, gray, dashed] (6, 2) -- (7.93, 2);
\draw[thick, gray, dashed] (8, 0) -- (8, 1.93);
\draw[thick, gray, dashed] (8.05, 2.05) -- (10, 4);
\draw[thick] (6, 2.5) -- (8.92, 2.5);
\draw[thick] (9, 0) -- (9, 2.42);
\draw[thick] (9.05, 2.55) -- (10, 3.5);
\draw[semithick] (8.5, 2) -- (9.5, 2) -- (9.5, 3) -- (8.5, 3) -- cycle;
\node at (8.5, 3.5) {$\mathcal{U}(L_{a, b})$};
\filldraw[gray] (8, 2) circle (1pt);
\filldraw[black] (9, 2.5) circle (1pt);
\node at (2, 0) {\phantom{$T_{0}$}};
\node at (2, 4) {\phantom{$T_{0}$}};
\end{tikzpicture}
\]
We then construct symplectic planes in $\CP^{2} \sd L_{a, b}$ that are $J_{a, b}$-holomorphic in $\mathcal{U}(L_{a, b})$ and asymptotic to Reeb orbits covering geodesics in the classes $(1, 1)$, $(0, -1)$ and $(-1, 0)$ in $H_{1}(L_{a, b}, \psi_{a, b})$. This is done via the following toric deformation:
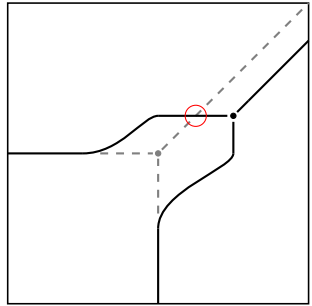
\[
\begin{tikzpicture}
\draw[semithick] (0, 0) -- (0, 4) -- (4, 4) -- (4, 0) -- cycle;
\draw[thick, gray, dashed] (1.23, 2) -- (1.93, 2);
\draw[thick, gray, dashed] (2, 1.2) -- (2, 1.93);
\draw[thick, gray, dashed] (2.05, 2.05) -- (4, 4);
\draw[thick] (2, 2.5) -- (2.92, 2.5);
\draw[thick] (3, 2) -- (3, 2.42);
\draw[thick] (3.05, 2.55) -- (3.2, 2.7);
\draw[thick] (0, 2) -- (1, 2);
\draw[thick] (2, 0) -- (2, 1);
\draw[thick] (1, 2) .. controls (1.5, 2) and (1.8, 2.5) .. (2, 2.5);
\draw[thick] (2, 1) .. controls (2, 1.5) and (3, 1.8) .. (3, 2);
\draw[thick] (3.2, 2.7) -- (4, 3.5);
\filldraw[gray] (2, 2) circle (1pt);
\filldraw[black] (3, 2.5) circle (1pt);
\draw[red] (2.5, 2.5) circle (4pt);
\node at (2, 0) {\phantom{$T_{0}$}};
\node at (2, 4) {\phantom{$T_{0}$}};
\end{tikzpicture}
\]
We have also circled a distinguished intersection point in the above figure.

We now consider the symplectically embedded surface with three cylindrical ends in $T^{\ast}\TT^{2}$ whose existence is guaranteed by Proposition \ref{p:3sphere}. This can be transported to $T^{\ast}L_{a, b}$ via the embedding $\psi_{a, b}$, and using the fact that all Reeb orbits in question match, we can glue this surface together with the above perturbed objects in order to construct a symplectically embedded sphere $\mathcal{S}_{\infty}$ in $\CP^{2} \sd (L \cup L_{1, 1} \cup B)$, which from the intersection properties of the above planes with $S_{\infty} \cup T_{0} \cup T_{\infty}$ must belong to the class of $S_{\infty}$. By construction, we see that $\mathcal{S}_{\infty}$ has a single positive and transverse intersection with $u$, as desired.
\end{proof}

At this point, we can actually employ the machinery of relative foliations. To start, fix an almost complex structure $J$ on $\CP^{2} \sd (L \cup L_{1, 1})$ that is adapted to $\psi_{1, 1}$ and some fixed parameterization of $L$, and that is integrable near the planes $u$, $v$, and $w$. Then perturb $J$ near $B$ so that it it is of a suitable form to perform a symplectic blowup along the interior of $B$. Once this is done, we are back in $(\DD_{1}, \omega_{\mspace{1mu}\DD_{1}})$, and $J$ lifts to an almost complex structure on $\DD_{1}$ that is integrable near the exceptional divisor $\mathcal{S}_{0}$ and makes the planes $u$, $v$, and $w$ into $J$-holomorphic planes with punctures on $L_{1, 1}$. There is a corresponding relative foliation $\mathcal{F}$, and we choose two unbroken leaves $\mathcal{T}_{0}$ and $\mathcal{T}_{\infty}$ with the property that the points $\pi(\mathcal{T}_{0})$ and $\pi(\mathcal{T}_{\infty})$ lie on opposite sides of each of the circles $\pi(L)$ and $\pi(L_{1, 1})$ in $\mathcal{S}_{0}$.

\begin{lem}\label{l:simpli3}
There exists a symplectically embedded sphere $\mathcal{S}_{\infty}$ in the class $(1, 1)$ that is disjoint from $L$ and $L_{1, 1}$, such that $L_{1, 1}$ is homologicaly nontrivial in $\DD_{1} \sd (\mathcal{S}_{0}\cup \mathcal{S}_{\infty} \cup \mathcal{T}_{0} \cup \mathcal{T}_{\infty})$.
\end{lem}

\begin{proof}
We first show that one of the planes $u$, $v$, and $w$ must be part of a broken leaf of $\mathcal{F}$. To this end, observe that these planes compactify to smooth disks with boundary on $L_{1, 1}$, and the homology classes of the corresponding boundaries in $H_{1}(L_{1, 1}; \psi_{1, 1})$ sum to $0$. Adding to this a chain in $L_{1, 1}$ with matching boundary, we obtain a chain in $\DD_{1}$ that is smoothly embedded and $J$-holomorphic outside of any arbitrarily small neighborhood of $L_{1, 1}$. Since this chain has area $3$ and does not intersect $\mathcal{S}_{0}$, it must represent the class $(1, 1)$. It is then a simple matter to see that the conclusions of Lemma \ref{l:buildingtop} hold for the planes $u$, $v$, and $w$. In particular, they cannot all be essential, as this would contradict the last part of the aforementioned lemma. It thus follows that one of them must be part of a broken leaf of $\mathcal{F}$. Without loss of generality, we can assume this plane to be $u$. We then use Lemma \ref{l:simpli2} to produce a symplectically embedded sphere $\mathcal{S}_{\infty}$ in $\CP^{2}$ in the class of $S_{\infty}$ that is disjoint from $L$, $L_{1, 1}$ and $B$, and that has a unique transverse and positive intersection with $u$. We can also modify $J$ so that it is integrable near $\mathcal{S}_{\infty}$ Blowing up along $B$ returns us to $(\DD_{1}, \omega_{\mspace{1mu}\DD_{1}})$, and it is now straightforward to see that the hypotheses of Lemma \ref{l:homtriv} are satisfied.
\end{proof}

We can now finally prove Proposition \ref{p:simpli}. First, we claim that the unbroken leaves $\mathcal{T}_{0}$ and $\mathcal{T}_{\infty}$ represent different classes in $H_{2}(\DD_{1} \sd L)$ and in $H_{2}(\DD_{1} \sd L_{1, 1})$. If this failed for say $L$, there would be a homotopy in $X \sd L$ from $\mathcal{T}_{0}$ to $\mathcal{T}_{\infty}$, and this would project to a connected subset of $\mathcal{S}_{0} \sd \pi(L)$ containing $\pi(\mathcal{T}_{0})$ and $\pi(\mathcal{T}_{\infty})$, contradicting the assumption that $\pi(L)$ separates these two points. Now by Theorem 1.4 of \cite{Mc1}, there is a symplectomorphism from the current model of $(\DD_{1}, \omega_{\mspace{1mu}\DD_{1}})$ to the toric model $(X, \omega_{X})$ taking $\mathcal{S}_{0}$ to $S_{0}$ and $\mathcal{S}_{\infty}$ to $S_{\infty}$. Composing this with a Hamiltonian diffeomorphism if necessary, we can assume that it also takes $\mathcal{T}_{0}$ to $T_{0}$ and $\mathcal{T}_{\infty}$ to $T_{\infty}$. Let $L_{\mathfrak{r}}$ and $L_{\mathfrak{s}}$ be the images of $L$ and $L_{1, 1}$, respectively, under the above symplectomorphism. These are now disjoint, monotone tori in $X \sd Y$, and (c) of Proposition \ref{p:simpli} is automatically satisfied, as is (b) in view of Lemma \ref{l:simpli3}. To see that (a) holds, simply observe that if $L_{\mathfrak{r}}$ were homologicaly nontrivial in $X \sd Y$, then Proposition \ref{p:exact} would imply that it must intersect $L_{\mathfrak{s}}$, contradicting our standing assumption. This finally completes the proof of Proposition \ref{p:simpli}.

\subsection{Building Types}

In view of the previous subsection, to prove Theorem \ref{t:main} it is enough to assume the conclusion of Proposition \ref{p:simpli} and derive a contradiction. So taking $L_{\mathfrak{r}}$ and $L_{\mathfrak{s}}$ as above, we fix parameterizations $\psi_{\mathfrak{r}}$ and $\psi_{\mathfrak{s}}$, and shrink the corresponding Weinstein neighborhoods so that they are disjoint from each other and from $Y$. We then choose an almost complex structure $J$ on $X \sd (L_{\mathfrak{r}} \cup L_{\mathfrak{s}})$ that is integrable near $Y$ and adapted to the parameterizations $\psi_{\mathfrak{r}}$ and $\psi_{\mathfrak{s}}$. From Theorem \ref{t:foliation} we obtain a relative foliation
\[
\mathcal{F} = \mathcal{F}(L_{\mathfrak{r}}, L_{\mathfrak{s}}, \psi_{\mathfrak{r}}, \psi_{\mathfrak{s}}, J),
\]
together with foliation classes $\beta_{\mathfrak{r}}$ and $\beta_{\mathfrak{s}}$, and a projection map $\pi \colon X \to S_{0}$ defined as before. The sets $C_{\mathfrak{r}} = \pi(L_{\mathfrak{r}})$ and $C_{\mathfrak{s}} = \pi(L_{\mathfrak{s}})$ are disjoint embedded circles in $S_{0}$, By (c) of Proposition \ref{p:simpli}, both of these must separate the points $\pi(T_{0})$ and $\pi(T_{\infty})$. Consequently, we can find disjoint closed disks $D_{\mathfrak{r}}$ and $D_{\mathfrak{s}}$ in $S_{0}$ with boundaries $C_{\mathfrak{r}}$ and $C_{\mathfrak{s}}$, respectively, such that $\pi(T_{0}) \in D_{\mathfrak{r}}$ and $\pi(T_{\infty}) \in D_{\mathfrak{s}}$. The closure of $S_{0} \sd (D_{\mathfrak{r}} \cup D_{\mathfrak{s}})$ is a closed annulus, which we will denote by $A$.

Observe that the leaves of $\mathcal{F}$ split up into three categories. There are unbroken leaves in the class $(0, 1)$, and two categories of broken leaves. The first consists of pairs of planes that are asymptotic to Reeb orbits in $\Gamma_{\beta_{\mathfrak{s}}} \cup \Gamma_{-\beta_{\mathfrak{s}}}$. We know that the totality of these planes divide into two families, which going forward we shall denote by $\mathfrak{s}_{0}$ and $\mathfrak{s}_{\infty}$. The family $\mathfrak{s}_{0}$ consists of those planes that intersect $S_{0}$, and the family $\mathfrak{s}_{\infty}$ consists of those planes that intersect $S_{\infty}$. This follows from the fact that $L_{\mathfrak{s}}$ is homologicaly nontrivial in $X \sd Y$, since otherwise the union of the planes in $\mathfrak{s}_{\infty}$ would compactify to a solid torus in $X \sd Y$ whose boundary is $L_{\mathfrak{s}}$, see Proposition 5.16 of \cite{DRGI}. The second category of broken leaves consists of pairs of planes asymptotic to Reeb orbits in $\Gamma_{\beta_{\mathfrak{r}}} \cup \Gamma_{-\beta_{\mathfrak{r}}}$. These again divide into two families $\mathfrak{r}_{0}$ and $\mathfrak{r}_{\infty}$. The family $\mathfrak{r}_{0}$ consists of planes intersecting $S_{0}$, but the family $\mathfrak{r}_{\infty}$ now consists of planes intersecting \emph{neither $S_{0}$ nor $S_{\infty}$}. This follows from Lemma \ref{l:homtriv}, since $L_{\mathfrak{r}}$ is homologicaly trivial in $X \sd Y$ but $C_{\mathfrak{r}}$ still separates $\pi(T_{0})$ and $\pi(T_{\infty})$. Consequently, the planes in $\mathfrak{r}_{0}$ intersect \emph{both $S_{0}$ and $S_{\infty}$}. This fact will have important ramifications.

Now consider a family $\{J_{t}\}_{t \geq 0}$ of tame almost complex structures on $(X, \omega_{X})$ that stretches to $J$, and extract a sequence $(J_{t_{i}})_{i=1}^{\infty}$ as usual. We then fix $d \geq 1$ and use (c) of Proposition \ref{p:blowupcurves} to obtain a sequence $(u_{i})_{i=1}^{\infty}$ of $J_{t_{i}}$-holomorphic spheres in the class $(1, d)$ passing through $2d$ generic point constraints. By Theorem \ref{t:compactness}, this produces a limit building $\mathbf{F}$ in the class $(1, d)$. Using the foliation $\mathcal{F}$, we can obtain a complete classification scheme for such buildings, and this is content of the following proposition.

\begin{prp}\label{p:types}
The building $\mathbf{F}$ is exactly one of five types:

\begin{t0}
$\mathbf{F}$ is a (possibly nodal) $J$-holomorphic sphere in $X \sd (L_{\mathfrak{r}} \cup L_{\mathfrak{s}})$ in the class $(1, d)$, with one essential sphere in the class $(1, c)$ for $1 \leq c \leq d$. Any remaining top level curves cover broken or unbroken leaves of $\mathcal{F}$. Any middle and bottom level curves cover cylinders asymptotic to orbits covering geodesics in multiples of $\beta_{\mathfrak{r}}$ and $\beta_{\mathfrak{s}}$.
\end{t0}

\begin{t1}
$\mathbf{F}$ has a unique essential curve $\underline{u}$ whose punctures on both $L_{\mathfrak{r}}$ and $L_{\mathfrak{s}}$ are all of foliation type. The image of $\pi \circ \underline{u}$ is $S_{0}$ minus finitely many points on $C_{\mathfrak{r}} \cup C_{\mathfrak{s}}$. Any remaining top level curves are either spheres covering unbroken leaves, or planes covering one of the components in a broken leaf of $\mathcal{F}$. Any middle and bottom level curves cover cylinders asymptotic to orbits covering geodesics in multiples of $\beta_{\mathfrak{r}}$ and $\beta_{\mathfrak{s}}$.
\end{t1}

\begin{t2a}
$\mathbf{F}$ has two essential curves. The first curve $u_{\mathfrak{r}}$ is a plane with a single puncture on $L_{\mathfrak{r}}$, not of foliation type. The closure of the image of $\pi \circ u_{\mathfrak{r}}$ is $D_{\mathfrak{r}}$. The second curve $\underline{u}$ has a single puncture on $L_{\mathfrak{r}}$, not of foliation type, and punctures on $L_{\mathfrak{s}}$, all of foliation type. The closure of the image of $\pi \circ \underline{u}$ is $A \cup D_{\mathfrak{s}}$. Any remaining top level curves cover broken or unbroken leaves of $\mathcal{F}$. Any middle and bottom level curves in $\RR \times S^{\ast}L_{\mathfrak{s}}$ and $T^{\ast}L_{\mathfrak{s}}$ cover cylinders asymptotic to orbits covering geodesics in multiples of $\beta_{\mathfrak{s}}$.
\end{t2a}

\begin{t2b}
$\mathbf{F}$ has two essential curves. The first curve $\underline{u}$ has a single puncture on $L_{\mathfrak{s}}$, not of foliation type, and punctures on $L_{\mathfrak{r}}$, all of foliation type. The closure of the image of $\pi \circ \underline{u}$ is $D_{\mathfrak{r}} \cup A$. The second curve $u_{\mathfrak{s}}$ is a plane with a single puncture on $L_{\mathfrak{s}}$, not of foliation type. The closure of the image of \smash{$\pi \circ u_{\mathfrak{s}}$} is $D_{\mathfrak{s}}$. Any remaining top level curves cover broken or unbroken leaves of $\mathcal{F}$. Any middle and bottom level curves in $\RR \times S^{\ast}L_{\mathfrak{r}}$ and $T^{\ast}L_{\mathfrak{r}}$ cover cylinders asymptotic to orbits covering geodesics in multiples of $\beta_{\mathfrak{r}}$.
\end{t2b}

\begin{t3}
$\mathbf{F}$ has three essential curves. The first curve $u_{\mathfrak{r}}$ is a plane with a single puncture on $L_{\mathfrak{r}}$, not of foliation type, and the closure of the image of $\pi \circ u_{\mathfrak{r}}$ is $D_{\mathfrak{r}}$. The second curve $\underline{u}$ is a cylinder with a single puncture on each of $L_{\mathfrak{r}}$ and $L_{\mathfrak{s}}$, neither of foliation type, and the closure of the image of $\pi \circ \underline{u}$ is $A$. The third curve \smash{$u_{\mathfrak{s}}$} is a plane with a single puncture on $L_{\mathfrak{s}}$, not of foliation type, and the closure of the image of $\pi \circ u_{\mathfrak{s}}$ is $D_{\mathfrak{s}}$. Any remaining top level curves cover broken or unbroken leaves of $\mathcal{F}$.
\end{t3}
\end{prp}

\begin{proof}
By Lemma \ref{l:buildingtop}, we know that $\mathbf{F}$ has at least one essential top level curve, which we denote by $u$. As a preliminary step, we observe that if $u$ has punctures and none of them are of foliation type, then it falls into exactly one of three categories:
\begin{enumerate}
\item[$\bullet$] A $J$-holomorphic plane with a single puncture on $L_{\mathfrak{r}}$. The closure of the image of $\pi \circ u$ is $D_{\mathfrak{r}}$ or $A \cup D_{\mathfrak{s}}$.
\item[$\bullet$] A $J$-holomorphic plane with a single puncture on $L_{\mathfrak{s}}$. The closure of the image of $\pi \circ u$ is $D_{\mathfrak{r}} \cup A$ or $D_{\mathfrak{s}}$.
\item[$\bullet$] A $J$-holomorphic cylinder with a single puncture on each of $L_{\mathfrak{r}}$ and $L_{\mathfrak{s}}$. The closure of the image of $\pi \circ u$ is $A$.
\end{enumerate}
Indeed, by Proposition \ref{p:essential}, we see that $u$ can have at most one puncture on either of $L_{\mathfrak{r}}$ or $L_{\mathfrak{s}}$, and this shows that it is either a plane or a cylinder. Observing that the image of $u$ must project via $\pi$ onto a connected subset of $S_{0}$ with boundary components equal to one or both of $C_{\mathfrak{r}}$ or $C_{\mathfrak{s}}$, the three possibilities for $u$ are easily deduced.

Now suppose that $u$ intersects the fiber $T_{0}$, itself an unbroken leaf of $\mathcal{F}$ since we are assuming it to be $J$-holomorphic. We can then conduct a case by case analysis based on the nature of the punctures of $u$.

\vspace{2mm}

\noindent\textsc{Case 1:} \emph{$u$ has no punctures.} This implies that $\pi \circ u$ must be a bijection onto $S_{0}$. It follows that $u$ is a $J$-holomorphic sphere in the class $(1, c)$ for some $0 \leq c \leq d$. Then again by Lemma \ref{l:buildingtop}, all remaining top level curves of $\mathbf{F}$ must cover leaves of $\mathcal{F}$, and all middle and bottom level curves cover cylinders asymptotic to orbits covering geodesics in multiples of $\beta_{\mathfrak{r}}$ and $\beta_{\mathfrak{s}}$. Furthermore, the top level curves that cover leaves fit together to form a possibly disconnected curve in the class $(0, d - c)$, and if $c = d$, then $\mathbf{F}$ consists only of the curve $u$. In any case, the building is of type 0.

\vspace{2mm}

\noindent\textsc{Case 2:} \emph{$u$ has punctures and they are all of foliation type.} In this case, we will show that $\mathbf{F}$ must be of type 1. First, by Lemma \ref{l:punctures}, we see that the image of $\pi \circ u$ includes points in each component of $S_{0} \sd C_{\mathfrak{r}} \cup C_{\mathfrak{s}}$. By Lemma \ref{l:buildingtop}, it follows that $\pi \circ u$ is actually a bijection onto $S_{0}$ minus a finite set of points on $C_{\mathfrak{r}} \cup C_{\mathfrak{s}}$. Any remaining top level curves must then cover unbroken leaves of $\mathcal{F}$ or be planes covering one of the components of a broken leaf. This shows that $\mathbf{F}$ is of type 1 with $\underline{u} = u$.

\vspace{2mm}

\noindent\textsc{Case 3:} \emph{$u$ has at least one puncture that is not of foliation type.} Since $u$ intersects $T_{0}$, the closure of the image of $\pi \circ u$ is either $D_{\mathfrak{r}}$ or $D_{\mathfrak{r}} \cup A$. In either case, we see by Proposition \ref{p:essential} that $u$ does not intersect $T_{\infty}$ and has exactly one puncture that is not of foliation type. If we suppose furthermore that the closure of the above image is $D_{\mathfrak{r}}$, then $u$ is a plane, for it cannot have any punctures on $L_{\mathfrak{s}}$. By Lemma \ref{l:buildingtop}, there is another essential curve $v$ of $\mathbf{F}$ that intersects $T_{\infty}$, and the images of $\pi \circ u$ and $\pi \circ v$ must be disjoint. This implies that the closure of the image of the latter is either $D_{\mathfrak{s}}$ or $A \cup D_{\mathfrak{s}}$. If it is $D_{\mathfrak{s}}$, then $v$ cannot have any punctures on $L_{\mathfrak{r}}$ and must therefore be a plane with a single puncture on $L_{\mathfrak{s}}$, not of foliation type. Lemma \ref{l:buildingtop} then guarantees the existence of yet a third essential curve $\underline{u}$, which must therefore be a cylinder with a puncture on either torus, both not of foliation type. The building $\mathbf{F}$ is thus of type 3 with $u_{\mathfrak{r}} = u$ and $u_{\mathfrak{s}} = v$. If the closure of the image of $\pi \circ v$ was instead $A \cup D_{\mathfrak{s}}$, then all its punctures on $L_{\mathfrak{s}}$ must be of foliation type, and it must have a single puncture on $L_{\mathfrak{r}}$, not of foliation type. The building $\mathbf{F}$ is now of type 2a, with $u_{\mathfrak{r}} = u$ and $\underline{u} = v$. Finally, if initially the closure of the image of $\pi \circ u$ was $D_{\mathfrak{r}} \cup A$, then an entirely analogous argument implies that $\mathbf{F}$ is of type 2b.
\end{proof}

\subsection{Existence of Special Buildings}

In this subsection, we derive the existence of special buildings in the class $(1, d)$ for $d > 0$ large enough. We then state the black box result mentioned at the beginning of this section, and with it we will be able to complete the proof of Theorem \ref{t:main} in the following section. We begin modestly with the following lemma, which will be of use in the proof of Proposition \ref{p:existence}.

\begin{lem}\label{l:area}
There is some $\epsilon > 0$ such that
\[
\omega_{X}(\overline{u}) \geq \epsilon(\overline{u} \cdot (S_{0} \cup S_{\infty}))
\]
for every (closed or punctured) $J$-holomorphic sphere $u$ in $X \sd (L_{\mathfrak{r}} \cup L_{\mathfrak{s}})$.
\end{lem}

\begin{proof}
Let $\Sigma$ be the compactified domain of $\overline{u}$, and consider closed tubular neighborhoods $\mathcal{N}(S_{0})$ and $\mathcal{N}(S_{\infty})$ of $S_{0}$ and $S_{\infty}$, respectively, that are disjoint from $L_{\mathfrak{r}} \cup L_{\mathfrak{s}}$. By appropriately perturbing these neighborhoods, we may assume that their preimages under $u$ are compact submanifolds with boundary in $\Sigma$, and since they are disjoint from the compactified punctures we may identify them with regular domains in $\RR^{2}$. It then follows that
\[
\omega_{X}(\overline{u}) \geq \int_{u^{-1}(\mathcal{N}(S_{0}))}u^{\ast}\omega_{X} + \int_{u^{-1}(\mathcal{N}(S_{\infty}))}u^{\ast}\omega_{X}.
\]
Now choose a Thom class $\tau$ for the interior of $\mathcal{N}(S_{0})$. It is a standard fact of differential topology that $\tau$ is Poincar\'{e} dual to $S_{0}$. Then choose $\epsilon > 0$ small enough so that the form $\omega = \omega_{X} - \epsilon\tau$ is a symplectic form on the interior of $\mathcal{N}(S_{0})$ that tames $J$. This is possible since the properties of being nondegenerate and taming a fixed almost complex structure are open conditions. Setting $Q = u^{-1}(\mathcal{N}(S_{0}))$, it now follows that
\[
\int_{Q}u^{\ast}\omega_{X} = \int_{Q}u^{\ast}\omega + \epsilon\int_{Q}u^{\ast}\tau \geq \epsilon\int_{Q}u^{\ast}\tau.
\]
Note that $Q$ is compact and that no points on its boundary are mapped by $u$ to intersections with $S_{0}$. After possibly perturbing $u$ to be transverse to $S_{0}$ while keeping it fixed on the boundary, and perhaps shrinking $\mathcal{N}(S_{0})$, we can appeal to another standard fact from differential topology, which says that the last integral on the right is exactly $\epsilon(u \cdot S_{0})$. Running through the same argument for $S_{\infty}$, we then obtain the desired result.
\end{proof}

We can now state and prove the main existence result of this subsection.

\begin{prp}\label{p:existence}
For $d > 0$ sufficiently large, there exists a building $\mathbf{F}$ of type 3 in the class $(1, d)$ whose top level curves are all of area $1$ and consist of three essential curves $u_{\mathfrak{r}}$, $\underline{u}$, and $u_{\mathfrak{s}}$, together with $d - 1$ planes in $\mathfrak{r}_{0} \cup \mathfrak{r}_{\infty}$ and $d - 1$ planes in $\mathfrak{s}_{0} \cup \mathfrak{s}_{\infty}$.
\end{prp}

\begin{proof}
Fix $d$ generic points on each of $L_{\mathfrak{r}}$ and $L_{\mathfrak{s}}$. By (c) of Proposition \ref{p:blowupcurves} and Theorem \ref{t:compactness}, we obtain a limit building $\mathbf{F}$ passing through the $2d$ point constraints. We will show that this building is of type 3 by ruling out all other types. In the process, we will deduce that the top level curves all have area $1$, and in particular are simply covered.

First, observe that $\mathbf{F}$ must have curves in its bottom level because the point constraints lie on $L_{\mathfrak{r}}$ and $L_{\mathfrak{s}}$, and so can only be satisfied by curves in $T^{\ast}L_{\mathfrak{r}} \cup T^{\ast}L_{\mathfrak{s}}$. Partition the middle and bottom level curves of $\mathbf{F}$ in $\RR \times S^{\ast}L_{\mathfrak{r}}$ and $T^{\ast}L_{\mathfrak{r}}$ into subcollections $\mathbf{F}_{\mathfrak{r}, 1}, \mathbf{F}_{\mathfrak{r}, 2}, \dots, \mathbf{F}_{\mathfrak{r}, k}$, where two curves are in the same subcollection if they have matching nodes or asymptotic orbits. Each of these subcollections is a \emph{subbuilding}, which following \cite{HO} has a well defined index given by
\[
\ind(\mathbf{F}_{\mathfrak{r}, i}) = 2s_{i} - 2.
\]
Here $s_{i}$ is the total number of asymptotic orbits of $\mathbf{F}_{\mathfrak{r}, i}$ matching with top level curves of $\mathbf{F}$. If $d_{i}$ of the $d$ point constraints on $L_{\mathfrak{r}}$ are satisfied by curves in $\mathbf{F}_{\mathfrak{r}, i}$, and these are chosen generically, then $\ind(\mathbf{F}_{\mathfrak{r}, i}) \geq 2d_{i}$. From this we deduce that $s_{i} \geq d_{i} + 1$, and therefore that the total number of punctures of top level curves on $L_{\mathfrak{r}}$ is at least $d + 1$ (in fact, at least $d + k$). A similar analysis shows that the total number of punctures of top level curves on $L_{\mathfrak{s}}$ is again at least $d + 1$.

Next, we see right away that $\mathbf{F}$ is not of type 0, since it would need to have at least $d + 1$ top level curves covering planes in $\mathfrak{r}_{0} \cup \mathfrak{r}_{\infty}$ and $d + 1$ top level curves covering planes in $\mathfrak{s}_{0} \cup \mathfrak{s}_{\infty}$. But these all have integral area at least $1$, which already contradicts the fact that $\mathbf{F}$ itself has area $2d + 1$. On the other hand, if $\mathbf{F}$ is of any other type, then it has an essential curve $\underline{u}$ that has punctures on both $L_{\mathfrak{r}}$ and $L_{\mathfrak{s}}$. In this case it follows from genus considerations that the number $k$ of subbuildings in the previous paragraph is at least the number of punctures on $\underline{u}$ on $L_{\mathfrak{r}}$. By going through each case and balancing areas, it is easy to see that $\mathbf{F}$ will always have at least $d - 1$ curves covering planes in $\mathfrak{r}_{0} \cup \mathfrak{r}_{\infty}$ and $d - 1$ curves covering planes in $\mathfrak{s}_{0} \cup \mathfrak{s}_{\infty}$. It also follows that all top level curves have area $1$ and are therefore simply covered. Furthermore, if $\underline{u}$ has punctures on $L_{\mathfrak{r}}$ all of foliation type, then $\mathbf{F}$ has exactly $d$ planes in $\mathfrak{r}_{0} \cup \mathfrak{r}_{\infty}$, and a simple argument using areas shows that these belong to exactly one of $\mathfrak{r}_{0}$ or $\mathfrak{r}_{\infty}$ according to whether the aforementioned punctures of $\underline{u}$ on $L_{\mathfrak{r}}$ are negative or positive, respectively. The same is true if $L_{\mathfrak{r}}$ is replaced by $L_{\mathfrak{s}}$.

Now if $\mathbf{F}$ were of type 1 or 2a, then the essential curve $\underline{u}$ has punctures on $L_{\mathfrak{s}}$ all of foliation type, and thus $d$ planes in one of $\mathfrak{s}_{0}$ or $\mathfrak{s}_{\infty}$. Since $\mathbf{F}$ has only $d - 1$ intersections with $S_{0}$, these must belong to $\mathfrak{s}_{\infty}$ and also account for all intersections of $\mathbf{F}$ with $S_{\infty}$. All nonessential top level curves of $\mathbf{F}$ with punctures on $L_{\mathfrak{r}}$ must then be planes in $\mathfrak{r}_{\infty}$, in order to not contribute any more intersections with $S_{\infty}$. This leaves $d - 1$ intersections with $S_{0}$ that must be contributed by the essential curves. On the other hand, if $\mathbf{F}$ were of type 2b, then it has exactly $d$ top level planes belonging to $\mathfrak{r}_{\infty}$, again since $\mathbf{F}$ has only $d - 1$ intersections with $S_{0}$. It also has $d - 1$ top level planes in $\mathfrak{s}_{0} \cup \mathfrak{s}_{\infty}$, which contribute a total of $d - 1$ intersections with $S_{0} \cup S_{\infty}$. This leaves $d$ intersections that must be contributed by the essential curves of $\mathbf{F}$. All in all, when $\mathbf{F}$ is of type 1, 2a, or 2b, there at least $d - 1$ intersections with $S_{0} \cup S_{\infty}$ which must be contributed by the essential curves of $\mathbf{F}$. Since these have total area at most $2$, this contradicts Lemma \ref{l:area} if $d$ is chosen sufficiently large, and so $\mathbf{F}$ is of type 3 as claimed.
\end{proof}

Finally, we present the aforementioned black box result that we will need to complete the proof of Theorem \ref{t:main} in the next section.

\begin{prp}\label{p:black}
For $d > 0$ sufficiently large, there exist symplectically embedded spheres
\[
F \colon \SS^{2} \to X \sd (L_{\mathfrak{r}} \cup L_{\mathfrak{s}}), \qquad G \colon \SS^{2} \to X \sd (L_{\mathfrak{r}} \cup L_{\mathfrak{s}}),
\]
in the classes $(1, d)$ and $(1, d-1)$, respectively, together with symplectically embedded disks
\[
E_{\mathfrak{r}} \colon (\DD^{2}, \SS^{1}) \to (X, L_{\mathfrak{r}}), \qquad E_{\mathfrak{s}} \colon (\DD^{2}, \SS^{1}) \to (X, L_{\mathfrak{s}}),
\]
disjoint and of Maslov index $2$, satisfying the following properties:
\begin{enumerate}
\item $F$, $G$, $E_{\mathfrak{r}}$, and $E_{\mathfrak{s}}$ are all $J$-holomorphic away from arbitrarily small neighborhoods of a collection of Lagrangian tori whose members are near to, and Lagrangian isotopic to, either $L_{\mathfrak{r}}$ or $L_{\mathfrak{s}}$.
\item The class of $E_{\mathfrak{r}}\vert_{\SS^{1}}$ and the class $\beta_{\mathfrak{r}}$ form an integral basis of $H_{1}(L_{\mathfrak{r}})$.
\item The class of $E_{\mathfrak{s}}\vert_{\SS^{1}}$ and the class $\beta_{\mathfrak{s}}$ form an integral basis of $H_{1}(L_{\mathfrak{s}})$.
\item Exactly one of $F$ and $G$ intersects the planes of $\mathfrak{r}_{0}$, and the other the planes of $\mathfrak{r}_{\infty}$.
\item Exactly one of $F$ and $G$ intersects the planes of $\mathfrak{s}_{0}$, and the other the planes of $\mathfrak{s}_{\infty}$.
\item $F \cdot E_{\mathfrak{r}} + G \cdot E_{\mathfrak{r}} = d - 1$, and these intersections are all positive.
\item $F \cdot E_{\mathfrak{s}} + G \cdot E_{\mathfrak{s}} = d - 1$, and these intersections are all positive.
\item $F$ and $G$ intersect in $2d - 2$ points where they are $J$-holomorphic, away from $E_{\mathfrak{r}} \cup E_{\mathfrak{s}}$ and from $T_{0} \cup T_{\infty}$.
\item $E_{\mathfrak{r}}$ intersects exactly $d - 1$ of the unbroken leaves over $\pi(F \cap G)$, each in a single point, and $E_{\mathfrak{r}}$ intersects the remaining $d - 1$ unbroken leaves over $\pi(F \cap G)$, each in a single point.
\end{enumerate}
\end{prp}

The proof of this proposition will be carried out in Section 7.

\section{Proof of Theorem \ref{t:main}}

In this section, we apply a sequence of blowups, inflations, and blowdowns to transform $(X, \omega_{X})$ into a new symplectic manifold in which the spheres and disks of Proposition \ref{p:black} can be used to contradict the conclusion of Proposition \ref{p:simpli}, thereby proving Theorem \ref{t:main}. This first requires some modification of the almost complex structure $J$ on $X \sd (L_{\mathfrak{r}} \cup L_{\mathfrak{s}})$. That being done, we perform the aforementioned sequence of transformations, keeping track of intersection numbers and areas of various objects along the way. Some concluding arguments then yield our contradiction.

\subsection{Intersection and Area Analysis}

As in the previous section, let $J$ be an almost complex structure that is adapted to parameterizations of $L_{\mathfrak{r}}$ and $L_{\mathfrak{s}}$. For the purposes of this section, it will be important to make the preliminary assumption that $J$ is \emph{compatible with $\omega_{X}$}, and not just tame. Consider the spheres $F$ and $G$ and the disks $E_{\mathfrak{r}}$ and $E_{\mathfrak{s}}$ of Proposition \ref{p:black}, and let $H$ be an unbroken leaf of $\mathcal{F}$ that is disjoint from $F \cap G$. For $i = 1, \dots, 2d - 2$, let $H_{i}$ be the unbroken leaf of $\mathcal{F}$ intersecting $p_{i} \in F \cap G$. By property (8) of Proposition \ref{p:black}, we know that these are all distinct. We then perturb $J$ to also make it integrable near $F$ and $G$, and do so through compatible almost complex structures that are adapted to $\psi_{\mathfrak{r}}$ and $\psi_{\mathfrak{r}}$ and fixed near $S_{0} \cup T_{0} \cup T_{\infty}$ and all the $H_{i}$. Observe that with this choice of perturbation, all of the properties listed in Proposition \ref{p:black} except for property (1) continue to hold. Next, we choose a Darboux ball $B_{i}$ of capacity $\epsilon > 0$ centered at each $p_{i}$, and arrange that the pushforward of $J$ at $p_{i}$ coincide with the standard complex structure at $0 \in B(\epsilon) \sub \RR^{4}$. Also, in each $B_{i}$ the spheres $F$, $G$, and $H_{i}$ intersect at the origin and are tangent to distinct complex planes, so by making $\epsilon$ smaller, we can perturb the spheres to actually coincide with their respective tangent planes in these coordinates. We then perturb $J$ a final time so that its pushforward on \emph{all of} $B_{i}$ is the standard complex structure in $B(\epsilon)$. This leaves the above spheres $J$-holomorphic, and allows us to symplectically blow up along the $B_{i}$. Finally, we extend $J$ smoothly away from $F$, $G$, the $H_{i}$ and the $B_{i}$ to all of $X$. This can be done without affecting any previous perturbations.

With these preliminaries out of the way, we now want to modify $(X, \omega_{X})$ through a sequence of blowups, inflations, and blowdowns, all while keeping track of intersection numbers and areas. We start with the following:

\bigskip

\begin{center}

\begin{tabular}{c c}
{\textbf{Initial Intersection Numbers}} & \qquad \qquad {\textbf{Initial Areas}} \\ \\

\begin{tabular}{c|c|c|c|c|c|}\cline{2-6}
& $F$ & $G$ & $E_{\mathfrak{r}}$ & $E_{\mathfrak{s}}$ & $H$ \\ \hline
\multicolumn{1}{|c|}{$F$} & $2d - 1$ & & & & \\ \hline
\multicolumn{1}{|c|}{$G$} & $2d - 2$ & $2d - 3$ & & & \\ \hline
\multicolumn{1}{|c|}{$E_{\mathfrak{r}}$} & $k$ & $(d - 1) - k$ & $*$ & & \\ \hline
\multicolumn{1}{|c|}{$E_{\mathfrak{s}}$} & $l$ & $(d - 1) - l$ & $0$ & $*$ & \\ \hline
\multicolumn{1}{|c|}{$H$} & $1$ & $1$ & $*$ & $*$ & $0$ \\ \hline
\end{tabular}

& \qquad \qquad

\begin{tabular}{c|c|}\cline{2-2}
& $\omega_{X}$-area \\ \hline
\multicolumn{1}{|c|}{$F$} & $2d + 1$ \\ \hline
\multicolumn{1}{|c|}{$G$} & $2d - 1$ \\ \hline
\multicolumn{1}{|c|}{$E_{\mathfrak{r}}$} & $1$ \\ \hline
\multicolumn{1}{|c|}{$E_{\mathfrak{s}}$} & $1$ \\ \hline
\multicolumn{1}{|c|}{$H$} & $2$ \\ \hline
\end{tabular}

\\

\end{tabular}

\end{center}

\bigskip

The entries in the above tables are symmetric along the main diagonal, which is why we leave the entries above it blank. Entries marked with an asterisk signify undetermined quantities that will be irrelevant for us. The remaining entries are either justified by our assumptions or the various properties listed in Proposition \ref{p:black}. The nonnegative integers $k$ and $l$ are undetermined since all we know for certain is that the sum of the intersections of $E_{\mathfrak{r}}$ and $E_{\mathfrak{s}}$ with $F$ and $G$ add up to $d-1$, by properties (6) and (7) of Proposition \ref{p:black}.

\medskip

\noindent{\textbf{Step 1:}} Perform a symplectic blowup along $B_{i}$ for $i = 1, \dots, 2d - 2$.

\medskip

We obtain a new symplectic manifold $(X_{1}, \omega_{1})$, which contains a new exceptional divisor $\mathcal{E}_{i}$ of area $\epsilon$ for $i = 1, \dots, 2d - 2$. We denote the proper transforms of all objects under consideration by the same symbols. Thus $F$ and $G$ are now disjoint symplectic spheres whose self-intersection numbers have been reduced by $2d - 2$. The $H_{i}$ are still symplectic spheres that are now disjoint from $F$ and $G$ and have self-intersection number $-1$. From property (9) of Proposition \ref{p:black}, we know that $E_{\mathfrak{r}}$ intersects exactly $d - 1$ of the $H_{i}$, each in a single point, and that $E_{\mathfrak{s}}$ intersects the remaining $d - 1$ of the $H_{i}$, again each in a single point. The following table records all relevant intersection numbers after Step 1:

\bigskip

\begin{center}

\begin{tabular}{c}
{\textbf{Intersection Numbers after Step 1}} \\ \\

\begin{tabular}{c|c|c|c|c|c|c|c|}\cline{2-8}
& $F$ & $G$ & $E_{\mathfrak{r}}$ & $E_{\mathfrak{s}}$ & $H$ & $\{\mathcal{E}_{i}\}$ & $\{H_{i}\}$ \\ \hline
\multicolumn{1}{|c|}{$F$} & $1$ & & & & & & \\ \hline
\multicolumn{1}{|c|}{$G$} & $0$ & $-1$ & & & & & \\ \hline
\multicolumn{1}{|c|}{$E_{\mathfrak{r}}$} & $k$ & $(d - 1) - k$ & $*$ & & & & \\ \hline
\multicolumn{1}{|c|}{$E_{\mathfrak{s}}$} & $l$ & $(d - 1) - l$ & $0$ & $*$ & & & \\ \hline
\multicolumn{1}{|c|}{$H$} & $1$ & $1$ & $*$ & $*$ & $0$ & & \\ \hline
\multicolumn{1}{|c|}{$\{\mathcal{E}_{i}\}$} & $2d - 2$ & $2d - 2$ & $0$ & $0$ & $0$ & $-1$ & \\ \hline
\multicolumn{1}{|c|}{$\{H_{i}\}$} & $0$ & $0$ & $d - 1$ & $d - 1$ & $0$ & $1$ & $-1$ \\ \hline
\end{tabular}

\\

\end{tabular}

\end{center}

\bigskip

In the above table, by the intersection number between, say, the individual sphere $F$ and the collection $\{\mathcal{E}_{i}\}$, we mean the sum of the intersection numbers of $F$ with all of the spheres $\mathcal{E}_{i}$. On the other hand, by the self-intersection number of the collection $\{\mathcal{E}_{i}\}$ with itself, we mean the self-intersection number of each individual $\mathcal{E}_{i}$. Similar comments apply to other intersection numbers. As for symplectic areas, the areas of $F$ and $G$ are now
\[
\omega_{1}(F) = (2d + 1) - (2d - 2)\epsilon, \qquad \omega_{1}(G) = (2d - 1) - (2d - 2)\epsilon.
\]
Likewise, the area of $H_{i}$ is now $2 - \epsilon$, but the areas of $E_{\mathfrak{r}}$, $E_{\mathfrak{s}}$, and the unbroken leaf $H$, are unchanged. The following table records all relevant symplectic areas after Step 1:

\bigskip

\begin{center}

\begin{tabular}{c}
{\textbf{Areas after Step 1}} \\ \\

\begin{tabular}{c|c|}\cline{2-2}
& $\omega_{1}$-area \\ \hline
\multicolumn{1}{|c|}{$F$} & $(2d + 1) - (2d - 2)\epsilon$ \\ \hline
\multicolumn{1}{|c|}{$G$} & $(2d - 1) - (2d - 2)\epsilon$ \\ \hline
\multicolumn{1}{|c|}{$E_{\mathfrak{r}}$} & $1$ \\ \hline
\multicolumn{1}{|c|}{$E_{\mathfrak{s}}$} & $1$ \\ \hline
\multicolumn{1}{|c|}{$H$} & $2$ \\ \hline
\multicolumn{1}{|c|}{$\mathcal{E}_{i}$} & $\epsilon$ \\ \hline
\multicolumn{1}{|c|}{$H_{i}$} & $2 - \epsilon$ \\ \hline
\end{tabular}

\\

\end{tabular}

\end{center}

\bigskip

\noindent{\textbf{Step 2:}} Perform an inflation of capacity $d - 1$ along $F$ and $G$.

\medskip

Observe that after Step 1, the almost complex structure $J$ lifts to one on $X_{1}$ that is compatible with $\omega_{1}$. We can then apply Proposition \ref{p:inflate} and perform an inflation of capacity $d - 1$ along $F$. This yields a new symplectic form on $X_{1}$, which now only \emph{tames} $J$. All previously $J$-holomorphic spheres in $X_{1}$ continue to be $J$-holomorphic, so that any intersections between $F$, $G$, the $H_{i}$, and the $\mathcal{E}_{i}$ are positive and transverse.

At this point, we would like to replace $J$ by a new almost complex structure which is \emph{compatible} with the intermediate symplectic form and which makes all symplectic spheres in the previous paragraph pseudoholomorphic. This requires possibly perturbing these spheres in a small Darboux ball centered at each intersection point after specifying the almost complex structure there. Fortunately, perturbing the spheres in this manner will not affect subsequent considerations since all that is relevant to us are the various intersection numbers and areas. Performing said perturbations, we continue to use the notation $J$ for the new compatible almost complex structure as well as the same notation for the various perturbed spheres.

We can now apply Proposition \ref{p:inflate} again and perform an inflation of capacity $d - 1$ along $G$, being sure to initially choose $\epsilon$ small enough so that $(2d - 2)\epsilon < d$ and hence that $d - 1 < \omega_{1}(G)$. We denote the resulting symplectic manifold by $(X_{2}, \omega_{2})$. The following table records all relevant symplectic areas after Step 2:

\bigskip

\begin{center}

\begin{tabular}{c}
{\textbf{Areas after Step 2}} \\ \\

\begin{tabular}{c|c|}\cline{2-2}
& $\omega_{2}$-area \\ \hline
\multicolumn{1}{|c|}{$F$} & $3d - (2d - 2)\epsilon$ \\ \hline
\multicolumn{1}{|c|}{$G$} & $d - (2d - 2)\epsilon$ \\ \hline
\multicolumn{1}{|c|}{$E_{\mathfrak{r}}$} & $d^{2} - 2d + 2$ \\ \hline
\multicolumn{1}{|c|}{$E_{\mathfrak{s}}$} & $d^{2} - 2d + 2$ \\ \hline
\multicolumn{1}{|c|}{$H$} & $2d$ \\ \hline
\multicolumn{1}{|c|}{$\mathcal{E}_{i}$} & $\epsilon + 2d - 2$ \\ \hline
\multicolumn{1}{|c|}{$H_{i}$} & $2 - \epsilon$ \\ \hline
\end{tabular}

\\

\end{tabular}

\end{center}

\bigskip

\noindent{\textbf{Step 3:}} Apply an inflation of capacity $\epsilon$ along each $\mathcal{E}_{i}$.

\medskip

Observe that after Step 2, the symplectic form $\omega_{2}$ now only tames $J$. All previously $J$-holomorphic spheres in $X_{2}$ continue to be $J$-holomorphic, and so any intersections between $F$, $G$, the $H_{i}$, and the $\mathcal{E}_{i}$ are positive and transverse. As in Step 2, after possibly perturbing these spheres near their various intersection points, we may assume that they are $J$-holomorphic for some new almost complex structure $J$ compatible with $\omega_{2}$. We then apply Proposition \ref{p:inflate} and perform an inflation of capacity $\epsilon$ along each $\mathcal{E}_{i}$, performing new perturbations and defining a new $J$ after each inflation as necessary. We denote the resulting symplectic manifold by $(X_{3}, \omega_{3})$. The following table records all relevant symplectic areas after Step 3:

\bigskip

\begin{center}

\begin{tabular}{c}
{\textbf{Areas after Step 3}} \\ \\

\begin{tabular}{c|c|}\cline{2-2}
& $\omega_{3}$-area \\ \hline
\multicolumn{1}{|c|}{$F$} & $3d$ \\ \hline
\multicolumn{1}{|c|}{$G$} & $d$ \\ \hline
\multicolumn{1}{|c|}{$E_{\mathfrak{r}}$} & $d^{2} - 2d + 2$ \\ \hline
\multicolumn{1}{|c|}{$E_{\mathfrak{s}}$} & $d^{2} - 2d + 2$ \\ \hline
\multicolumn{1}{|c|}{$H$} & $2d$ \\ \hline
\multicolumn{1}{|c|}{$\mathcal{E}_{i}$} & $2d - 2$ \\ \hline
\multicolumn{1}{|c|}{$H_{i}$} & $2$ \\ \hline
\end{tabular}

\\

\end{tabular}

\end{center}

\bigskip

\noindent{\textbf{Step 4:}} Perform a symplectic blowdown along $H_{i}$ for $i = 1, \dots, 2d - 2$.

\medskip

Observe that after Step 3, the symplectic form $\omega_{3}$ now only tames $J$. All previously $J$-holomorphic spheres in $X_{3}$ continue to be $J$-holomorphic, and therefore symplectic. After additional perturbations near intersection points, we obtain a new compatible $J$ which is integrable near $F$ and $G$ and is of standard form near each $H_{i}$ so that the symplectic blowdown can be performed, see Theorem 3.14 in \cite{We2}. After performing these blowdowns, we denote the resulting symplectic manifold by $(X_{4}, \omega_{4})$. The spheres $F$ and $G$, and the unbroken leaf $H$ are unaffected. Each sphere $\mathcal{E}_{i}$ has its self-intersection number increased by $1$ and gains a new intersection with each of the disks $E_{\mathfrak{r}}$ and $E_{\mathfrak{s}}$. This gives a total of $d - 1$ new intersections between the collection $\{\mathcal{E}_{i}\}$ and each of the two disks. The following table records all relevant intersection numbers after Step 4:

\bigskip

\begin{center}

\begin{tabular}{c c}
{\textbf{Intersection Numbers after Step 4}} \\ \\

\begin{tabular}{c|c|c|c|c|c|c|}\cline{2-7}
& $F$ & $G$ & $E_{\mathfrak{r}}$ & $E_{\mathfrak{s}}$ & $H$ & $\{\mathcal{E}_{i}\}$ \\ \hline
\multicolumn{1}{|c|}{$F$} & $1$ & & & & & \\ \hline
\multicolumn{1}{|c|}{$G$} & $0$ & $-1$ & & & & \\ \hline
\multicolumn{1}{|c|}{$E_{\mathfrak{r}}$} & $k$ & $(d - 1) - k$ & $*$ & & & \\ \hline
\multicolumn{1}{|c|}{$E_{\mathfrak{s}}$} & $l$ & $(d - 1) - l$ & $0$ & $*$ & & \\ \hline
\multicolumn{1}{|c|}{$H$} & $1$ & $1$ & $*$ & $*$ & $0$ & \\ \hline
\multicolumn{1}{|c|}{$\{\mathcal{E}_{i}\}$} & $2d - 2$ & $2d - 2$ & $d - 1$ & $d - 1$ & $0$ & $0$ \\ \hline
\end{tabular}

\\

\end{tabular}

\end{center}

\bigskip

As for symplectic areas, each sphere $\mathcal{E}_{i}$ has its area increased by $2$ and the disks $E_{\mathfrak{r}}$ and $E_{\mathfrak{s}}$ both have their areas increased by $2d - 2$. The following table records all relevant symplectic areas after Step 4:

\bigskip

\begin{center}

\begin{tabular}{c}
{\textbf{Areas after Step 4}} \\ \\

\begin{tabular}{c|c|}\cline{2-2}
& $\omega_{4}$-area \\ \hline
\multicolumn{1}{|c|}{$F$} & $3d$ \\ \hline
\multicolumn{1}{|c|}{$G$} & $d$ \\ \hline
\multicolumn{1}{|c|}{$E_{\mathfrak{r}}$} & $d^{2}$ \\ \hline
\multicolumn{1}{|c|}{$E_{\mathfrak{s}}$} & $d^{2}$ \\ \hline
\multicolumn{1}{|c|}{$H$} & $2d$ \\ \hline
\multicolumn{1}{|c|}{$\mathcal{E}_{i}$} & $2$ \\ \hline
\end{tabular}

\\

\end{tabular}

\end{center}

\bigskip

\subsection{Concluding Arguments}

We have now transformed $(X, \omega_{X})$ into a new symplectic manifold $(X_{4}, \omega_{4})$ with symplectically embedded spheres $F$ and $G$. The next proposition shows that we have not strayed very far.

\begin{prp}\label{p:ruled}
The symplectic manifold $(X_{4}, \omega_{4})$ is symplectomorphic to $(X, d\mspace{2mu}\omega_{X})$ by a symplectomorphism identifying the sphere $F$ with $S_{\infty}$, the sphere $G$ with $S_{0}$, and $T_{0}$ and $T_{\infty}$ with themselves.
\end{prp}

\begin{proof}
We start by determining the rank of $H_{2}(X_{4})$. Initially, the rank of $H_{2}(X)$ is $2$, and performing the $2d - 2$ blowups in Step 1 increases this to $2 + (2d - 2) = 2d$. The inflations in Step 2 and Step 3 have no effect, and performing the $2d - 2$ blowdowns in Step 4 decreases the rank back to $2$. Thus $H_{2}(X_{4})$ has rank $2$, and since $(X_{4}, \omega_{4})$ contains symplectically embedded spheres with self-intersection numbers $\pm1$, it follows from Theorem 1.4 of \cite{Mc1} that it must be symplectomorphic to $(X, c\mspace{2mu}\omega_{X})$ for some $c > 0$, which in fact must be equal to
\[
\omega_{4}(F)/3 = 3d/3 = d.
\]
This symplectomorphism takes $F$, $G$, $T_{0}$, and $T_{\infty}$ to symplectically embedded spheres in $(X, d\mspace{2mu}\omega_{X})$, and we follow it up by a Hamiltonian diffeomorphism that takes the resulting nodal configuration of spheres to the axes $Y$.
\end{proof}

We now have two disjoint Lagrangian tori $L_{\mathfrak{r}}$ and $L_{\mathfrak{s}}$ in $(X_{4}, \omega_{4})$, which by Proposition \ref{p:ruled} is a monotone symplectic manifold.

\begin{prp}\label{p:newmonotone}
The tori $L_{\mathfrak{r}}$ and $L_{\mathfrak{s}}$ are monotone in $(X_{4}, \omega_{4})$.
\end{prp}

\begin{proof}
Before any of the steps needed to obtain $(X_{4}, \omega_{4})$ from $(X, \omega_{X})$, we consider a plane in the family $\mathfrak{r}_{0}$. Its compactification is a disk with boundary on $L_{\mathfrak{r}}$, which has Maslov index $2$ and whose boundary represents the foliation class $\beta_{\mathfrak{r}}$. As a smooth map, this disk is unaffected by the four steps that transform $(X, \omega_{X})$ into $(X_{4}, \omega_{4})$. In particular, it sill has Maslov index $2$ and its boundary still represents $\beta_{\mathfrak{r}}$. Its area however is now exactly $d$, since by property (4) of Proposition \ref{p:black}, it is intersected by exactly one of $F$ or $G$, so that the inflations from Step 2 increase its area by exactly $d - 1$. On the other hand, we know by property (2) of Proposition \ref{p:black} that the class of $E_{\mathfrak{r}}\vert_{\SS^{1}}$ and the foliation class $\beta_{\mathfrak{r}}$ form an integral basis for $H_{1}(L_{\mathfrak{r}})$ in $(X, \omega_{X})$, and this continues to hold in $(X_{4}, \omega_{4})$. It will now follow from Lemma \ref{l:monotone} that $L_{\mathfrak{r}}$ is monotone in $(X_{4}, \omega_{4})$ if the new Maslov index of $E_{\mathfrak{r}}$ is $(2/d)d^{2} = 2d$. To see that this is indeed the case, observe that after Step 3, the Maslov index is still $2$, and because $E_{\mathfrak{r}}$ intersects only $d - 1$ of the $H_{i}$, each in a single point, performing the $2d - 2$ blowdowns in Step 4 will then increase its Maslov index to $2 + (2d - 2) = 2d$. This shows that $L_{\mathfrak{r}}$ is monotone, and the same argument applies to $L_{\mathfrak{s}}$.
\end{proof}

On a purely topological level, we can also deduce the following fact, which when combined with the symplectomorphism of Proposition \ref{p:ruled} and Proposition \ref{p:exact}, yields the desired contradiction and consequently the proof of Theorem \ref{t:main}.

\begin{prp}\label{p:homtriv}
The tori $L_{\mathfrak{r}}$ and $L_{\mathfrak{s}}$ are both homologicaly nontrivial in $X_{4} \sd (F \cup G \cup T_{0} \cup T_{\infty})$. 
\end{prp}

\begin{proof}
This will be proved by an argument similar to the one used in the proof of Lemma \ref{l:homtriv}. Consider the almost complex structure $J$ on $X \sd (L_{\mathfrak{r}} \cup L_{\mathfrak{s}})$ before Step 1, after we have made all additional perturbations. Then the foliation $\mathcal{F}(L_{\mathfrak{r}}, L_{\mathfrak{s}}, \psi_{\mathfrak{r}}, \psi_{\mathfrak{s}}, J)$ yields a projection map $\pi \colon X \to S_{0}$. Extending $J$ to all of $X$ and performing the symplectic blowups in Step 1, we get a blowdown map, which induces a diffeomorphism
\[
X_{1} \sd (\mathcal{E}_{1} \cup \cdots \cup \mathcal{E}_{2d - 2}) \approx X \sd \{p_{1}, \dots, p_{2d - 2}\}.
\]
Similarly, after performing the symplectic blowdowns in Step 4, we have a diffeomorphism
\[
X_{3} \sd (H_{1} \cup \cdots \cup H_{2d - 2}) \approx X_{4} \sd \{q_{1}, \dots, q_{2d - 2}\},
\]
where $q_{i}$ denotes the image of $H_{i}$ under the blowdown map. Combining these yields a diffeomorphism
\[
X \sd (H_{1} \cup \cdots \cup H_{2d - 2}) \approx X_{4} \sd (\mathcal{E}_{1} \cup \cdots \cup \mathcal{E}_{2d - 2}),
\]
and $\pi$ can therefore be viewed as a map
\[
\pi \colon X_{4} \sd (\mathcal{E}_{1} \cup \cdots \cup \mathcal{E}_{2d - 2}) \to S_{0} \sd \{\pi(p_{1}), \dots, \pi(p_{2d - 2})\}.
\]
After the four steps needed to transform $(X, \omega_{X})$ to $(X_{4}, \omega_{4})$, the broken and unbroken leaves of $\mathcal{F}$ still have unique positive and transverse intersections with $F$ and $G$. Of course, since we are redefining $J$ at each step, the resulting leaves may no longer be $J$-holomorphic in $X_{4} \sd (L_{\mathfrak{r}} \cup L_{\mathfrak{s}})$.

Now write $Y^{\prime} = F \cup G \cup T_{0} \cup T_{\infty}$ in $X_{4}$ and choose an embedded path
\[
\gamma \colon [0, 1] \to S_{0} \sd \{\pi(p_{1}), \dots, \pi(p_{2d - 2})\}
\]
starting at $\pi(T_{0})$ and ending at $\pi(T_{\infty})$. Since $C_{\mathfrak{r}}$ separates these points, we may assume that $\gamma$ intersects it transversally at some unique time $t_{\mathfrak{r}}$. For each $t \in [0, 1]$, we then choose an embedded path $\sigma_{t}$ in $\pi^{-1}(\gamma(t))$ starting at $F$ and ending $G$. The union of these defines a homology class $\Sigma \in H_{2}(X_{4}, Y^{\prime})$. Now by construction, the set $K_{\mathfrak{r}} = L_{\mathfrak{r}} \cap \pi^{-1}(\gamma(t_{0}))$ is a circle bounding (perturbed) compactified planes from $\mathfrak{r}_{0}$ and $\mathfrak{r}_{\infty}$, which glue together to form a smooth sphere intersecting $F$ and $G$. Since $K_{\mathfrak{r}}$ separates these intersection points by properties (4) and (5) of Proposition \ref{p:black}, it must intersect the path $\sigma_{t_{\mathfrak{r}}}$, and this is the only intersection between $L_{\mathfrak{r}}$ and $\Sigma$. Therefore $L_{\mathfrak{r}}$ must be homologicaly nontrivial in $X \sd Y^{\prime}$. The same argument applies to $L_{\mathfrak{s}}$.
\end{proof}

\section{Deformation and Intersection Arguments}

The final section of this paper is devoted to the proof of Proposition \ref{p:black}. As the reader shall see, the subsequent considerations are primarily of a local nature, as most of the analysis is done in arbitrarily small Weinstein neighborhoods of various Lagrangian tori. Consequently, the proofs of many of the results below carry over \emph{mutatis mutandis} from those found in Subsection 3.8 of \cite{HK}. In an effort to make this paper more self-contained, we have endeavored to nevertheless give a full account of the arguments involved in proving Proposition \ref{p:black}, even at the expense of added length and redundancy.

The section begins with a review of the setup prior to the statement of Proposition \ref{p:black}, and the introduction of notation that will be used throughout. We will also describe a \emph{straightening} procedure, which yields a geometrically simple description of foliation leaves in terms of coordinates in a Weinstein neighborhood. We then outline a method for deforming certain punctured curves inside said Weinstein neighborhood given a fixed set of \emph{translation data}, and from this obtain the various objects whose existence is claimed in Proposition \ref{p:black}. We then study the intersections that are created between the aforementioned objects and how they depend on the translation data. Making an appropriate choice of said data will then allow us to achieve the various properties listed in Proposition \ref{p:black}.

\subsection{Preliminaries}

Recall that $Y  = S_{0} \cup S_{\infty} \cup T_{0} \cup T_{\infty}$. We have two disjoint, monotone Lagrangian tori $L_{\mathfrak{r}}$ and $L_{\mathfrak{s}}$ in $X \sd Y$ satisfying the properties listed in Proposition \ref{p:simpli}. We then fix parameterizations $\psi_{\mathfrak{r}}$ and $\psi_{\mathfrak{s}}$. The corresponding Weinstein neighborhoods of $L_{\mathfrak{r}}$ and $L_{\mathfrak{s}}$ will always be assumed to disjoint from each other and from $Y$. For any almost complex structure $J$ on $X \sd (L_{\mathfrak{r}} \cup L_{\mathfrak{s}})$ that is integrable near $Y$ and adapted to $\psi_{\mathfrak{r}}$ and $\psi_{\mathfrak{s}}$, we obtain a relative foliation 
\[
\mathcal{F} = \mathcal{F}(L_{\mathfrak{r}}, L_{\mathfrak{s}}, \psi_{\mathfrak{r}}, \psi_{\mathfrak{s}}, J),
\]
together with foliation classes $\beta_{\mathfrak{r}}$ and $\beta_{\mathfrak{s}}$, and a map $\pi \colon X \to S_{0}$ defined as before. It will be convenient for what is to come to choose the above parameterizations so that $\beta_{\mathfrak{r}}$ is represented by $(0, -1)$ in $H_{1}(L_{\mathfrak{r}}; \psi_{\mathfrak{r}})$ and that $\beta_{\mathfrak{s}}$ is represented by $(0, -1)$ in $H_{1}(L_{\mathfrak{s}}; \psi_{\mathfrak{s}})$. Observe that this does not affect the properties in Proposition \ref{p:simpli}. The sets $C_{\mathfrak{r}} = \pi(L_{\mathfrak{r}})$ and $C_{\mathfrak{s}} = \pi(L_{\mathfrak{s}})$ are disjoint embedded circles in $S_{0}$ that separate $\pi(T_{0})$ and $\pi(T_{\infty})$, and there are disjoint closed disks $D_{\mathfrak{r}}$ and $D_{\mathfrak{s}}$ in $S_{0}$ bounding $C_{\mathfrak{r}}$ and $C_{\mathfrak{s}}$, respectively, such that $\pi(T_{0}) \in D_{\mathfrak{r}}$ and $\pi(T_{\infty}) \in D_{\mathfrak{s}}$. The closure of $S_{0} \sd (D_{\mathfrak{r}} \cup D_{\mathfrak{s}})$ is a closed annulus $A$. The leaves of $\mathcal{F}$ are either unbroken spheres in the class $(0, 1)$, or broken leaves consisting of pairs of planes belonging to one of four families. In view of Proposition \ref{p:simpli}, the planes in $\mathfrak{s}_{0}$ intersect $S_{0}$, those in $\mathfrak{s}_{\infty}$ intersect $S_{\infty}$, those in $\mathfrak{r}_{0}$ intersect both, and those in $\mathfrak{r}_{\infty}$ intersect neither.

In terms of the induced coordinates $(q_{1}, q_{2}, p_{1}, p_{2})$ inside $\mathcal{U}(L_{\mathfrak{r}})$, we may assume that for some fixed $\epsilon > 0$, we have 
\[
\mathcal{U}(L_{\mathfrak{r}}) =  \{\vert p_{1} \vert < \epsilon, \vert p_{2} \vert < \epsilon\}.
\]
Similarly, in terms of the induced coordinates $(Q_{1}, Q_{2}, P_{1}, P_{2})$ inside $\mathcal{U}(L_{\mathfrak{s}})$, we may assume that for that very same $\epsilon$, we have 
\[
\mathcal{U}(L_{\mathfrak{s}}) =  \{\vert P_{1} \vert < \epsilon, \vert P_{2} \vert < \epsilon\}.
\]
After shrinking $\epsilon$ and perturbing $J$ outside of the above neighborhoods, it is possible to \emph{straighten} the leaves of $\mathcal{F}$ so that they have a particularly nice local form in the above coordinates. This is the content of Lemma 3.4 in \cite{HK}, which we reformulate and prove here for the sake of completeness.

\begin{lem}
For some perturbation $J$ and after possibly shrinking $\epsilon$, the leaves of $\mathcal{F}$ have the following description inside $\mathcal{U}(L_{\mathfrak{r}})$:
\begin{enumerate}
\item Unbroken leaves coincide with the cylinders
\[
\{q_{1} = \theta, p_{1} = \delta, \vert p_{2} \vert < \epsilon\}
\]
ranging over $\theta \in \SS^{1}$ and every nonzero $\delta \in (-\epsilon, \epsilon)$.
\item Planes in $\mathfrak{r}_{0}$ coincide with the cylinders
\[
\{q_{1} = \theta, p_{1} = 0, -\epsilon < p_{2} < 0\}
\]
ranging over $\theta \in \SS^{1}$.
\item Planes in $\mathfrak{r}_{\infty}$ coincide with the cylinders
\[
\{q_{1} = \theta, p_{1} = 0, 0 < p_{2} < \epsilon\}
\]
ranging over $\theta \in \SS^{1}$.
\end{enumerate}
Analogous local descriptions for the leaves of $\mathcal{F}$ hold inside $\mathcal{U}(L_{\mathfrak{s}})$.
\end{lem}

\begin{proof}
Items (2) and (3) directly follow from the proofs of Lemma 5.14 and Proposition 5.16 in \cite{DRGI}, together with the surrounding discussion. The idea is that the cylinders $\{q_{1} = \theta, p_{1} = 0, \vert p_{2} \vert < \epsilon\}$ correspond to trivial Reeb cylinders, and the asymptotic convergence properties of the planes in $\mathfrak{r}_{0} \cup \mathfrak{r}_{\infty}$ imply that these can be symplectically glued to these cylinders far enough along their cylindrical ends. Making $\epsilon$ as small as necessary and perturbing $J$ then shows that these modified planes become the new planes of $\mathfrak{r}_{0} \cup \mathfrak{r}_{\infty}$, and now have the desired local form.

To prove item (1), we smoothly identify a neighborhood of $\mathfrak{r}_{0}$ and $\mathfrak{r}_{\infty}$ with $(-\epsilon^{\prime}, \epsilon^{\prime}) \times \SS^{1} \times \DD^{2}$, where the planes comprising broken leaves correspond to the sets $\{0\} \times \SS^{1} \times \DD^{2}$, and circles of the form $
\{p_{1} = \delta, q_{1} = \theta, p_{2} = \pm \epsilon\}$ coincide with the circles $\{\delta\} \times \{\theta\} \times \SS^{1}$. The disks $D_{\delta, \theta} = \{\delta\} \times \{\theta\} \times \DD^{2}$ with $\vert \delta \vert < \epsilon^{\prime}$ extend smoothly into $\mathcal{U}(L_{\mathfrak{r}})$ along the holomorphic cylinders $\{p_{1} = \delta, q_{1} = \theta, \vert p_{2} \vert < \epsilon\}$, and if $\epsilon^{\prime}$ is chosen small enough then these can be assumed to be symplectic, being $C^{0}$-close to the planes in $\mathfrak{r}_{0} \cup \mathfrak{r}_{\infty}$. We then choose $\epsilon < \epsilon^{\prime}$ and deform $J$ outside of $\mathcal{U}(L_{\mathfrak{r}})$ to make the disks $D_{\delta, \theta}$ $J$-holomorphic. It is easy to see that the broken leaves of the resulting relative foliation continue to satisfy (2) and (3) while the unbroken leaves now satisfy (1).
\end{proof}

Assuming that $J$ is now chosen so that the conclusions of the above lemma hold, we can assume without loss of generality that in $\mathcal{U}(L_{\mathfrak{r}})$, the regions $\{p_{1} < 0\}$ and $\{p_{1} > 0\}$ are mapped into the interiors of $D_{\mathfrak{r}}$ and $A \cup D_{\mathfrak{s}}$, respectively. Likewise, we can assume that in $\mathcal{U}(L_{\mathfrak{s}})$, the regions $\{P_{1} < 0\}$ and $\{P_{1} > 0\}$ are mapped into the interiors of $D_{\mathfrak{r}} \cup A$ and $D_{\mathfrak{s}}$, respectively.

\subsection{Deformation Procedures}

We will now describe two deformation procedures, which we then apply to top level curves of buildings produced via Proposition \ref{p:existence} in order to obtain the various curves of Proposition \ref{p:black}. These deformations can in a sense be localized near each of $L_{\mathfrak{r}}$ and $L_{\mathfrak{s}}$, and for simplicity we will first describe the deformations as they will be performed inside $\mathcal{U}(L_{\mathfrak{r}})$. It is then a straightforward matter to perform them near the two tori simultaneously, and we will describe the results in Subsection 7.4. For now, any vector $\mathbf{v} = (a, b) \in (-\epsilon, \epsilon)^{2}$ will be referred to as a \emph{translation vector}. We can use it to translate $L_{\mathfrak{r}}$ to a nearby Lagrangian torus
\[
L_{\mathfrak{r}}(\mathbf{v}) = \{p_{1} = a, p_{2} = b\} \sub \mathcal{U}(L_{\mathfrak{r}}).
\]
We note that this translated torus is no longer monotone, and inherits an induced parameterization $\psi_{\mathfrak{r}}(\mathbf{v})$ and consequently a canonical isomorphism from $H_{1}(L_{\mathfrak{r}}; \psi_{\mathfrak{r}})$ to $H_{1}(L_{\mathfrak{r}}(\mathbf{v}); \psi_{\mathfrak{r}}(\mathbf{v}))$. More generally, for a finite, nonempty collection of \emph{translation data}
\[
\mathbf{V} = \{\mathbf{v}_{1}, \dots, \mathbf{v}_{k}\} = \{(a_{1}, b_{2}), \dots, (a_{k}, b_{k})\},
\]
we set
\[
L_{\mathfrak{r}}(\mathbf{V}) = L_{\mathfrak{r}}(\mathbf{v_{1}}) \cup \cdots \cup L_{\mathfrak{r}}(\mathbf{v_{k}}),
\]
and define an almost complex structure $J_{\mathbf{V}}$ on $X \sd L_{\mathfrak{r}}(\mathbf{V})$, which coincides with $J$ outside $\mathcal{U}(L_{\mathfrak{r}})$, and within it takes on the form
\[
J_{\mathbf{V}}(\partial_{q_{i}}) = -\rho_{\mathbf{V}}\partial_{p_{i}}, \qquad i = 1, 2,
\]
where $\rho_{\mathbf{V}}$ is a positive function away from $L_{\mathfrak{r}}(\mathbf{V})$ and in a neighborhood of each $L_{\mathfrak{r}}(\mathbf{v}_{i})$ has the form
\[
\rho_{\mathbf{V}} = \sqrt{(p_{1} - a_{i})^{2} + (p_{2} - b_{i})^{2}}.
\]
We will say that $J_{\mathbf{V}}$ is \emph{adapted to $\psi_{\mathfrak{r}}(\mathbf{V})$}. The set of all such almost complex structures will be denoted by $\mathcal{J}_{\mathbf{V}}$. Fixing a family $\{J_{t}\}_{t \geq 0}$ of tame almost complex structures on $(X, \omega_{X})$ that stretches to a generically chosen $J_{\mathbf{V}}$, we obtain a foliation $\mathcal{F}(\mathbf{V})$ of $X \sd (L_{\mathfrak{r}}(\mathbf{V}) \cup L_{\mathfrak{s}})$. We must be careful to make only small translations, so that positivity of intersections ensures $\mathcal{F}(\mathbf{V})$ continues to have the same simple form that we have been considering. (Indeed, constructing relative foliations for arbitrary Lagrangian tori by the splitting procedure can also result in index $0$ cylinders in the top level, see \cite{DRGI}.) We can then straighten the leaves of $\mathcal{F}(\mathbf{V})$ so that a leaf intersecting $\mathcal{U}(L_{\mathfrak{r}})$ does so along cylinders of the form $\{q_{1} = \theta, p_{1} = \delta, \vert p_{2} \vert < \epsilon\}$, while such a leaf is broken if and only if $\mathbf{V}$ contains a translation vector of the form $(\delta, b_{i})$. See also Lemma 3.26 of \cite{HK}.

Our first deformation procedure is based on the following lemma, which is Lemma 3.27 in \cite{HK}. It is ultimately a consequence of the inverse function theorem, since for small enough translation vectors $\mathbf{v}$, a regular curve with punctures on $L_{\mathfrak{r}}$ will persist to one with punctures on $L_{\mathfrak{r}}(\mathbf{v})$ along a suitably constructed isotopy. Since the proof is of a relatively standard nature and does not contain any ideas relevant to the rest of the section, we have opted not to reproduce it here.

\begin{lem}\label{l:fukaya1}
Let $u$ be a regular $J$-holomorphic curve in $X \sd (L_{\mathfrak{r}} \cup L_{\mathfrak{s}})$ with $m \geq 0$ punctures on $L_{\mathfrak{r}}$ and $n \geq 0$ punctures on $L_{\mathfrak{s}}$.
\begin{enumerate}
\item[(a)] For any translation vector $\mathbf{v}$ with $\Vert \mathbf{v} \Vert$ sufficiently small, there is a regular $J_{\mathbf{v}}$-holomorphic curve $u(\mathbf{v})$ with $m$ punctures on $L_{\mathfrak{r}}(\mathbf{v})$ and $n$ punctures on $L_{\mathfrak{s}}$.
\item[(b)] The asymptotic orbits of $u(\mathbf{v})$ cover geodesics that have the same representatives in $\smash{H_{1}(L_{\mathfrak{r}}(\mathbf{v}); \psi_{\mathfrak{r}}(\mathbf{v}))}$ and $H_{1}(L_{\mathfrak{s}}; \psi_{\mathfrak{s}})$ as the geodesics covered by orbits of $u$ have in $H_{1}(L_{\mathfrak{r}}; \psi_{\mathfrak{r}})$ and $H_{1}(L_{\mathfrak{s}}; \psi_{\mathfrak{s}})$, respectively.
\end{enumerate}
\end{lem}

For $d > 0$ sufficiently large, we now appeal to Proposition \ref{p:existence} to obtain two pseudoholomorphic buildings $\mathbf{F}$ and $\mathbf{G}$ in the classes $(1, d)$ and $(1, d - 1)$, respectively. The top level curves of both of these buildings are all of area $1$ and are thus simply covered. The top level curves of $\mathbf{F}$ are
\[
\{u_{\mathfrak{r}}, \underline{u}, u_{\mathfrak{s}}, u_{1}, \dots, u_{d - 1}, \mathfrak{u}_{1}, \dots, \mathfrak{u}_{d - 1}\}.
\]
Here $u_{\mathfrak{r}}$, $\underline{u}$, and $u_{\mathfrak{s}}$ are the essential curves of $\mathbf{F}$. For some $0 \leq \alpha_{0} \leq d - 1$, the planes $u_{1}, \dots, u_{\alpha_{0}}$ belong to $\mathfrak{r}_{0}$ and the planes $u_{\alpha_{0} + 1}, \dots, u_{d - 1}$ belong to $\mathfrak{r}_{\infty}$. For some $0 \leq \beta_{0} \leq d - 1$, the planes $\mathfrak{u}_{1}, \dots, \mathfrak{u}_{\beta_{0}}$ belong to $\mathfrak{s}_{0}$ and the planes $\mathfrak{u}_{\beta_{0} + 1}, \dots, \mathfrak{u}_{d - 1}$ belong to $\mathfrak{s}_{\infty}$. The top level curves of $\mathbf{G}$ are
\[
\{v_{\mathfrak{r}}, \underline{v}, v_{\mathfrak{s}}, v_{1}, \dots, v_{d - 2}, \mathfrak{v}_{1}, \dots, \mathfrak{v}_{d - 2}\}.
\]
Here $v_{\mathfrak{r}}$, $\underline{v}$, and $v_{\mathfrak{s}}$ are the essential curves of $\mathbf{G}$. For some $0 \leq \gamma_{0} \leq d - 2$, the planes $v_{1}, \dots, v_{\gamma_{0}}$ belong to $\mathfrak{r}_{0}$ and the planes $v_{\gamma_{0} + 1}, \dots, v_{d - 2}$ belong to $\mathfrak{r}_{\infty}$. For some $0 \leq \delta_{0} \leq d - 2$, the planes $\mathfrak{v}_{1}, \dots, \mathfrak{v}_{\delta_{0}}$ belong to $\mathfrak{s}_{0}$ and the planes $\mathfrak{v}_{\delta_{0} + 1}, \dots, \mathfrak{v}_{d - 2}$ belong to $\mathfrak{s}_{\infty}$. We would like to apply Lemma \ref{l:fukaya1} to the top level curves of $\mathbf{F}$ and $\mathbf{G}$, and as such we need to know that these curves are regular for any choice of adapted almost complex structure. This is the content of the following lemma, which is Lemma 3.28 in \cite{HK}.

\begin{lem}
For any tame almost complex structure on $(X \sd (L_{\mathfrak{r}} \cup L_{\mathfrak{s}}), \omega_{X})$ that is adapted to $\psi_{\mathfrak{r}}$ and $\psi_{\mathfrak{s}}$, the top level curves of $\mathbf{F}$ and $\mathbf{G}$ are all regular.
\end{lem}

\begin{proof}
The top level curves of both buildings are simply covered and of area $1$. Since they are parts of limits of embedded spheres, it follows that they are also embedded. Furthermore, each such curve is either a plane or a cylinder. The automatic regularity criterion in \cite{We1}, see in particular Remark 1.2, in conjunction with the index formula in Proposition \ref{p:indextop}, implies that these will all be regular provided the Maslov indices of their compactifications are positive. Now if $u$ is a top level plane of either building, then by monotonicity, we know that $\mu(\overline{u}) = 2\mspace{1mu}\omega_{X}(\overline{u}) = 2$, so $u$ is automatically regular. On the other hand, if $u$ is a cylinder, then we can cap off one end of its compactification $\overline{u}$ with a smooth disk $D$ in $X$. Using additivity of symplectic area and Maslov index, we see that
\begin{align*}
2\mspace{1mu}\omega_{X}(\overline{u}) + 2\mspace{1mu}\omega_{X}(D) &= 2\mspace{1mu}\omega_{X}(\overline{u} \cup D) \\
&= \mu(\overline{u} \cup D) \\
&= \mu(\overline{u}) + \mu(D) \\
&= \mu(\overline{u}) + 2\mspace{1mu}\omega_{X}(D).
\end{align*}
This implies that $\mu(\overline{u}) = 2\mspace{1mu}\omega_{X}(\overline{u}) = 2$, showing that $u$ is regular as desired.
\end{proof}

With this result in hand and focusing on $\mathbf{F}$ for the time being, we choose a translation vector $\mathbf{v} = (a, b)$ as in (a) of Lemma \ref{l:fukaya1} and apply that lemma to the top level curves of $\mathbf{F}$ with punctures on $L_{\mathfrak{r}}$, leaving all remaining curves untouched. This yields a deformed building $\mathbf{F}(\mathbf{v})$ whose top level curves are now
\[
\{\smash{u_{\mathfrak{r}, \mathbf{V}}}, \underline{u}(\mathbf{v}), u_{\mathfrak{s}}, u_{1}(\mathbf{v}), \dots, u_{d - 1}(\mathbf{v}), \mathfrak{u}_{1}, \dots, \mathfrak{u}_{d - 1}\}.
\]
The middle and bottom level curves of $\mathbf{F}(\mathbf{v})$ are identical to those of $\mathbf{F}$ except that we now view curves in copies of $\RR \times S^{\ast}L_{\mathfrak{r}}$ and $T^{\ast}L_{\mathfrak{r}}$ as being curves in the corresponding copies of $\RR \times S^{\ast}L_{\mathfrak{r}}(\mathbf{v})$ and $T^{\ast}L_{\mathfrak{r}}(\mathbf{v})$, respectively. We can also define a deformed building $\mathbf{G}(\mathbf{v})$ in the same way. These buildings admit compactifications that can then be deformed in an arbitrarily small neighborhood of $L_{\mathfrak{r}}(\mathbf{v})$, resulting in smooth spheres $F \colon \SS^{2} \to X \sd (L_{\mathfrak{r}} \cup L_{\mathfrak{s}})$ and $G \colon \SS^{2} \to X \sd (L_{\mathfrak{r}} \cup L_{\mathfrak{s}})$. Using our description of the foliation $\mathcal{F}(\mathbf{v})$ in $\mathcal{U}(L_{\mathfrak{r}})$ together with the fact that essential curves project to disjoint images in $S_{0}$ under $\pi$, we can show that these spheres are $J_{\mathbf{v}}$-holomorphic away from an arbitrarily small neighborhood of $L_{\mathfrak{r}}(\mathbf{v})$. In fact, we have the following more general result, combining Lemma 3.29 and Lemma 3.30 in \cite{HK}.

\begin{lem}\label{l:deform1}
Let $\mathbf{V} = \{\mathbf{0}, \mathbf{v}_{1}, \mathbf{v}_{2}\} = \{(0, 0), (a_{1}, b_{1}), (a_{2}, b_{2})\}$ be translation data with $\Vert \mathbf{v}_{1} \Vert$ chosen small enough to define $\mathbf{F}(\mathbf{v}_{1})$ and $\mathbf{G}(\mathbf{v}_{1})$. If $a_{1}$ and $b_{1}$ are nonzero, $\vert a_{1} \vert$ is sufficiently small with respect to $\vert b_{1} \vert$, and $\Vert \mathbf{v}_{2} \Vert$ is sufficiently small with respect to $\vert a_{1} \vert$, then each top level curve of both $\mathbf{F}(\mathbf{v}_{1})$ and $\mathbf{G}(\mathbf{v}_{1})$ is $J_{\mathbf{V}}$-holomorphic for some $J_{\mathbf{V}}$ in $\mathcal{J}_{\mathbf{V}}$.
\end{lem}

\begin{proof}
We will give a proof when $\mathbf{v}_{2} = \mathbf{0}$, since the general case follows by similar arguments. Now, we have mentioned that the (straightened) leaves of the induced foliation $\mathcal{F}(\mathbf{V})$ intersect $\mathcal{U}(L_{\mathfrak{r}})$ in cylinders of the form $\{p_{1} = \delta, q_{1} = \theta, \vert p_{2} \vert < \epsilon\}$. Thus assuming that $a_{1} = 0$ and $b_{1} \neq 0$ for the time being, which means the translation is strictly vertical, we see that the preimages of the interiors of $D_{\mathfrak{r}}$ and $A \cup D_{\mathfrak{s}}$ under the projection corresponding to $L_{\mathfrak{r}}(\mathbf{v})$ intersect $\mathcal{U}(L_{\mathfrak{r}})$ in the sets $\{p_{1} < 0\}$ and $\{p_{1} > 0\}$, respectively. This is because the fibers over the interiors of $D_{\mathfrak{r}}$ and $A \cup D_{\mathfrak{s}}$ do not break over $L_{\mathfrak{r}}$. Hence the closures of the essential curves of $\mathbf{F}(\mathbf{v})$ are disjoint from $L_{\mathfrak{r}}$, since the interiors project to $S_{0} \sd (C_{\mathfrak{r}} \cup C_{\mathfrak{s}})$ and the boundaries of the compactifications lie on $L_{\mathfrak{r}}(\mathbf{v})$. By continuity, this continues to hold if we allow $\vert a_{1} \vert \neq 0$ but sufficiently small with respect to $\vert b_{1} \vert$, and apply Lemma \ref{l:fukaya1} to obtain a further deformation. It follows that the essential curves of $\mathbf{F}(\mathbf{v})$ are $J_{\mathbf{V}}$-holomorphic whenever $J_{\mathbf{V}}$ differs from $J_{\mathbf{v}}$ in a small enough neighborhood of $L_{\mathfrak{r}}$. On the other hand, top level curves of $\mathbf{F}(\mathbf{v})$ that cover planes in broken leaves intersect $\mathcal{U}(L_{\mathfrak{r}})$ in straightened cylinders lying in $\{p_{1} = a_{1}\}$. Since these cylinders are holomorphic for any almost complex structure adapted to the parameterizations, the result follows.
\end{proof}

We now describe our second deformation procedure. Fix translation data $\mathbf{V} = \{\mathbf{0}, \mathbf{v}_{1}, \mathbf{v}_{2}\} = \{(0, 0), (a_{1}, b_{1}), (a_{2}, b_{2})\}$ with $b_{1}$ and $b_{2}$ nonzero and $\vert a_{1} \vert $ and $\vert a_{2} \vert$ sufficiently small. We will deform the essential plane $u_{\mathfrak{r}}$ of $\mathbf{F}$ into a plane $\smash{u_{\mathfrak{r}, \mathbf{V}}}$ in $X \sd (L_{\mathfrak{r}}(\mathbf{V}) \cup L_{\mathfrak{s}})$, which will be $J_{\mathbf{V}}$-holomorphic for some $J_{\mathbf{V}}$ in $\mathcal{J}_{\mathbf{V}}$ and which will have the same asymptotic orbit as $u_{\mathfrak{r}}$. The difficulty lies in showing that this plane does not degenerate in the course of the above deformation, and this behavior is ruled out in the proof of the following lemma, which combines Lemma 3.31 and Corollary 3.32 in \cite{HK}.

\begin{lem}\label{l:deform2}
If $b_{1}$ and $b_{2}$ are nonzero and $\vert a_{1} \vert$ and $\vert a_{2} \vert$ are sufficiently small, there is a $J_{\mathbf{V}}$-holomorphic plane $\smash{u_{\mathfrak{r}, \mathbf{V}}}$ in $X \sd (L_{\mathfrak{r}}(\mathbf{V}) \cup L_{\mathfrak{s}})$, in the same class as $u_{\mathfrak{r}}$. This plane is essential and disjoint from the region $\{p_{1} > 0\}$. The closure of the image of $\pi \circ u_{\mathfrak{r}, \mathbf{V}}$ is $D_{\mathfrak{r}}$.
\end{lem}

\begin{proof}
Once again, the general case follows from the case $\mathbf{v}_{2} = \mathbf{0}$ and $a_{1} = 0$ by continuity and Lemma \ref{l:fukaya1}, so suppose that $0 < \vert b_{1} \vert < \epsilon$. Let $\{J_{t}\}_{t \in [0, 1]}$ be a smooth family of almost complex structures in $\mathcal{J}_{\mathbf{V}}$ with $J_{0} = J$, such that $J_{t}$ is adapted to $\psi_{\mathfrak{r}}$ for all $t \in [0, 1)$ and $J_{1}$ is adapted to $\psi_{\mathfrak{r}}(\mathbf{V})$. We study the deformations of the top level curve $u_{\mathfrak{r}}$ as $t \to 1$. The monotonicity of $L_{\mathfrak{r}}$ shows that the deformed plane $u_{\mathfrak{r}}(t)$ persists for all $t \in [0, 1)$, and we must show that this remains true for $t = 1$. Arguing by contradiction, assume that for some sequence $t_{i} \to 1$ as $i \to + \infty$, the sequence of curves $u_{\mathfrak{r}}(t_{i})$ converges to a nontrivial pseudoholomorphic building $\mathbf{H}$ with curves that have punctures on $L_{\mathfrak{r}}(\mathbf{v}_{1})$. The contradiction will be obtained by showing that none of the curves of $\mathbf{H}$ intersect $T_{0}$. Start by fixing a component $v$ of $\mathbf{H}$ with an asymptotic orbit covering a geodesic in the class $(k, l) \in H_{1}(L_{\mathfrak{r}}(\mathbf{v}_{1}); \psi_{\mathfrak{r}}(\mathbf{v}_{1}))$. The rest of the proof will proceed in four steps.

\vspace{2mm}

\noindent\textsc{Step 1:} \emph{We show that $k \leq 0$.} To see this, first observe that the closure of the image of $\pi \circ u_{\mathfrak{r}}$ is $D_{\mathfrak{r}}$ by definition, and consequently, the plane $u_{\mathfrak{r}}$ is disjoint from all the leaves of $\mathcal{F}$ that intersect $\mathcal{U}(L_{\mathfrak{r}})$ in the region $\{p_{1} > 0\}$. This continues to be true for $u_{\mathfrak{r}}(t)$ when $t < 1$, and by continuity, it follows that the same must be true of $v$. Near the relevant puncture, the curve $v$ can be smoothly compactified so that the resulting disk has boundary described by the map
\[
\theta \mapsto (0, b_{1}, q_{1} + k\theta, q_{2} + l\theta)
\]
for some $q_{1}, q_{2} \in \SS^{1}$. (We recall that we are assuming $a_{1} = 0$, and that the boundary is a geodesic on a flat torus, which allows us to deduce the above form.) Now observe that the tangent space at any point on the above boundary is spanned by the set $\{k\partial_{q_{1}} + l\partial_{q_{2}}, k\partial_{p_{1}} + l\partial_{p_{2}}\}$. This means that if $k$ were positive, the curve $v$ would have to intersect the region $\{p_{1} > 0\}$, a contradiction.

\vspace{2mm}

\noindent\textsc{Step 2:} \emph{If $k = 0$, we show that $v$ must cover a plane or cylinder in a broken leaf of $\mathcal{F}(\mathbf{V})$.} By a cylinder of a broken leaf, we mean a cylinder in a middle or bottom level with asymptotic orbits matching those of the planes in a broken leaf. This follows from the same arguments as Lemma 6.2 in \cite{HL}. The idea is that the asymptotic properties of pseudoholomorphic curves together with the fact that $v$ lies in the region $\{p_{1} \leq 0\}$ imply that if $v$ does not cover a plane or cylinder in a broken leaf, then it must intersect all nearby leaves of the foliation, \emph{including} those that lie in the region $\{p_{1} > 0\}$, which is impossible. The argument relies on Siefring's intersection theory for asymptotically cylindrical curves, see for instance \cite{We2}.

\vspace{2mm}

\noindent\textsc{Step 3:} \emph{We obtain the desired contradiction.} We partition the curves of $\mathbf{H}$ into the following three subcollections:
\begin{itemize}
\item The collection $\mathbf{H}_{\mathrm{Top}}$ of top level curves.
\item The subbuilding $\mathbf{H}_{\mathbf{0}}$ consisting of those middle and bottom level curves mapping to copies of $\RR \times S^{\ast}L_{\mathfrak{r}}$ and $T^{\ast}L_{\mathfrak{r}}$.
\item The subbuilding $\mathbf{H}_{\mathbf{v}}$ consisting of those middle and bottom level curves mapping to copies of $\RR \times S^{\ast}L_{\mathfrak{r}}(\mathbf{v}_{1})$ and $T^{\ast}L_{\mathfrak{r}}(\mathbf{v}_{1})$.
\end{itemize}
The top level curves in $\mathbf{H}_{\mathrm{Top}}$ with punctures on $L_{\mathfrak{r}}(\mathbf{v}_{1})$ all compactify to surfaces with boundary circles giving rise to classes $(k_{1}, l_{1}), \dots, (k_{m}, l_{m})$ in $H_{1}(L_{\mathfrak{r}}(\mathbf{v}_{1}), \psi_{\mathfrak{r}}(\mathbf{v}_{1}))$. These classes must sum to zero, since they constitute the boundary of the cycle obtained by compactifying all of the curves in $\mathbf{H}_{\mathbf{v}_{1}}$. It then follows from Step 1 that $k_{i} = 0$ for $i = 1, \dots, m$. Step 2 in turn implies that each curve in $\mathbf{H}$ with a puncture on $L_{\mathfrak{r}}(\mathbf{v}_{1})$ must cover a plane or cylinder in a broken leaf of $\mathcal{F}(\mathbf{v}_{1})$. Now partition the curves of $\mathbf{H}_{\mathrm{Top}} \cup \mathbf{H}_{\mathbf{v}_{1}}$ into subcollections based on the matchings of their asymptotic orbits in $\RR \times S^{\ast}L_{\mathfrak{r}}(\mathbf{v}_{1})$ and $T^{\ast}L_{\mathfrak{r}}(\mathbf{v}_{1})$, and denote these subcollections by $\mathbf{H}_{1}, \dots, \mathbf{H}_{n}$. The compactifications of these subcollections represent elements of $\pi_{2}(X, L_{\mathfrak{r}})$, and must have therefore have positive integral area. But since $u_{\mathfrak{r}}$ had area $1$, we must have $n = 1$, and the area of the compactification of $\mathbf{H}_{1}$ is $1$. If we are assuming that $\mathbf{H}$ includes curves asymptotic to $L_{\mathfrak{r}}(\mathbf{v}_{1})$, then all curves of $\mathbf{H}_{1}$ have punctures on $L_{\mathfrak{r}}(\mathbf{v}_{1})$, and the above discussion implies that each such curve must cover a plane or cylinder in a broken leaf of $\mathcal{F}(\mathbf{V})$ through $L_{\mathfrak{r}}(\mathbf{v}_{1})$. Since none of these leaves intersect $T_{0}$, as do none of the curves of $\mathbf{H}_{\mathbf{0}}$, we see that none of the curves of $\mathbf{H}$ intersect $T_{0}$, yielding the desired contradiction.

\vspace{2mm}

\noindent\textsc{Step 4:} \emph{We finish the proof of the lemma.} It remains to prove the other assertions of the lemma. First, setting $\smash{u_{\mathfrak{r}, \mathbf{V}}} = u_{\mathfrak{r}}(1)$, we see that this plane is disjoint from the region $\{p_{1} > 0\}$ since $u_{\mathfrak{r}}$ is and no new intersections can appear in the deformation process. Furthermore, since $\smash{u_{\mathfrak{r}, \mathbf{V}}}$ does not cover a leaf of the foliation, it follows from positivity of intersections that it is disjoint from the set $\{p_{1} = 0\}$, since otherwise it would intersect the region $\{p_{1} > 0\}$ as well. This implies that the closure of the image of $\pi \circ \smash{u_{\mathfrak{r}, \mathbf{V}}}$ is $D_{\mathfrak{r}}$, and completes the proof of the lemma.
\end{proof}

\subsection{Intersection Arguments}

We now use the deformation procedures described in the previous subsection to obtain curves in $X \sd (L_{\mathfrak{r}} \cup L_{\mathfrak{s}})$ with desirable intersection properties. The analysis of such intersections can again be localized near either of $L_{\mathfrak{r}}$ or $L_{\mathfrak{s}}$, and we will focus primarily on intersections near the former, only giving an account of the full analysis in Subsection 7.4 below. As before, we will fix translation data
\[
\mathbf{V} = \{\mathbf{0}, \mathbf{v}_{1}, \mathbf{v}_{2}\} = \{(0, 0), (a_{1}, b_{1}), (a_{2}, b_{2})\}
\]
and will always assume that $\mathbf{v}_{1}$ and $\mathbf{v}_{2}$ are distinct. If $\Vert \mathbf{v}_{1} \Vert$ is sufficiently small so that Lemma \ref{l:fukaya1} and Lemma \ref{l:deform1} apply, the deformed building $\mathbf{F}(\mathbf{v}_{1})$ is well defined and its top level curves
\[
\{u_{\mathfrak{r}}(\mathbf{v}_{1}), \underline{u}(\mathbf{v}_{1}), u_{\mathfrak{s}}, u_{1}(\mathbf{v}_{1}), \dots, u_{d - 1}(\mathbf{v}_{1}), \mathfrak{u}_{1}, \dots, \mathfrak{u}_{d - 1}\}
\]
are all $J_{\mathbf{V}}$-holomorphic for some $J_{\mathbf{V}}$ in $\mathcal{J}_{\mathbf{V}}$. Similarly, we choose $\Vert \mathbf{v}_{2} \Vert$ sufficiently small so that the deformed building $\mathbf{G}(\mathbf{v}_{2})$ is well defined. Its top level curves are 
\[
\{v_{\mathfrak{r}}(\mathbf{v}_{2}), \underline{v}(\mathbf{v}_{2}), v_{\mathfrak{s}}, v_{1}(\mathbf{v}_{2}), \dots, v_{d - 2}(\mathbf{v}_{2}), \mathfrak{v}_{1}, \dots, \mathfrak{v}_{d - 2}\}.
\]
We also choose $b_{1}$ and $b_{2}$ nonzero and $\vert a_{1} \vert$ and $\vert a_{2} \vert$ sufficiently small so that Lemma \ref{l:deform2} yields a $J_{\mathbf{V}}$-holomorphic plane $\smash{u_{\mathfrak{r}, \mathbf{V}}}$ in $X \sd (L_{\mathfrak{r}}(\mathbf{V}) \cup L_{\mathfrak{s}})$ that is in the same class as $u_{\mathfrak{r}}$. This plane is essential and disjoint from the region $\{p_{1} > 0\}$. The closure of the image of $\pi \circ \smash{u_{\mathfrak{r}, \mathbf{V}}}$ is $D_{\mathfrak{r}}$, and so the plane intersects the leaves of $\mathcal{F}(\mathbf{V})$ passing through the sets $\{q_{1} = \theta, p_{1} = \delta\}$ with $\delta < 0$ exactly once, in view of our description of leaves in $\mathcal{U}(L_{\mathfrak{r}})$.

The intersection number between each top level curve of $\mathbf{F}(\mathbf{v}_{1})$ or $\mathbf{G}(\mathbf{v}_{2})$ and the curve $\smash{u_{\mathfrak{r}, \mathbf{V}}}$ is well defined since the curves of $\mathbf{F}(\mathbf{v}_{1})$ and $\mathbf{G}(\mathbf{v}_{2})$ are all disjoint from $L_{\mathfrak{r}}$. We denote the \emph{sum total} of these intersection numbers by $\mathbf{F}(\mathbf{v}_{1}) \ast \smash{u_{\mathfrak{r}, \mathbf{V}}}$ and $\mathbf{G}(\mathbf{v}_{2}) \ast \smash{u_{\mathfrak{r}, \mathbf{V}}}$. Similarly, the intersection number between each top level curve of $\mathbf{F}(\mathbf{v}_{1})$ or $\mathbf{G}(\mathbf{v}_{2})$ and any of the planes in $\mathfrak{r}_{0}$ or $\mathfrak{r}_{\infty}$ is well defined and all such intersections are positive. These intersection numbers are denoted by
\[
\mathbf{F}(\mathbf{v}_{1}) \ast \mathfrak{r}_{0}, \qquad \mathbf{F}(\mathbf{v}_{1}) \ast \mathfrak{r}_{\infty}, \qquad \mathbf{G}(\mathbf{v}_{2}) \ast \mathfrak{r}_{0}, \qquad \mathbf{G}(\mathbf{v}_{2}) \ast \mathfrak{r}_{\infty}.
\]
This notation is sensible since the intersection number is independent of the choice of plane in $\mathfrak{r}_{0}$ or $\mathfrak{r}_{\infty}$. We then let $F$ and $G$ be the smooth deformed compactifications of $\mathbf{F}(\mathbf{v}_{1})$ and $\mathbf{G}(\mathbf{v}_{2})$, respectively, described in the previous subsection. We also let $E_{\mathfrak{r}} \colon (\DD^{2}, \SS^{1}) \to (X, L_{\mathfrak{r}})$ be the smooth compactification of the plane $\smash{\smash{u_{\mathfrak{r}, \mathbf{V}}}}$. Finally, we let $\smash{\overline{\mathfrak{r}}_{0}}$ and $\smash{\overline{\mathfrak{r}}_{\infty}}$ be the solid tori obtained by compactifying the planes of $\mathfrak{r}_{0}$ and $\mathfrak{r}_{\infty}$, respectively. This is carefully proved in Proposition 5.16 of \cite{DRGI}. It is then clear that
\[
F \cdot E_{\mathfrak{r}} = \mathbf{F}(\mathbf{v}_{1}) \ast \smash{u_{\mathfrak{r}, \mathbf{V}}}, \qquad F \ast \overline{\mathfrak{r}}_{0} = \mathbf{F}(\mathbf{v}_{1}) \ast \mathfrak{r}_{0}, \qquad F \ast \overline{\mathfrak{r}}_{\infty} = \mathbf{F}(\mathbf{v}_{1}) \ast \mathfrak{r}_{\infty},
\]
and similarly for $G$ and $\mathbf{G}(\mathbf{v}_{2})$. Recall that $\alpha_{0}$ is the number of top level curves of $\mathbf{F}$ lying in $\mathfrak{r}_{0}$, so that there are $d - 1 - \alpha_{0}$ top level curves lying in $\mathfrak{r}_{\infty}$. Likewise, recall that $\gamma_{0}$ is the number of top level curves of $\mathbf{G}$ lying in $\mathfrak{r}_{0}$, so that there are $d - 2 - \gamma_{0}$ top level curves lying in $\mathfrak{r}_{\infty}$. We then have the following result, which combines Lemma 3.33, Lemma 3.34, and Lemma 3.35 in \cite{HK}.

\begin{lem}\label{l:intersect1}
Suppose that $a_{1}$ and $a_{2}$ are negative, that $b_{1}$ and $b_{2}$ are nonzero, and that $\vert a_{1} \vert$ and $\vert a_{2} \vert$ are sufficiently small with respect to $\vert b_{1} \vert$ and $\vert b_{2} \vert$. If $b_{1} > 0$, we have
\[
F \ast \overline{\mathfrak{r}}_{0} = 0, \qquad F \ast \overline{\mathfrak{r}}_{\infty} = 1, \qquad F \cdot E_{\mathfrak{r}} = \alpha_{0}.
\]
On the other hand, if $b_{1} < 0$, then we have
\[
F \ast \overline{\mathfrak{r}}_{0} = 1, \qquad F \ast \overline{\mathfrak{r}}_{\infty} = 0, \qquad F \cdot E_{\mathfrak{r}} = d - 1 - \alpha_{0}.
\]
Similarly, if $b_{2} > 0$, we have
\[
G \ast \overline{\mathfrak{r}}_{0} = 0, \qquad G \ast \overline{\mathfrak{r}}_{\infty} = 1, \qquad G \cdot E_{\mathfrak{r}} = \gamma_{0} + v_{\mathfrak{r}}(\mathbf{v}_{2}) \cdot \smash{u_{\mathfrak{r}, \mathbf{V}}}.
\]
On the other hand, if $b_{2} < 0$, then we have
\[
G \ast \overline{\mathfrak{r}}_{0} = 1, \qquad G \ast \overline{\mathfrak{r}}_{\infty} = 0, \qquad G \cdot E_{\mathfrak{r}} = d - 2 - \gamma_{0} + v_{\mathfrak{r}}(\mathbf{v}_{2}) \cdot \smash{u_{\mathfrak{r}, \mathbf{V}}}.
\]
Furthermore, when $b_{1}$ and $b_{2}$ have opposite signs, then
\[
v_{\mathfrak{r}}(\mathbf{v}_{2}) \cdot \smash{u_{\mathfrak{r}, \mathbf{V}}} = v_{\mathfrak{r}}(\mathbf{v}_{2}) \cdot u_{\mathfrak{r}}(\mathbf{v}_{1}).
\]
\end{lem}

\begin{proof}
We will first prove the statements concerning $F$, focusing on when $b_{1} > 0$ since the proof of the other case is nearly identical. First, since the disks in $\overline{\mathfrak{r}}_{0} \cup \overline{\mathfrak{r}}_{\infty}$ fit together to form spheres in the class $(0, 1)$, and since $F$ is in the class $(1, d)$, we have
\[
F \ast \overline{\mathfrak{r}}_{0} + F \ast \overline{\mathfrak{r}}_{\infty} = 1.
\]
We will now show that $F \ast \overline{\mathfrak{r}}_{\infty} \geq 1$, which by positivity of intersections implies that $F \ast \overline{\mathfrak{r}}_{\infty} = 1$ and $F \ast \overline{\mathfrak{r}}_{0} = 0$. It is sufficient to show that $\underline{u}(\mathbf{v}_{1})$ intersects a plane of $\mathfrak{r}_{\infty}$ at least once. To see this, observe that $\underline{u}(\mathbf{v}_{1})$ is essential and projects under $\pi$ to the region of $S_{0}$ bounded by $\pi(L_{\mathfrak{r}}(\mathbf{v}_{1}))$ and $\pi(L_{\mathfrak{s}})$. It therefore intersects $\mathcal{U}(L_{\mathfrak{r}})$ in the region $\{p_{1} > a_{1}\}$ and intersects all leaves of the foliation that meet $\mathcal{U}(L_{\mathfrak{r}})$ in this set. By making $\mathcal{U}(L_{\mathfrak{r}})$ sufficiently small, we can also assume that these intersections lie inside the region $\{p_{2} > 0\}$. This holds when $a_{1} = 0$ since $b_{1} > 0$ and remains true for small $\vert a_{1} \vert \neq 0$ by continuity. But the planes of $\mathfrak{r}_{\infty}$ are exactly those intersecting $\mathcal{U}(L_{\mathfrak{r}})$ in the region $\{p_{1} = 0, p_{2} > 0\}$, so the desired conclusion follows immediately.

We next need to show that $F \cdot E_{\mathfrak{r}} = \alpha_{0}$. We do this by showing that $\smash{u_{\mathfrak{r}, \mathbf{V}}}$ intersects $u_{i}(\mathbf{v}_{1})$ exactly once for $i = 1, \dots, \alpha_{0}$, and that $\smash{u_{\mathfrak{r}, \mathbf{V}}}$ is disjoint from all other top level curves of $\mathbf{F}(\mathbf{v}_{1})$. Right away, since the closure of the image of $\pi \circ \smash{u_{\mathfrak{r}, \mathbf{V}}}$ is $D_{\mathfrak{r}}$, any top level curves of $\mathbf{F}(\mathbf{v}_{1})$ that project into different regions of $S_{0}$ will not intersect it. This holds in particular of $u_{\mathfrak{s}}$ as well as $\mathfrak{u}_{i}$ for $i = 1, \dots, d - 1$. All remaining intersections will be analyzed in the following three steps.

\vspace{2mm}

\noindent\textsc{Step 1:} \emph{Intersections with the $u_{i}(\mathbf{v}_{1})$.} Since $\smash{u_{\mathfrak{r}, \mathbf{V}}}$ is essential with respect to $\mathcal{F}$, it intersects each of its leaves either once or not at all. It also intersects $\mathcal{U}(L_{\mathfrak{r}})$ in the region $\{p_{1} < 0\}$ and compactifies to a disk with boundary on $L_{\mathfrak{r}} = \{p_{1} = p_{2} = 0\}$. The condition $b_{1} > 0$ then implies that for all $a_{1} < 0$ with $\vert a_{1} \vert$ sufficiently small with respect to $b_{1}$, the plane $\smash{u_{\mathfrak{r}, \mathbf{V}}}$ intersects cylinders of the form $\{q_{1} = \theta, p_{1} = a_{1}, p_{2} < b_{1}\}$ exactly once. Furthermore, the planes $u_{i}(\mathbf{v}_{1})$ all cover planes in broken leaves of $\mathcal{F}$ that intersect $\mathcal{U}(L_{\mathfrak{r}})$, and for $1 \leq i \leq \alpha_{0}$, they intersect $\mathcal{U}(L_{\mathfrak{r}})$ in cylinders of the form $\{q_{1} = \theta, p_{1} = a_{1}, p_{2} < b_{1}\}$. In this case the plane $\smash{u_{\mathfrak{r}, \mathbf{V}}}$ intersects the leaf of $\mathcal{F}$ containing $u_{i}(\mathbf{v}_{1})$ at a point on the latter plane. On the other hand, for $i > \alpha_{0}$, the planes $u_{i}(\mathbf{v}_{1})$ intersect $\mathcal{U}(L_{\mathfrak{r}})$ in cylinders of the form $\{q_{1} = \theta, p_{1} = a_{1}, p_{2} > b_{1}\}$. In this case, the plane $\smash{u_{\mathfrak{r}, \mathbf{V}}}$ intersects the leaf of $\mathcal{F}$ containing $u_{i}(\mathbf{v}_{1})$ at a point in the complement of the latter plane. It follows that $u_{i}(\mathbf{v}_{1})$ is disjoint from these planes when $i > \alpha_{0}$.

\vspace{2mm}

\noindent\textsc{Step 2:} \emph{Intersections with $\underline{u}(\mathbf{v}_{1})$.} We now show that when $\vert a_{1} \vert$ is sufficiently small with respect to $\vert b_{1} \vert$, the plane $\smash{u_{\mathfrak{r}, \mathbf{V}}}$ is disjoint from $\underline{u}(\mathbf{v}_{1})$. This is obviously true for the portion of the latter curve that lies outside $\mathcal{U}(L_{\mathfrak{r}})$, since its projection under $\pi$ is contained in the interior of $A \cup D_{\mathfrak{s}}$. Now supposing that $a_{1} = 0$, we see that the portions of $\underline{u}(\mathbf{v}_{1})$ and $\smash{u_{\mathfrak{r}, \mathbf{V}}}$ in $\mathcal{U}(L_{\mathfrak{r}})$ are contained in the regions $\{p_{1} > 0\}$ and $\{p_{1} < 0\}$, respectively, and have punctures on disjoint tori. By continuity, this also holds for all $a_{1} < 0$ whenever $\vert a_{1} \vert$ is sufficiently small with respect to $\vert b_{1} \vert$. This proves that $\smash{u_{\mathfrak{r}, \mathbf{V}}}$ is disjoint from $\underline{u}(\mathbf{v}_{1})$ under the above conditions.

\vspace{2mm}

\noindent\textsc{Step 3:} \emph{Intersections with $u_{\mathfrak{r}}(\mathbf{v}_{1})$.} We still need to show that $\smash{u_{\mathfrak{r}, \mathbf{V}}}$ is disjoint from $u_{\mathfrak{r}}(\mathbf{v}_{1})$, at least when $\vert a_{1} \vert$ is sufficiently small with respect to $\vert b_{1} \vert$. From the proof of Lemma \ref{l:deform2}, we see that the compactifications of $\smash{u_{\mathfrak{r}, \mathbf{V}}}$ and $u_{\mathfrak{r}}$ are homotopic in the space of smooth maps $(\DD^{2}, \SS^{1}) \to (\pi^{-1}(D_{\mathfrak{r}}), L_{\mathfrak{r}})$, and so for $\vert a_{1} \vert$ sufficiently small, it is enough to show that $u_{\mathfrak{r}}(\mathbf{v}_{1}) \cdot u_{\mathfrak{r}} = 0$. We claim that this is equivalent to $\mu(\overline{u}_{\mathfrak{r}}) = 2$, and since this latter fact is true by monotonicity of $L_{\mathfrak{r}}$, we are reduced to establishing this equivalence.

To this end, recall that $\mu(\overline{u}_{\mathfrak{r}}) = 2c_{1}(\overline{u}_{\mathfrak{r}})$, where $c_{1}(\overline{u}_{\mathfrak{r}})$ is the relative Chern number, an algebraic count of zeros of a generic section of the line bundle $\Lambda^{2}_{\CC}(\overline{u}_{\mathfrak{r}}^{\ast}TX)$ whose restriction to the boundary is nonvanishing and tangent to the line bundle $\Lambda^{2}(\overline{u}_{\mathfrak{r}}^{\ast}TL_{\mathfrak{r}})$. We will construct such a section explicitly, making use of the decomposition $\overline{u}_{\mathfrak{r}}^{\ast}(TX) = T\DD^{2} \oplus \nu(\overline{u}_{\mathfrak{r}})$, where $\nu(\overline{u}_{\mathfrak{r}})$ is the normal bundle to the embedding $\overline{u}_{\mathfrak{r}}$. Fixing polar coordinates $(r, \theta)$ on $\DD^{2}$, we view $r\partial_{\theta}$ as a section of $\overline{u}_{\mathfrak{r}}^{\ast}(TX)$ under the above decomposition. Its restriction to the boundary is nonvanishing and tangent to $\Lambda^{2}(\overline{u}_{\mathfrak{r}}^{\ast}TL_{\mathfrak{r}})$.

Making the translation vector $\mathbf{v}_{1}$ smaller if necessary, we can assume that $\overline{u}_{\mathfrak{r}}(\mathbf{v}_{1})$ is sufficiently $C^{1}$-close to $\overline{u}_{\mathfrak{r}}$ so as to be identifiable with a section $\sigma$ of $\nu(\overline{u}_{\mathfrak{r}})$, which we view as a section of $\overline{u}_{\mathfrak{r}}^{\ast}(TX)$. Through a continuous rotation in $\nu(\overline{u}_{\mathfrak{r}})$, the restriction of $\sigma$ to the boundary is homotopic through nonvanishing sections to a section of $\overline{u}_{\mathfrak{r}}^{\ast}TL_{\mathfrak{r}}$ that is orthogonal to $\partial_{\theta}$. The above discussion then implies that the section $r\partial_{\theta} \wedge \sigma$ of $\Lambda^{2}_{\CC}(\overline{u}_{\mathfrak{r}}^{\ast}TX)$ is nonvanishing and tangent to $\Lambda^{2}(\overline{u}_{\mathfrak{r}}^{\ast}TL_{\mathfrak{r}})$ when restricted to the boundary. Furthermore, its zeros correspond to the union of the zeros of $r\partial_{\theta}$ and $\sigma$. Since $\overline{u}_{\mathfrak{r}}$ is embedded, the algebraic count of zeros of $\sigma$ is precisely the intersection number $u_{\mathfrak{r}}(\mathbf{v}_{1}) \cdot u_{\mathfrak{r}}$. It now follows that
\[
2 = \mu(\overline{u}_{\mathfrak{r}}) = 2c_{1}(r\partial_{\theta} \wedge \sigma) = 2(1 + u_{\mathfrak{r}}(\mathbf{v}_{1}) \cdot u_{\mathfrak{r}}),
\]
establishing the equivalence and implying that $\smash{u_{\mathfrak{r}, \mathbf{V}}}$ is disjoint from $u_{\mathfrak{r}}(\mathbf{v}_{1})$.

\vspace{2mm}

The statements in the lemma concerning $G$ are proved in an identical manner, except that the term $v_{\mathfrak{r}}(\mathbf{v}_{2}) \cdot \smash{u_{\mathfrak{r}, \mathbf{V}}}$ need not vanish. Instead, if $b_{1}$ and $b_{2}$ have opposite signs, and assuming momentarily that $a_{1} = a_{2} = 0$, we observe that the compactification of the plane $v_{\mathfrak{r}}(\mathbf{v}_{2})$ projects onto $D_{\mathfrak{r}}$ and has boundary on $L_{\mathfrak{r}}(\mathbf{v}_{2})$. The assumption on $b_{1}$ and $b_{2}$ together with an examination of the proof of Lemma 3.27 in \cite{HK} (which is our Lemma \ref{l:fukaya1}) shows that the compactifications of $u_{\mathfrak{r}}$ and $u_{\mathfrak{r}}(\mathbf{v}_{1})$ are connected in an appropriate space of maps of pairs. Since this is also true of $u_{\mathfrak{r}}$ and $\smash{u_{\mathfrak{r}, \mathbf{V}}}$, invariance of intersection numbers under homotopies implies that
\[
v_{\mathfrak{r}}(\mathbf{v}_{2}) \cdot \smash{u_{\mathfrak{r}, \mathbf{V}}} = v_{\mathfrak{r}}(\mathbf{v}_{2}) \cdot u_{\mathfrak{r}} = v_{\mathfrak{r}}(\mathbf{v}_{2}) \cdot u_{\mathfrak{r}}(\mathbf{v}_{1}).
\]
The general case now follows by continuity when $\vert a_{1} \vert$ and $\vert a_{2} \vert$ are chosen to be sufficiently small. This proves the lemma.
\end{proof}

To complete our intersection analysis, we will seek to understand the intersection numbers between the top level curves of $\mathbf{F}(\mathbf{v}_{1})$ and those of $\mathbf{G}(\mathbf{v}_{2})$. These are always well defined when $\mathbf{v}_{1}$ and $\mathbf{v}_{2}$ are distinct and when $\vert b_{1} \vert \neq \vert b_{2} \vert$ and $\vert a_{1} \vert \neq \vert a_{2} \vert$ are sufficiently small. We then have the following lemma, which is Lemma 3.36 in \cite{HK}.

\begin{lem}\label{l:intersect2}
Suppose that $a_{1} < a_{2} < 0$, with $\vert a_{1} \vert$ and $\vert a_{2} \vert$ sufficiently small. If $b_{1} > b_{2}$, then
\begin{align*}
u_{i}(\mathbf{v}_{1}) \cdot v_{\mathfrak{r}}(\mathbf{v}_{2}) &= 1, \qquad i = 1, \dots, \alpha_{0}, \\
v_{i}(\mathbf{v}_{2}) \cdot \underline{u}(\mathbf{v}_{1}) &= 1, \qquad i = \gamma_{0} + 1, \dots, d - 2.
\end{align*}
On the other hand, if $b_{1} < b_{2}$, then
\begin{align*}
u_{i}(\mathbf{v}_{1}) \cdot v_{\mathfrak{r}}(\mathbf{v}_{2}) &= 1, \qquad i = \alpha_{0} + 1, \dots, d - 1, \\
v_{i}(\mathbf{v}_{2}) \cdot \underline{u}(\mathbf{v}_{1}) &= 1, \qquad i = 1, \dots, \gamma_{0}.
\end{align*}
Moreover, all of the above intersection points project into $D_{\mathfrak{r}}$.
\end{lem}

\begin{proof}
We will only prove that $u_{i}(\mathbf{v}_{1}) \cdot v_{\mathfrak{r}}(\mathbf{v}_{2}) = 1$ for $i = 1, \dots, \alpha_{0}$ when $b_{1} > b_{2}$, leaving the remaining (virtually identical) arguments to the reader. Since the plane $v_{\mathfrak{r}}(\mathbf{v}_{2})$ is essential with respect to $\mathcal{F}$, it suffices to detect a single intersection for some $i$. As observed in the proof of Lemma \ref{l:intersect1}, the planes $u_{i}(\mathbf{v}_{1})$ in this range intersect $\mathcal{U}(L_{\mathfrak{r}})$ in cylinders of the form $\{q_{1} = \theta, p_{1} = a_{1}, p_{2} < b_{1}\}$. Assuming momentarily that $a_{2} = 0$, we know that $v_{\mathfrak{r}}(\mathbf{v}_{2})$ has a puncture on $L_{\mathfrak{r}}(\mathbf{v}_{2})$, and so intersects $u_{i}(\mathbf{v}_{1})$ when $\vert a_{1} \vert$ is sufficiently small since its compactification has boundary intersecting all cylinders of the form $\{q_{1} = \theta, p_{1} = 0, p_{2} < b_{1}\}$. By continuity, this holds whenever $\vert a_{2} \vert  \neq 0$ is sufficiently small. Furthermore, the corresponding intersection number must be equal to $1$ by positivity of intersections, and since $a_{1} < 0$, the intersection point must project into $D_{\mathfrak{r}}$.
\end{proof}

The following corollary is immediate from the previous lemma.

\begin{cor}\label{c:intersect2}
If $b_{1} > b_{2}$, then $F \cap G$ contains at least $\alpha_{0} + d - 2 - \gamma_{0}$ points in $\mathcal{U}(L_{\mathfrak{r}})$ that project into $D_{\mathfrak{r}}$. On the other hand, if $b_{1} < b_{2}$, then $F \cap G$ contains at least $d - 1 - \alpha_{0} + \gamma_{0}$ points in $\mathcal{U}(L_{\mathfrak{r}})$ that project into $D_{\mathfrak{r}}$.
\end{cor}

It follows from Lemma \ref{l:intersect1} that when $b_{1}$ and $b_{2}$ have opposite signs, any excess intersection points between $F$ and $G$ in $\pi^{-1}(D_{\mathfrak{r}})$ beyond what is described by Corollary \ref{c:intersect2} correspond to intersections between $G$ and $E_{\mathfrak{r}}$.

\subsection{Proof of the Black Box}

We finally prove Proposition \ref{p:black}, using the techniques and analysis of the previous two subsections. However, up until this point we have primarily localized our considerations near $L_{\mathfrak{r}}$, and the proof of the proposition will require applying these considerations near both of $L_{\mathfrak{r}}$ and $L_{\mathfrak{s}}$. Fortunately, the brunt of the work has already been done, and the changes that are needed for the more general case are primarily notational. We shall now proceed to outline these changes, and begin by fixing translation data
\begin{align*}
\mathbf{x}_{1} = \{\mathbf{v}_{1}, \mathbf{w}_{1}\} &= \{(a_{1}, b_{1}), (c_{1}, d_{1})\}, \\
\mathbf{x}_{2} = \{\mathbf{v}_{2}, \mathbf{w}_{2}\} &= \{(a_{2}, b_{2}), (c_{2}, d_{2})\}, \\
\mathbf{V} = \{\mathbf{0}, \mathbf{v}_{1}, \mathbf{v}_{2}\} &= \{(0, 0), (a_{1}, b_{1}), (a_{2}, b_{2})\}, \\
\mathbf{W} = \{\mathbf{0}, \mathbf{w}_{1}, \mathbf{w}_{2}\} &= \{(0, 0), (c_{1}, d_{1}), (c_{2}, d_{2})\}.
\end{align*}
We then set $\mathbf{X} = \{\mathbf{V}, \mathbf{W}\}$. An arbitrary almost complex structure in the intersection $\mathcal{J}_{\mathbf{V}} \cap \mathcal{J}_{\mathbf{W}}$ will be denoted by $J_{\mathbf{X}}$. We then have the following analog of Lemma \ref{l:fukaya1}.

\begin{lem}\label{l:fukaya2}
Let $u$ be a regular $J$-holomorphic curve in $X \sd (L_{\mathfrak{r}} \cup L_{\mathfrak{s}})$ with $m \geq 0$ punctures on $L_{\mathfrak{r}}$ and $n \geq 0$ punctures on $L_{\mathfrak{s}}$.
\begin{enumerate}
\item[(a)] For any translation data $\mathbf{x} = \{\mathbf{v}, \mathbf{w}\}$ with $\Vert \mathbf{x} \Vert$ sufficiently small, there is a regular $J_{\mathbf{x}}$-holomorphic curve $u(\mathbf{x})$ with $m$ punctures on $L_{\mathfrak{r}}(\mathbf{v})$ and $n$ punctures on $L_{\mathfrak{s}}(\mathbf{w})$.
\item[(b)] The asymptotic orbits of $u(\mathbf{x})$ cover geodesics that have the same representatives in $\smash{H_{1}(L_{\mathfrak{r}}(\mathbf{v}); \psi_{\mathfrak{r}}(\mathbf{v}))}$ and $H_{1}(L_{\mathfrak{s}}(\mathbf{w}); \psi_{\mathfrak{s}}(\mathbf{w}))$ as the geode\-sics covered by orbits of $u$ have in $H_{1}(L_{\mathfrak{r}}; \psi_{\mathfrak{r}})$ and $H_{1}(L_{\mathfrak{s}}; \psi_{\mathfrak{s}})$, respectively.
\end{enumerate}
\end{lem}

We also have the following analogue of Lemma \ref{l:deform2}.

\begin{lem}\label{l:deform3}
Let $b_{1}$, $b_{2}$, $d_{1}$, and $d_{2}$ be nonzero and let $\vert a_{1} \vert$, $\vert a_{2} \vert$, $\vert c_{1} \vert$, and $\vert c_{2} \vert$ be sufficiently small.
\begin{enumerate}
\item There is a $J_{\mathbf{X}}$-holomorphic plane $\smash{u_{\mathfrak{r}, \mathbf{X}}}$ in $X \sd (L_{\mathfrak{r}}(\mathbf{V}) \cup L_{\mathfrak{s}}(\mathbf{W}))$, in the same class as $u_{\mathfrak{r}}$. This plane is essential and disjoint from the region $\{p_{1} > 0\}$. The closure of the image of $\pi \circ \smash{u_{\mathfrak{r}, \mathbf{X}}}$ is $D_{\mathfrak{r}}$.
\item There is a $J_{\mathbf{X}}$-holomorphic plane $\smash{u_{\mathfrak{s}, \mathbf{X}}}$ in $X \sd (L_{\mathfrak{r}}(\mathbf{V}) \cup L_{\mathfrak{s}}(\mathbf{W}))$, in the same class as $u_{\mathfrak{s}}$. This plane is essential and disjoint from the region $\{P_{1} < 0\}$. The closure of the image of $\pi \circ \smash{u_{\mathfrak{s}, \mathbf{X}}}$ is $D_{\mathfrak{s}}$.
\end{enumerate}
\end{lem}

We will assume that all the parameters in the translation data $\mathbf{X}$ are chosen appropriately so that the deformed buildings $\mathbf{F}(\mathbf{x}_{1})$ and $\mathbf{G}(\mathbf{x}_{2})$ are well defined, with smoothly deformed compactifications $F$ and $G$. We also choose the above parameters so that Lemma \ref{l:deform3} yields $J_{\mathbf{X}}$-holomorphic planes $\smash{u_{\mathfrak{r}, \mathbf{X}}}$ and $\smash{u_{\mathfrak{s}, \mathbf{X}}}$ with compactifications $E_{\mathfrak{r}}$ and $E_{\mathfrak{s}}$, respectively. The boundaries of these compactifications belong to the same homology classes as the boundaries of the compactifications of $u_{\mathfrak{r}}$ and $u_{\mathfrak{s}}$, respectively.

We first want to justify why $F$ and $G$ can be replaced by smooth symplectic spheres without affecting any of the upcoming intersection analysis. To do this, we show that these symplectic spheres can be obtained by modifying $F$ and $G$ away from all neighborhoods of Lagrangian tori where this analysis will take place, namely all translated tori. First, we observe that the asymptotic orbits of all top level curves of $\mathbf{F}(\mathbf{x}_{1})$ and $\mathbf{G}(\mathbf{x}_{2})$ are simply covered and generically distinct. For small enough perturbations, we can then assume that the restrictions of  these curves to neighborhoods of the Lagrangian tori are symplectically isotopic to the corresponding curves of $\mathbf{F}$ and $\mathbf{G}$. Since these latter buildings are limits of smooth embedded pseudoholomorphic spheres, we can, after another small perturbation, assume that the top level curves of $\mathbf{F}$ and $\mathbf{G}$, restricted to a compact subset of $X \sd (L_{\mathfrak{r}} \cup L_{\mathfrak{s}})$, extend to smooth symplectic spheres in $X$. The symplectic replacements for $F$ and $G$ are obtained by combining the isotopies with these extensions.

At this point, observe that properties (1), (2), and (3) of Proposition \ref{p:black} follow immediately. It therefore remains to verify properties (4) through (9), and the following lemmas, which are straightforward generalizations of those in the previous subsection, will be needed.

\begin{lem}\label{l:intersect3}
Suppose that $a_{1}$ and $a_{2}$ are negative, that $b_{1}$ and $b_{2}$ are nonzero, and that $\vert a_{1} \vert$ and $\vert a_{2} \vert$ are sufficiently small with respect to $\vert b_{1} \vert$ and $\vert b_{2} \vert$. If $b_{1} > 0$, we have
\[
F \ast \overline{\mathfrak{r}}_{0} = 0, \qquad F \ast \overline{\mathfrak{r}}_{\infty} = 1, \qquad F \cdot E_{\mathfrak{r}} = \alpha_{0}.
\]
On the other hand, if $b_{1} < 0$, then we have
\[
F \ast \overline{\mathfrak{r}}_{0} = 1, \qquad F \ast \overline{\mathfrak{r}}_{\infty} = 0, \qquad F \cdot E_{\mathfrak{r}} = d - 1 - \alpha_{0}.
\]
Similarly, if $b_{2} > 0$, we have
\[
G \ast \overline{\mathfrak{r}}_{0} = 0, \qquad G \ast \overline{\mathfrak{r}}_{\infty} = 1, \qquad G \cdot E_{\mathfrak{r}} = \gamma_{0} + v_{\mathfrak{r}}(\mathbf{v}_{2}) \cdot \smash{u_{\mathfrak{r}, \mathbf{X}}}.
\]
On the other hand, if $b_{2} < 0$, then we have
\[
G \ast \overline{\mathfrak{r}}_{0} = 1, \qquad G \ast \overline{\mathfrak{r}}_{\infty} = 0, \qquad G \cdot E_{\mathfrak{r}} = d - 2 - \gamma_{0} + v_{\mathfrak{r}}(\mathbf{v}_{2}) \cdot \smash{u_{\mathfrak{r}, \mathbf{X}}}.
\]
Furthermore, when $b_{1}$ and $b_{2}$ have opposite signs, then
\[
v_{\mathfrak{r}}(\mathbf{x}_{2}) \cdot \smash{u_{\mathfrak{r}, \mathbf{X}}} = v_{\mathfrak{r}}(\mathbf{x}_{2}) \cdot u_{\mathfrak{r}}(\mathbf{x}_{1}).
\]
\end{lem}

\begin{lem}\label{l:intersect4}
Suppose that $a_{1} < a_{2} < 0$ with $\vert a_{1} \vert$ and $\vert a_{2} \vert$ sufficiently small. If $b_{1} > b_{2}$, then $F \cap G$ contains at least $\alpha_{0} + d - 2 - \gamma_{0}$ points in $\mathcal{U}(L_{\mathfrak{r}})$ that project into $D_{\mathfrak{r}}$. On the other hand, if $b_{1} < b_{2}$, then $F \cap G$ contains at least $d - 1 - \alpha_{0} + \gamma_{0}$ points in $\mathcal{U}(L_{\mathfrak{r}})$ that project into $D_{\mathfrak{r}}$.
\end{lem}

\begin{lem}\label{l:intersect5}
Suppose that $c_{1}$ and $c_{2}$ are positive, that $d_{1}$ and $d_{2}$ are nonzero, and that $\vert c_{1} \vert$ and $\vert c_{2} \vert$ are sufficiently small with respect to $\vert d_{1} \vert$ and $\vert d_{2} \vert$. If $d_{1} > 0$, we have
\[
F \ast \overline{\mathfrak{s}}_{0} = 0, \qquad F \ast \overline{\mathfrak{s}}_{\infty} = 1, \qquad F \cdot E_{\mathfrak{s}} = \beta_{0}.
\]
On the other hand, if $d_{1} < 0$, then we have
\[
F \ast \overline{\mathfrak{s}}_{0} = 1, \qquad F \ast \overline{\mathfrak{s}}_{\infty} = 0, \qquad F \cdot E_{\mathfrak{s}} = d - 1 - \beta_{0}.
\]
Similarly, if $d_{2} > 0$, we have
\[
G \ast \overline{\mathfrak{s}}_{0} = 0, \qquad G \ast \overline{\mathfrak{s}}_{\infty} = 1, \qquad G \cdot E_{\mathfrak{s}} = \delta_{0} + v_{\mathfrak{s}}(\mathbf{x}_{2}) \cdot \smash{u_{\mathfrak{s}, \mathbf{X}}}.
\]
On the other hand, if $d_{2} < 0$, then we have
\[
G \ast \overline{\mathfrak{s}}_{0} = 1, \qquad G \ast \overline{\mathfrak{s}}_{\infty} = 0, \qquad G \cdot E_{\mathfrak{s}} = d - 2 - \delta_{0} + v_{\mathfrak{s}}(\mathbf{x}_{2}) \cdot \smash{u_{\mathfrak{s}, \mathbf{X}}}.
\]
Furthermore, when $d_{1}$ and $d_{2}$ have opposite signs, then
\[
v_{\mathfrak{s}}(\mathbf{x}_{2}) \cdot \smash{u_{\mathfrak{s}, \mathbf{X}}} = v_{\mathfrak{s}}(\mathbf{x}_{2}) \cdot u_{\mathfrak{s}}(\mathbf{x}_{1}).
\]
\end{lem}

\begin{lem}\label{l:intersect6}
Suppose that $c_{1} > c_{2} > 0$ with $\vert c_{1} \vert$ and $\vert c_{2} \vert$ sufficiently small. If $d_{1} > d_{2}$, then $F \cap G$ contains at least $\beta_{0} + d - 2 - \delta_{0}$ points in $\mathcal{U}(L_{\mathfrak{s}})$ that project into $D_{\mathfrak{s}}$. On the other hand, if $d_{1} < d_{2}$, then $F \cap G$ contains at least $d - 1 - \beta_{0} + \delta_{0}$ points in $\mathcal{U}(L_{\mathfrak{s}})$ that project into $D_{\mathfrak{s}}$.
\end{lem}

We will now once and for all fix $a_{1} < a_{2} < 0$ and $c_{1} > c_{2} > 0$, all chosen small enough in absolute value so that the above lemmas hold. We can then make choices of the remaining parameters in our translation data according to the following four exhaustive cases:
\begin{enumerate}
\item[$\bullet$] \textsc{Case 1:} $\gamma_{0} \geq \alpha_{0}$ and $\delta_{0} \geq \beta_{0}$. In this case we choose
\[
b_{1} < 0 < b_{2}, \qquad d_{1} < 0 < d_{2}.
\]
\item[$\bullet$] \textsc{Case 2:} $\gamma_{0} \geq \alpha_{0}$ and $\beta_{0} \geq \delta_{0} + 1$. In this case we choose
\[
b_{1} < 0 < b_{2}, \qquad d_{2} < 0 < d_{1}.
\]
\item[$\bullet$] \textsc{Case 3:} $\alpha_{0} \geq \gamma_{0} + 1$ and $\delta_{0} \geq \beta_{0}$. In this case we choose
\[
b_{2} < 0 < b_{1}, \qquad d_{1} < 0 < d_{2}.
\]
\item[$\bullet$] \textsc{Case 4:} $\alpha_{0} \geq \gamma_{0} + 1$ and $\beta_{0} \geq \delta_{0} + 1$. In this case we choose
\[
b_{2} < 0 < b_{1}, \qquad d_{2} < 0 < d_{1}.
\]
\end{enumerate}
The remaining analysis is identical whichever case we are in. We will conduct it in the first case, leaving the other three cases for the interested reader.

First observe that the conditions on $b_{1}$ and $b_{2}$ yield, through Lemma \ref{l:intersect3}, that
\[
F \ast \mathfrak{r}_{0} = 1, \qquad F \ast \mathfrak{r}_{\infty} = 0, \qquad G \ast \mathfrak{r}_{0} = 0, \qquad G \ast \mathfrak{r}_{\infty} = 1.
\]
Similarly, the conditions on $d_{1}$ and $d_{2}$ yield, through Lemma \ref{l:intersect5}, that
\[
F \ast \mathfrak{s}_{0} = 1, \qquad F \ast \mathfrak{s}_{\infty} = 0, \qquad G \ast \mathfrak{s}_{0} = 0, \qquad G \ast \mathfrak{s}_{\infty} = 1.
\]
Taken together, these facts imply properties (4) and (5) of Proposition \ref{p:black}. Next, it is clear that $F \cdot G = 2d - 2$ for homological reasons. In view of our choice of parameters, Lemma \ref{l:intersect4} and Lemma \ref{l:intersect6} imply that
\[
2d - 2 = F \cdot G \geq (d - 1 - \alpha_{0} + \gamma_{0}) + (d - 1 - \beta_{0} + \delta_{0}),
\]
where the first parenthetical expression represents intersections in $\mathcal{U}(L_{\mathfrak{r}})$ that project into $D_{\mathfrak{r}}$, and the second represents intersections in $\mathcal{U}(L_{\mathfrak{s}})$ that project into $D_{\mathfrak{s}}$. Since we are assuming that $\gamma_{0} \geq \alpha_{0}$ and $\delta_{0} \geq \beta_{0}$, both of these expressions are at least $d - 1$. The above inequality then forces $\alpha_{0} = \gamma_{0}$ and $\beta_{0} = \delta_{0}$. It follows that $F \cap G$ consists of \emph{exactly} $2d - 2$ points, with $d - 1$ of these contained in each of $\mathcal{U}(L_{\mathfrak{r}})$ and $\mathcal{U}(L_{\mathfrak{s}})$ and projecting into $D_{\mathfrak{r}}$ and $D_{\mathfrak{s}}$, respectively. In particular, these intersections occur away from $T_{0} \cup T_{\infty}$, and a careful examination of the proof of Lemma \ref{l:intersect1} shows that they occur away from $E_{\mathfrak{r}} \cup E_{\mathfrak{s}}$ as well. This proves property (8) of Proposition \ref{p:black}, and property (9) is immediate by positivity of intersections, since $\smash{u_{\mathfrak{r}, \mathbf{X}}}$ projects into $D_{\mathfrak{r}}$ and $\smash{u_{\mathfrak{r}, \mathbf{X}}}$ projects into $D_{\mathfrak{s}}$.

Finally, since $F \cdot G = \mathbf{F}(\mathbf{x}_{1}) \cdot \mathbf{G}(\mathbf{x}_{2})$, it follows from the previous paragraph that there can be no intersections between the essential curves of $\mathbf{F}(\mathbf{x}_{1})$ and those of $\mathbf{G}(\mathbf{x}_{2})$. In particular, we have
\[
u_{\mathfrak{r}}(\mathbf{x}_{1}) \cdot v_{\mathfrak{r}}(\mathbf{x}_{2}) = 0, \qquad u_{\mathfrak{s}}(\mathbf{x}_{1}) \cdot v_{\mathfrak{s}}(\mathbf{x}_{2}) = 0.
\]
Combining these equalities with the final statements of Lemma \ref{l:intersect3} and Lemma \ref{l:intersect5}, we see that
\[
v_{\mathfrak{r}}(\mathbf{x}_{2}) \cdot \smash{u_{\mathfrak{r}, \mathbf{X}}} = 0, \qquad v_{\mathfrak{s}}(\mathbf{x}_{2}) \cdot \smash{u_{\mathfrak{s}, \mathbf{X}}} = 0.
\]
In view of our choice of parameters and the equality $\alpha_{0} = \gamma_{0}$, the second and third statements of Lemma \ref{l:intersect3} imply that
\[
F \cdot E_{\mathfrak{r}} + G \cdot E_{\mathfrak{r}} = (d - 1 - \alpha_{0}) + \gamma_{0} = d - 1.
\]
Likewise, our choice of parameters and the equality $\beta_{0} = \delta_{0}$, together with the second and third statements of Lemma \ref{l:intersect5} imply that
\[
F \cdot E_{\mathfrak{s}} + G \cdot E_{\mathfrak{s}} = (d - 1 - \beta_{0}) + \delta_{0} = d - 1.
\]
This verifies properties (6) and (7), and thus finally completes the proof of Proposition \ref{p:black}.

\appendix

\section{Three-Punctured Symplectic Spheres}

This appendix constructs a $3$-punctured \emph{symplectic} sphere in $T^{\ast}\TT^{2}$ that will be needed in the proof of Proposition \ref{p:simpli}. The key result is the following, and the proof takes up the rest of the appendix.

\begin{prp}\label{p:3sphere}
There exists a $3$-punctured symplectic sphere in $T^{\ast}\TT^{2}$, which is asymptotic to the Reeb orbits
\begin{gather*}
R_{0} = (-\theta, 0, -1, 0), \quad R_{1} = (0, -\theta, 0, -1), \\
R_{\infty} = (\theta, \theta, 1/\sqrt{2}, 1/\sqrt{2})
\end{gather*}
in a model for the positive cylindrical end which is specified below. This sphere can be modified to yield an embedded symplectic surface with cylindrical ends coinciding precisely with the trivial cylinders over the above orbits.
\end{prp}

The idea is to start with an embedded holomorphic curve $\CC \sd \{0, 1\} \to \CC^{\ast} \times \CC^{\ast}$ and compose it with the diffeomorphism $\CC^{\ast} \times \CC^{\ast} \approx T^{\ast}\TT^{2}$ given in polar coordinates by
\[
(r_{1}, t_{1}, r_{2}, t_{2}) \mapsto (t_{1}, t_{2}, -\log(r_{1}), -\log(r_{2})).
\]
The natural candidate for such a curve is
\[
z \mapsto \left(\frac{1}{z}, \frac{1}{z - 1}\right).
\]
Here we are thinking of $\CC \sd \{0, 1\}$ as $\SS^{2} \sd \{0, 1, \infty\}$. The composition of this curve with the above diffeomorphism yields a smooth embedding $w \colon \CC \sd \{0, 1\} \to T^{\ast}\TT^{2}$. In polar coordinates $z = re^{2\pi i \theta}$ centered at $0$, the representation of $w$ is given by
\[
w(r, \theta) = \left(-\theta, -\arctan\left(\frac{r\sin\theta}{r\cos\theta - 1}\right), \log(r), \frac{1}{2}\log(r^{2} - 2r\cos\theta + 1)\right).
\]
On the other hand, in polar coordinates $z = 1 + re^{2\pi i \theta}$ centered at $1$, the representation is given by
\[
w(r, \theta) = \left(-\arctan\left(\frac{r\sin\theta}{r\cos\theta + 1}\right), -\theta, \frac{1}{2}\log(r^{2} + 2r\cos\theta + 1), \log(r)\right).
\]
Finally, in coordinates $z = (1/r)e^{-2\pi i \theta}$ centered at $\infty$, the representation is given by
\[
w(r, \theta) = \left(\theta, \arctan\left(\frac{\sin\theta}{\cos\theta - r}\right), -\log(r), -\frac{1}{2}\log\left(\frac{r^{2}}{r^{2} - 2r\cos\theta + 1}\right)\right).
\]
For technical reasons, we will modify $w$ by replacing each instance of $r$ in the above three representations by a smooth, positive, and increasing function $f(r)$ such that $f(r) = r^{2}$ for $0 < r < \epsilon$ and $f(r) = r$ for $r > 1 - \epsilon$, where $\epsilon > 0$ is small enough. The resulting map is still a smooth embedding, which we continue to denote by $w$. This is not $J_{\mathrm{std}}$-holomorphic, but we do have the following fact.

\begin{lem}\label{l:three}
The image of $w$ is a symplectic submanifold of $T^{\ast}\TT^{2}$.
\end{lem}

\begin{proof}
It will be convenient throughout the proof to set
\[
A^{\pm}(r, \theta) = f(r)^{2} \pm 2f(r)\cos\theta + 1.
\]
It is easily checked that $A^{\pm}(r, \theta) > 0$ on the whole domain of $w$. Working in polar coordinate $z = re^{2\pi i \theta}$ centered at $0$, the map $w$ is given by
\[
w(r, \theta) = \left(-\theta, -\arctan\left(\frac{f(r)\sin\theta}{f(r)\cos\theta - 1}\right), \log(f(r)), \frac{1}{2}\log(A^{-}(r, \theta))\right).
\]
In fact, this representation of $w$ extends to its entire domain minus two small open disks centered at each of $1$ and $\infty$, where we must use different polar representations of $w$. But this means we can check that $w$ is symplectic by working in each of its three polar coordinate representations. Setting $e_{1} = d{w(1, 0)}$ and $e_{2} = d{w(0, 1)}$, a straightforward computation shows that
\[
d\lambda(e_{1}, e_{2}) = f^{\prime}(r)\left(\frac{1}{f(r)} + \frac{f(r)}{A^{-}(r, \theta)}\right).
\]
It is easy to check that this expression is positive for any $(r, \theta)$ in the domain of $w$. Similarly, one can compute that
\[
d\lambda(e_{1}, e_{2}) = f^{\prime}(r)\left(\frac{1}{f(r)} + \frac{f(r)}{A^{+}(r, \theta)}\right)
\]
and that
\[
d\lambda(e_{1}, e_{2}) = f^{\prime}(r)\left(\frac{1}{f(r)} + \frac{1}{f(r)A^{-}(r, \theta)}\right)
\]
in the representations near $1$ and $\infty$, respectively. Since these expressions are always positive by our choice of $f(r)$, it follows that the image of $w$ is indeed symplectic as claimed.
\end{proof}

We next want to examine the asymptotic behavior of $w$ near its three punctures. Since $w$ itself is not $J_{\mathrm{std}}$-holomorphic, we cannot deduce any conclusions from the usual theory of punctured curves. Luckily, the situation is still salvageable. Consider the Reeb orbits
\[
R_{0} = (-\theta, 0, -1, 0), \quad R_{1} = (0, -\theta, 0, -1), \quad R_{\infty} = (\theta, \theta, 1/\sqrt{2}, 1/\sqrt{2}),
\]
in $S^{\ast}\TT^{2}$. From the proof of Lemma \ref{l:three}, we know that $w$ can be written in polar coordinates $z = re^{2\pi i \theta}$ centered at $0$ and for $r$ small enough as
\[
w(r, \theta) = \left(-\theta, -\arctan\left(\frac{r^{2}\sin\theta}{r^{2}\cos\theta - 1}\right), \log(r^{2}), \frac{1}{2}\log(r^{4} - 2r^{2}\cos\theta + 1)\right).
\]
We compose this with the symplectic identification $T^{\ast}\TT^{2} \sd \TT^{2} \approx \RR \times S^{\ast}\TT^{2}$ given by
\[
(q_{1}, q_{2}, p_{1}, p_{2}) \mapsto \left(\log(\Vert (p_{1}, p_{2}) \Vert), q_{1}, q_{2}, \frac{p_{1}}{\Vert (p_{1}, p_{2}) \Vert}, \frac{p_{2}}{\Vert (p_{1}, p_{2}) \Vert}\right),
\]
and thereby obtain a cylindrical representation of $w$ near the puncture at $0$. Some standard, through rather involved, analysis shows that this converges to the trivial Reeb cylinder over $R_{0}$ in the $C^{1}$-topology as $r \to 0^{+}$. (The convergence statements are with respect to the standard metric on all Euclidean spaces and the flat metric on $\TT^{2}$.) Similar computations using polar coordinates centered near the punctures at $1$ and $\infty$ allow us to draw analogous conclusions, with the Reeb cylinder over $R_{0}$ replaced with the ones over $R_{1}$ and $R_{\infty}$, respectively.

The upshot of this analysis is the following gluing procedure that can be performed near each puncture of $w$. For concreteness, consider the puncture at $0$. We know that the cylindrical representation of $w$ converges to the trivial Reeb cylinder over $R_{0}$ in the $C^{1}$-topology as $r \to 0^{+}$. This implies that, far enough along the positive cylindrical end of $T^{\ast}\TT^{2}$, we can multiply $w$ by a smooth cutoff function in order to glue it to said trivial cylinder. Performing this far enough along the cylindrical end, we can choose this cutoff function to have arbitrarily small derivative. As a consequence, the resulting object can be arranged to be symplectic. The same procedure can be performed near the punctures at $1$ and $\infty$. This proves Proposition \ref{p:3sphere}.

\bibliographystyle{plain}
\bibliography{Packing_Integral_Tori_in_Del_Pezzo_Surfaces.bib}

\end{document}